\documentclass[11pt, notitlepage]{article}
\usepackage{amssymb,amsmath,comment}
\usepackage{epsfig,enumerate,amsmath,amsfonts,amssymb}
\usepackage{amsthm}
\usepackage{a4wide}
\usepackage{parskip}
\usepackage{enumitem}
\usepackage{xcolor}
\usepackage{etoolbox}   
\usepackage{hyperref}
\usepackage{refcheck}
\usepackage{latexsym}
\usepackage{cite}

\allowdisplaybreaks
\newcommand{\Om} {\Omega}

\newcommand{\be} {\begin{equation}}
\newcommand{\ee} {\end{equation}}
\newcommand{\bea} {\begin{eqnarray}}
\newcommand{\eea} {\end{eqnarray}}
\newcommand{\Bea} {\begin{eqnarray*}}
\newcommand{\Eea} {\end{eqnarray*}}

\newcommand{\de} {\delta}
\newcommand{\ga} {\gamma}

\newcommand{\De} {\Delta}
\newcommand{\la} {\lambda}

\newcommand{\M} {\mathcal{W}_n}

\def\R{{\mathbb R}}

\def\H{{\mathbb {H}^N}}

\def\W{{\overset{\scriptscriptstyle\circ}{S}{}^Q_1(\Omega)}}
\def\mou{{\left|u\right|^{2\alpha}}}
\def\nah{\nabla_{\mathbb{H}}}
\def\I{I_{\lambda}}
\def\N{{N_\lambda}}

\DeclareMathOperator*{\essinf}{ess\,inf}
\DeclareMathOperator*{\esssup}{ess\,sup}
\newcommand{\fint}{\mathchoice{\ooalign{$\displaystyle\int$\cr\hidewidth$\displaystyle-$\hidewidth\cr}}{\ooalign{$\int$\cr\hidewidth$-$\hidewidth\cr}}{\ooalign{$\int$\cr\hidewidth$-$\hidewidth\cr}}{\ooalign{$\int$\cr\hidewidth$-$\hidewidth\cr}}}
\newcommand{\ra} {\rightarrow}

\numberwithin{equation}{section}

\newtheorem{definition}{Definition}[section]
\newtheorem{theorem}{Theorem}[section]
\newtheorem{rem}{Remark}[section]
\newtheorem{lemma}{Lemma}[section]
\makeatletter
\patchcmd{\proof}{\itshape}{\normalfont}{}{}
\patchcmd{\proof}{\@addpunct{.}}{\@addpunct{.}\normalfont}{}{}
\makeatother
\newtheorem{prop}{Proposition}[section]

\begin{document}

\setlength{\abovedisplayskip}{3pt}
\setlength{\belowdisplayskip}{3pt}
\date{}
	\title{    Positive and nodal solutions for a parametric quasilinear $Q$-sub-Laplacian problem with critical exponential growth on the Heisenberg group
}
\author{ {\bf Ankit Mishra$\,^{1,}$\footnote{e-mail: {\tt ankitmishra.rs.mat23@itbhu.ac.in, ankitmishra18295@gmail.com}}}, \bf Sarika Goyal$\,^{2,}$\footnote{e-mail: {\tt sarika@nsut.ac.in, sarika1.iitd@gmail.com}}, \bf Divya Goel $\,^{1,}$ \footnote{e-mail: {\tt divya.mat@iitbhu.ac.in},\\
     \tt Corresponding Author: Divya Goel }
     \\ $^1\,$ Department of Mathematical Sciences,\\ Indian Institute of Technology (BHU), India\\
   $^2\,$ Department of Mathematics,\\ Netaji Subhash University of Technology, India }

\maketitle
\begin{abstract}
In this article, we investigate the following modified quasilinear equation  with parameter driven by the $Q$-subLaplacian:
\begin{align*}
\begin{cases}
    -\Delta_Q u - \Delta_Q\bigl(|u|^{2\alpha}\bigr)\,|u|^{2\alpha-2} u = \lambda f(\xi,u)  & \text{in } \Omega, \\[2mm]
    u = 0 & \text{on } \partial\Omega,
\end{cases}
\end{align*}
where $\Delta_Q(\cdot):= \mathrm{div}_{\mathbb{H}}\bigl(|\nabla_{\mathbb{H}}(\cdot)|^{Q-2}\nabla_{\mathbb{H}}(\cdot)\bigr)$ denotes the $Q$-subLaplacian on the Heisenberg group $\mathbb{H}^N$, $Q=2N+2$ is the homogeneous dimension, $\Omega \subset \mathbb{H}^N$ is a smooth bounded domain, $\lambda>0$, $\alpha > \frac{1}{2}$, and $f$ has critical or subcritical exponential growth of order $\exp\bigl(\beta|t|^{2\alpha Q/(Q-1)}\bigr)$. We prove three results: the existence of a nontrivial positive weak solution in the critical case for all large $\lambda$, and the existence of a least-energy nodal solution with exactly two nodal domains, under subcritical and critical exponential growth. A change of variables $u=g(v)$ reduces the problem to a quasilinear problem whose energy functional is of class $C^1$; the exponent $2\alpha Q/(Q-1)$ arises from the growth of $g$. We handled the exponential growth using the sharp Moser-Trudinger inequality of Cohn and Lu. The positive solution is obtained by the mountain pass theorem and  the nodal solutions by minimization on a nodal Nehari set.
     \medskip
     
\end{abstract}

\section{Introduction}
In this paper, we study the  following modified quasilinear  $Q$-subLaplacian problem with parameter
\begin{equation} 
\begin{cases}
    -\Delta_Q u - \Delta_Q (\left|u\right|^{2\alpha})\left|u\right|^{2\alpha-2} u= \lambda f(\xi,u)  & \text{in }   \Omega, \\ 
           u=0  & \text{on } \partial\Omega, 
\end{cases}
\tag{$\mathcal{P}_{\lambda}$}
\label{eq:P1}
\end{equation}
where $\Omega$ is a smooth bounded domain in the Heisenberg group $\H$, $Q=2N+2$ is the
homogeneous dimension of $\H$, and $\Delta_Q u=\nah\cdot(\left|\nah u\right|^{Q-2}\nah u)$ is
the $Q$-sub-Laplacian, $\nah$ denotes the horizontal gradient. Here $\lambda>0$,  
$\alpha>\frac12$, and $f$ has critical exponential growth of order
$\exp(\beta\left|t\right|^{\frac{2\alpha Q}{Q-1}})$. 
Problem \eqref{eq:P1} is the Heisenberg counterpart of a family of equations arising in mathematical physics. Consider the quasilinear Schr\"odinger equation
\begin{equation}{\label{eq:f1}}
   i \partial_tz =-\Delta z + W(x)z- \tilde{h}(\left|z\right|^2)z- \Delta h(\left|z\right|^2)h^{\prime}(\left|z\right|^2)z,
\end{equation}
where $z: \mathbb{R} \times \mathbb{R}^N \to \mathbb{C}$, $W:\mathbb{R}^N \to \mathbb{R}$ is a potential, and the nonlinearities
$h,\tilde{h}:[0,\infty)\to\mathbb{R}$ are given. Solitary waves are solutions of the form $z(t,x)=e^{-iEt}u(x)$ with $E\in\mathbb R$ and $u$ real. Substituting this ansatz into \eqref{eq:f1} gives the elliptic equation
 \begin{equation}{\label{eq:f2}}
    -\Delta u + V(x)u- \De h(u^2)h^{\prime}(u^2)u= k(u), \qquad x\in \mathbb{R}^N,
\end{equation}  
where $V(x)= W(x)-E$ and $k(u)=\tilde{h}(\left|u\right|^2)u$. The case $h(s)=s$ is the superfluid film equation of Kurihara \cite{kurihara1981large, laedke1983evolution}, which models the time evolution of the condensate wave function in a superfluid film. Equation \eqref{eq:f1} also describes the propagation of a high-irradiance laser in a plasma and the self-channeling of a high-power ultra-short laser in matter \cite{brandi1993relativistic,chen1993necessary, de1997global, ritchie1994relativistic}. For $h(s)=s$, \eqref{eq:f2} reduces to
\begin{equation}{\label{eq:f3}}
-\Delta u + V(x)u - u \Delta(u^2) = k(u), \qquad x \in \mathbb{R}^N .
\end{equation}
We summarize some known results on \eqref{eq:f3} and its variants. For the general
quasilinear term $\Delta(\left|u\right|^{2\alpha})\left|u\right|^{2\alpha-2}u$ with a parameter
$\alpha>\frac34$ and a subcritical superlinear nonlinearity, the existence of a positive
solution, a negative solution and a sequence of high-energy solutions was obtained in
\cite{wu2014multiple} by a perturbation method. In the Sobolev-critical case, bound states were
obtained in \cite{wang2012bound}. Fang and Szulkin \cite{fang2013multiple} considered the case $\alpha=1$ with a nonlinearity
behaving like $t$ at the origin and like $t^{3}$ at infinity, and obtained a nontrivial solution
by combining the change of variables with minimax methods. The authors proved infinitely many
solutions when the nonlinearity is odd.
We do not attempt to review the remaining literature; existence, multiplicity and concentration
have all been studied under various conditions on $V$ and $k$, and we refer to
\cite{liu2003soliton, deng2016existence, zhang2015infinitely} and the references therein. The
operator $-\Delta_p u-\Delta_p(u^2)u$, $1<p<\infty$, arises in plasma physics and fluid
mechanics \cite{bass} and in dissipative quantum mechanics \cite{hasse1980general}.

The difficulty in \eqref{eq:P1} is structural. The term $\Delta_Q(\left|u\right|^{2\alpha})\left|u\right|^{2\alpha-2}u$ prevents the natural energy functional
\begin{align*}
J_\lambda(u)=\frac1Q\int_\Omega\left(1+(2\alpha)^{Q-1}\left|u\right|^{Q(2\alpha-1)}\right)\left|\nah u\right|^{Q}d\xi-\lambda\int_\Omega F(\xi,u)\,d\xi
\end{align*}
from being finite on all of $\W$, so critical point theory does not apply to $J_\lambda$ directly. Three approaches to this difficulty appear in the literature: the perturbation method,
constrained minimization, and a nonlinear change of variables. We adopt the third. The first existence results for \eqref{eq:f3} obtained by variational methods are due to Poppenberg, Schmitt and Wang \cite{poppenberg2002existence}, who treated the case $N=1$, and the case of radially symmetric $V$ in higher dimensions, by constrained minimization; the same method was subsequently used by Liu, Wang and Wang \cite{liu2004solutions}. Moameni \cite{moameni2006existence} reduced \eqref{eq:f3} to a semilinear equation by a change of
variables, working in an Orlicz space, and do {\'O}, Miyagaki and Soares \cite{miyagaki2007soliton} carried out the same reduction in the usual Sobolev framework. We refer to \cite{colin2004solutions, liu2003soliton} for the properties of this change of variables.

Nodal solutions for the generalized quasilinear equation
$-\operatorname{div}(g^{2}(u)\nabla u)+g(u)g^{\prime}(u)\left|\nabla u\right|^{2}+V(x)u=k(u)$,
in which the choice $g^{2}(u)=1+2u^{2}$ recovers \eqref{eq:f3} while other choices of $g$
correspond to diffusion terms of the form $\Delta(\left|u\right|^{2\alpha})
\left|u\right|^{2\alpha-2}u$, were obtained in \cite{deng2014nodal}; see also
\cite{deng2016existence} for the concentration behaviour of such solutions and
\cite{zhang2015infinitely} for infinitely many sign-changing solutions in the nonautonomous
case.  We summarize the results on nodal solutions that are closest to ours. In \cite{liu2004solutions}, Liu et al.\ considered
\begin{equation*}
-\Delta u + V(x)u-u \Delta(u^2)= \left|u\right|^{p-2}u
\end{equation*}
 with $4 \leq p \leq 22^{*}$, $2^{*}=\frac{2N}{N-2}$, and $V \in C(\mathbb{R}^N)$ strictly positive, bounded, with a finite limit at infinity; they obtained a least energy sign-changing solution by approximating on a Nehari manifold in a suitable subset of $H^1(\mathbb{R}^N).$ In \cite{deng2011infinitely}, Deng et al.\ treated \eqref{eq:f3} with $p\in (4, 22^{*})$ and showed that for every integer $k\geq 0$ there is a pair of sign-changing solutions with $k$ nodes, by a minimizing argument and an energy comparison. The critical case was settled in \cite{deng2013nodal}. Recently, Yang et al.\ \cite{yang2019least} dealt with 
 \begin{align*}
 -\Delta u + V(x)u-u \Delta(u^2)=a(x)\left(g(u)+\left|u\right|^{p-2}u\right), \qquad x\in \mathbb{R}^N,
 \end{align*} 
under conditions on $a$ and $V$, and obtained a sign-changing solution by the Nehari manifold, deformation arguments and $L^{\infty}$-estimates. 

 On the Heisenberg group, $Q$ is the exponent at which the Folland-Stein embedding fails, and its role is taken over by the sharp Moser-Trudinger inequality of Cohn and Lu
\cite{cohn2001best}. Problems for the $Q$-sub-Laplacian with exponential nonlinearities have been studied in \cite{lamlutang2012}, and a concentration-compactness principle of Moser-Trudinger type on $\H$ was established in \cite{lilzhu2018}. For the quasilinear operator in \eqref{eq:P1}, however, little work has been done in this direction: the change of variables must be carried out for the $Q$-sub-Laplacian rather than for the Laplacian, and the resulting critical growth is of order $\exp(\beta\left|t\right|^{\frac{2\alpha Q}{Q-1}})$ rather than $\exp(\beta\left|t\right|^{\frac{Q}{Q-1}})$. To our knowledge, sign-changing solutions have not been obtained for \eqref{eq:P1}.

In the present paper we establish three results for \eqref{eq:P1}: existence of
a positive solution in the critical case, and existence of a least energy nodal solution with
exactly two nodal domains, first in the subcritical and then in the critical case. In the two
critical results the parameter $\lambda$ is required to be large, for different reasons: in
Theorem~\ref{T1} so that the mountain pass level lies below the threshold at which compactness
is lost, and in Theorem~\ref{T3} so that the least nodal energy does.

We outline the arguments. The change of variables $u=g(v)$, with $g$ defined by
\begin{align*}
g^{\prime}(t)=\left(1+(2\alpha)^{Q-1}\left|g(t)\right|^{Q(2\alpha-1)}\right)^{-1/Q},
\qquad g(0)=0,
\end{align*}
turns \eqref{eq:P1} into the quasilinear problem $-\Delta_Qv=\lambda f(\xi,g(v))g^{\prime}(v)$,
whose functional $I_\lambda$ is of class $C^1$ on $\W$.   The change of variables has two consequences. The first  is
a property of $g$ which helps us to tackle the nonlinearity using the Moser-Trudinger inequality of Cohn and Lu \cite{cohn2001best}. The second concerns the functional: $I_\lambda$ splits over positive and
negative parts,
\begin{align*}
I_\lambda(v)=I_\lambda(v^{+})+I_\lambda(v^{-}),\qquad
\langle I_\lambda^{\prime}(v),v^{\pm}\rangle=\langle I_\lambda^{\prime}(v^{\pm}),v^{\pm}\rangle,
\end{align*}
So the nodal Nehari set is cut out by two independent scalar equations. 

The main difficulty is to recover the  compactness which fails due to the critical nonlinearity. A Palais-Smale sequence for $I_\lambda$ may fail to converge, and the usual concentration-compactness argument does not apply, since the
exponential nonlinearity lies outside every Lebesgue space scale. To overcome this, we proceed in two steps. First we  prove that the horizontal gradients of a bounded Palais-Smale sequence converge almost everywhere. The proof combines a Moser-Trudinger bound on small balls with the finiteness of
the concentration set of $\{\left|\nah v_n\right|^{Q}\}$, and the cut-off functions are built from the group dilations $\delta_\theta$ and left translations rather than from their Euclidean analogues. We then show that the mountain pass level lies below $\frac{1}{2\alpha Q}(\alpha_Q/\beta_0)^{Q-1}$, the level at which compactness fails. This is done by evaluating $I_\lambda$ on the Moser sequence supported in the largest gauge ball of $\Omega$, and it is here that $\lambda$ must exceed $\lambda_1$. The nodal case has a threshold of its own, and there the condition on $\lambda$ enters through Lemma~\ref{lem:L12}, which sends $m_\lambda$ to zero as $\lambda\to\infty$. Once a minimizer on the nodal Nehari set is found, a quantitative deformation lemma and a degree argument show that it is a critical point of $I_\lambda$, and the regularity of Proposition~\ref{prop:reg} makes the count of nodal
domains meaningful.

We now state the hypotheses on $f$. Throughout, $f:\Omega\times\R\to\R$ is of class $C^1$ in its second variable for every $\xi\in\bar\Omega$, and $F(\xi,t)=\int_0^tf(\xi,s)\,ds$.
\begin{itemize}
   \item[$(f_1)$]  $f(\xi,t)>0$ for $t>0$ and $f(\xi,t)=0$ for $t\leq 0;$
  \item[$(f_1^{\prime})$] $f(\xi,t)>0$ for $t>0$, $f(\xi,t)<0$ for $t<0$ and $f(\xi,0)=0;$ 
  \item[$(f_2)$] $t\mapsto\dfrac{f(\xi,t)}{\left|t\right|^{2\alpha Q-2}t}$ is increasing on $(-\infty,0)$ and on $(0,\infty)$, for almost every $\xi\in\Omega;$
  \item[$(f_3)$] $\displaystyle \lim_{t \to 0}\frac{\left|f(\xi,t)\right|}{{\left|t\right|}^{2 \alpha Q-1}}=0$ uniformly in $\xi\in\Omega;$
  \item[$(f_4)$] there exists $\mu> 2 \alpha Q$ such that $0<\mu F(\xi,t)\leq tf(\xi,t)$ for all $(\xi,t)\in \Omega\times\R \setminus \{0\};$
  \item[$(f_5)$] there exist $R_0, M_0>0$ such that $0<F(\xi,t)\leq M_0f(\xi,t)$ for all $t\geq R_0$ and $\xi\in\Omega;$

  \item[$(f_6)$] $\displaystyle\alpha_0:=\liminf_{t\to\infty}
\frac{t\,f(\xi,t)}{\exp\left(\beta_0\left|t\right|^{\frac{2\alpha Q}{Q-1}}\right)}>0$ uniformly in $\xi\in\Omega$. In this case, we set
\begin{align}\label{eq:lambda1}
\lambda_1:=\frac{Q}{(Q-1)\,\omega_{Q-1}\,r^{Q}\,\alpha_0}
\left(\frac{\alpha_Q}{\beta_0}\right)^{Q-1},
\end{align}
where $\omega_{Q-1}$ is the constant of \eqref{eq:constants} and $r$ is the radius of the largest gauge ball contained in $\Omega$; by left invariance of $\|\cdot\|$, we may and do
assume this ball is centred at the origin;
       \item[$(\mathcal{A}_1)$] $f$ has subcritical growth at $+\infty$: 
       \begin{align*}
       \lim_{{\left|t\right|}\to+\infty}\frac{\left|f(\xi,t)\right|}{\exp{(\beta {\left|t\right|}^\frac{2\alpha Q}{Q-1})}}=0
       \end{align*} 
       uniformly for almost every $\xi\in\Omega$ and every $\beta>0;$
       \item[$(\mathcal{A}_2)$] $f$ has critical growth at $+\infty$: there exists $\beta_0>0$ such that 
       \begin{align*}
       \lim_{{\left|t\right|}\to+\infty}\frac{\left|f(\xi,t)\right|}{\exp{(\beta {\left|t\right|}^\frac{2\alpha Q}{Q-1})}}= \begin{cases}
           0, &\text{uniformly in } \xi\in\Omega,\; \forall\; \beta>\beta_0,\\ +\infty, &\text{uniformly in } \xi\in\Omega, \;\forall\;\beta<\beta_0.
          \end{cases}
       \end{align*} 
   \end{itemize}

The semilinear counterpart of \eqref{eq:P1}, namely
$-\Delta_Qu=\lambda f(\xi,u)$ on a bounded domain of $\H$ with $f$ of critical exponential growth, was treated by Sun and Song \cite{sunsong2022}, who obtained least energy nodal
solutions with exactly two nodal domains by a constrained variational method, together with the quantitative deformation lemma. The presence of the term
$\Delta_Q(\left|u\right|^{2\alpha})\left|u\right|^{2\alpha-2}u$ changes the problem in three ways: the natural functional $J_\lambda$ is not finite on $\W$, so no constraint set can be
defined for it at all; after the change of variables the critical growth is of order
$\exp(\beta\left|t\right|^{2\alpha Q/(Q-1)})$ rather than
$\exp(\beta\left|t\right|^{Q/(Q-1)})$; and the compactness analysis requires the almost everywhere convergence of horizontal gradients established in Section~3, which is not needed in the semilinear case. No positive-solution result of the type of Theorem~\ref{T1} is obtained in \cite{sunsong2022}. With this, we state our main results. 
\begin{theorem}{\label{T1}}
Suppose that $(f_1)$, $(f_2)-(f_6)$ hold and that $f$ has critical growth
$(\mathcal{A}_2)$. Then problem \eqref{eq:P1} has a nontrivial positive weak solution for
every $\lambda>\lambda_1$.
\end{theorem}
\begin{theorem}{\label{T2}}
Suppose that $(f_1^{\prime}), (f_2)-(f_5)$ hold and that $f$ has subcritical growth $(\mathcal{A}_1)$ at $+\infty$. Then, for every $\lambda>0$, problem \eqref{eq:P1} possesses a least energy nodal solution $z\in \N$ satisfying $\I(z)=\displaystyle \inf_{\N} \I$. Moreover, $z$ has exactly two nodal domains.
\end{theorem}
\begin{theorem}{\label{T3}}
Suppose that $(f_1^{\prime}), (f_2)-(f_5)$ hold and that $f$ has critical growth $(\mathcal{A}_2)$ at $+\infty$. Then there exists $\lambda_0>0$ such that for every $\lambda \geq\lambda_0$ problem \eqref{eq:P1} has a least energy nodal solution $z\in \N$ satisfying $\I(z)= \displaystyle \inf_{\N} \I$ with exactly two nodal domains.
\end{theorem}

\begin{rem}
The hypotheses above are not vacuous. For $q>2\alpha Q$ the function
\begin{align*}
f(\xi,t)=\left|t\right|^{q-2}t\,\exp\left(\beta_0\left|t\right|^{\frac{2\alpha Q}{Q-1}}\right)
\end{align*}
satisfies $(f_1^{\prime})$, $(f_2)$, $(f_3)$, $(f_5)$ and $(\mathcal{A}_2)$, together with $(f_4)$ for any $\mu\in(2\alpha Q,q)$, and $(f_6)$ with $\alpha_0=+\infty$. Note that $q=2\alpha Q$ would fail $(f_4)$, since $tf(\xi,t)/F(\xi,t)\to2\alpha Q$ as $t\to0$.
\end{rem}

\begin{rem}
The constants $\sigma_Q$ and $\omega_{Q-1}$ of \eqref{eq:constants} coincide in the Euclidean setting, where $\left|\nabla\rho\right|\equiv1$. On $\H$ they do not: the gauge gradient $\left|\nah\rho\right|=\left|(x,y)\right|/\rho$ degenerates on the centre of the group, so that $\sigma_Q<\omega_{Q-1}$ strictly. 
Both appear in $(f_6)$: the first through the normalization $\alpha_Q=Q\sigma_Q^{1/(Q-1)}$ of the extremal sequence used in Lemma~\ref{lem:L9}, the second through the volume of gauge balls. The threshold in $(f_6)$ is therefore genuinely
sub-Riemannian, and not a transcription of its Euclidean counterpart. 
\end{rem}

Some part of Section~2, in particular the change of variables and the properties of $g$, is shared with the companion paper \cite{mishra2026critical}, where the corresponding problem on all of $\H$ with polynomial critical growth is treated. The analysis here is different: the loss of compactness is governed by the Moser-Trudinger inequality rather than by the Folland-Stein embedding.

This paper is organized as follows. Section $2$ collects the necessary facts on the Heisenberg group, sets up the variational framework, and establishes the regularity of weak solutions of the transformed problem. Section $3$ is devoted to convergence results for Palais-Smale sequences. In Section $4$, we prove the existence of a positive solution by mountain pass arguments. In Section $5$, we establish the existence of a nodal solution, in the subcritical and in the critical case.

\section{Preliminaries, Variational Framework}
For the convenience of the reader, we will give a quick summary of the Heisenberg group and the quasilinear subelliptic operator. The Heisenberg group $\H = {\R}^N\times{\R}^N\times{\R}, N\in \mathbb{N}$ is a Lie group, endowed with the following non-abelian group law 
\begin{align*}
\xi \circ \xi'=(x,y,t) \circ(x', y',t')=\left(x+x',\,y+y',\,t+t'+ 2\big(\langle x',y \rangle - \langle x,y' \rangle\big)\right),
\end{align*}
where $\xi=(x,y,t), \xi'= (x', y',t') \in \H.$ 

With this law, the inverse of $\xi \in \H$ is given by $\xi^{-1}=-\xi$ and $(\xi \circ \xi')^{-1}=(\xi')^{-1} \circ \xi^{-1}.$ The corresponding Lie algebra of left invariant vector fields is generated by vector fields given as follows: 
\begin{align*}
X_i= \frac{\partial}{\partial x_i}+2y_i \frac{\partial}{\partial t}, \; Y_i= \frac{\partial}{\partial y_i}- 2x_i \frac{\partial}{\partial t}, \; T= \frac{\partial}{\partial t}.
\end{align*} 
One can easily verify that for all $i,j=1,2,\dots,N,$ 
\begin{align*}
[X_i,X_j]=[Y_i,Y_j]=[X_i,T]=[Y_i,T]=0
\end{align*} 
and 
\begin{align*}
[X_i,Y_j]= -4\delta_{ij} \frac{\partial}{\partial t}.
\end{align*} 
With the help of these relations, the Heisenberg's canonical commutation relations of quantum mechanics for position and momentum are established, hence it is named `Heisenberg group' \cite{heisenberg2013physical}. 
Define the left translation $\tau_\xi: \H \to \H$ by 
\begin{align*}
\tau_\xi(\xi')=\xi \circ \xi',
\end{align*} 
and a natural $\mathbb{H}$- dilation $\delta_\theta: \H \to \H$ for $\theta >0$ is defined as  
\begin{align*}
\delta_\theta(x,y,t)=(\theta x, \theta y, \theta^2t).
\end{align*}  

 The Jacobian determinant of $\delta_\theta$ is $\theta^Q$. The number $Q=2N+2$ is called the homogeneous dimension of $\H$, and it plays a role equivalent to topological dimension in Euclidean space. The homogeneous norm on $\H$ is defined by 
 \begin{align*}
 \rho(\xi)=|\xi|=|(x,y,t)|=\left(t^2+\left(|x|^2+|y|^2\right)^2\right)^{\frac14} \quad \text{ for all } \xi=(x,y,t) \in \H,
 \end{align*}
 and $B_R(\xi_0):=\{\xi\in\H:\rho(\xi_0^{-1}\circ\xi)<R\}$ denotes the gauge ball of radius $R$ centred at $\xi_0$.
 
 The subelliptic Laplacian or Kohn Laplacian $\Delta_{\mathbb{H}}$ on $\H$ is a second order self-adjoint operator defined as follows: 
 \begin{align*}
 \Delta_{\mathbb{H}}= \sum_{j=1}^{N} \left(X_{j}^2 +Y_{j}^2 \right),
 \end{align*}
and $\nabla_{\mathbb H}u=(X_1u,\dots,X_Nu,Y_1u,\dots,Y_Nu)$ denotes the horizontal gradient, so that $\nabla_{\mathbb H}u$ takes values in $\R^{2N}$. 
 For more details on the Heisenberg functional framework, we refer to \cite{garofalo1990frequency,loiudice2005improved} and the references therein. 
 
We shall also need the polar coordinate formula on $\H$ (see \cite{stein1974estimates}): there is a Radon measure $d\varsigma$ on the gauge sphere $\Sigma=\{\xi\in\H:\rho(\xi)=1\}$ such that
 \begin{align}\label{eq:polar}
 \int_{\H}u(\xi)\,d\xi=\int_0^\infty\int_{\Sigma}u(\delta_s\zeta)\,s^{Q-1}\,d\varsigma(\zeta)\,ds .
 \end{align}
 Accordingly we set
 \begin{align}\label{eq:constants}
 \omega_{Q-1}:=\varsigma(\Sigma),\qquad 
 \sigma_Q:=\int_{\Sigma}\left|\nabla_{\mathbb H}\rho\right|^Q d\varsigma ,
 \end{align}
 so that $\left|B_R(0)\right|=\frac{\omega_{Q-1}}{Q}R^Q$. Since $\left|\nabla_{\mathbb H}\rho(\xi)\right|=\left|(x,y)\right|/\rho(\xi)\le 1$, with strict inequality off the set $\{t=0\}$, one has
 \begin{align*}
 \sigma_Q<\omega_{Q-1}.
 \end{align*}
 
 For $1 \leq p<\infty$, the Lebesgue space $L^p(\H)$ is defined as 
\begin{align*}
L^p(\H)=\left\{u:\H \to \mathbb{R} \text{ measurable }: \int_{\H} \left|u\right|^pd\xi <\infty \right\}
\end{align*} 
equipped with the usual norm $\|u\|_{L^p(\H)}=\left(\int_{\H}|u|^p\,d\xi\right)^{1/p}.$ Analogous to space $W^{1,N}(\mathbb{R}^N)$, Folland and Stein \cite{stein1974estimates} introduced the space 
\begin{align*}
S^Q_1(\H)= \{u\in L^Q(\H): X_j u, Y_j u \in L^Q(\H), \text{ for all } j=1,2,3,\dots,N\}
\end{align*} 
with the corresponding norm  
\begin{align*}
\|u\|_{S^Q_1(\H)}= \left(\int_\H \left|\nah u\right|^Q d\xi + \int_\H \left|u\right|^Q d\xi\right)^\frac{1}{Q}.
\end{align*} 
The space $\overset{\scriptscriptstyle\circ}{S}{}^Q_1(\Omega)$ is the closure of $C_0^{\infty}(\Omega)$ in $S^Q_1(\H)$, equipped with the norm  
\begin{align*}
\|u\|= \|\nabla_{\mathbb{H}} u\|_Q:= \left( \int_\Omega \left|\nabla_{\mathbb{H}} u\right|^Q d\xi \right)^\frac{1}{Q}.
\end{align*}
Since $\Omega$ is a bounded domain, the embedding $\W\hookrightarrow L^r(\Omega)$ is continuous and compact for every $r\in[1,\infty)$, see \cite{stein1974estimates}. Throughout we write $Q^{\prime}=\frac{Q}{Q-1}$.

To deal with exponential nonlinearties, the Moser-Trudinger inequality plays an important role. Let us discuss some Moser-Trudinger type inequalities on the Heisenberg group.

\begin{prop}{\label{prop:A1}}\cite{cohn2001best}
Let $\Omega\subset\H$ be a bounded domain.
\begin{enumerate}[label=\textnormal{(\roman*)}]
\item There exists $C=C(Q)>0$ such that
\begin{align*}
\sup_{u\in\W,\ \|\nah u\|_{L^Q(\Omega)}\le1}\ \frac{1}{|\Omega|}\int_\Omega
\exp\left(\beta\left|u(\xi)\right|^{\frac{Q}{Q-1}}\right)d\xi\le C
\quad\text{whenever } 0<\beta\le\alpha_Q,
\end{align*}
where $\alpha_Q= Q \left(2\pi^N \Gamma(\frac{1}{2})\Gamma(\frac{Q-1}{2}) 
\Gamma(\frac{Q}{2})^{-1}\Gamma(N)^{-1}\right)^{\frac{1}{Q-1}}$ 
(see also \cite{lamlutang2012} for this form of the constant).

\item For every $v\in\W$ and every $\beta>0$,
$\displaystyle\int_\Omega\exp\left(\beta\left|v\right|^{\frac{Q}{Q-1}}\right)d\xi<\infty$.
\item $\alpha_Q$ is sharp: if $\beta>\alpha_Q$ the supremum in $(i)$ is infinite.
\end{enumerate}
Moreover, $\alpha_Q=Q\,\sigma_Q^{1/(Q-1)}$, with $\sigma_Q$ as in \eqref{eq:constants}.
\end{prop}
\begin{lemma} \label{lem:L2}\cite{lilzhu2018}
 Assume $\{v_n\}\subset\W$ with $\|v_n\|=1$ such that $v_n\rightharpoonup v\neq0$ in $\W$ and $\nabla_{\mathbb{H}} v_n \to \nabla_{\mathbb{H}} v$ a.e. in $\Omega$. Then, for $\gamma$ satisfying $0<\gamma< \frac{\alpha_Q}{2^{Q'}(1-\|v\|^Q)^\frac{1}{Q-1}}$, 
 \begin{align*}
 \sup_n \int_\Omega \exp\left(\gamma \left|v_n\right|^{\frac{Q}{Q-1}}\right)d\xi<\infty.
 \end{align*}  
\end{lemma}
The natural energy functional  $J_{\lambda}: \W \to \mathbb{R}$ associated with the problem \eqref{eq:P1} is given by 
\begin{equation*}
J_{\lambda}(u)= \frac{1}{Q} \int_\Omega \left(1+(2\alpha)^{Q-1} \left|u\right|^{Q(2\alpha-1)}\right) \left|\nabla_{\mathbb{H}}u\right|^Q\ d\xi -\lambda \int_\Omega F(\xi,u)d\xi,
\end{equation*}
where $F(\xi,t)= \int_0 ^t f(\xi,s)ds.$ One can easily see that the functional $J_{\lambda}$ is not well defined for all $u \in \W.$ To deal with this, we use a transformation $u=g(v)$, where $g$ is the unique solution of the initial value problem
\begin{align*}
g^{\prime}(t)= \frac{1}{\left(1+(2\alpha)^{Q-1} \left|g(t)\right|^{Q(2\alpha -1)}\right)^\frac{1}{Q}}, \quad t\in [0,\infty),
\end{align*}
extended to $\R$ by $g(t)=-g(-t)$ \text{ for } $t\in (-\infty,0]$.

\begin{lemma}{\label{lem:L3}}
    The function g possesses the following properties: 
    \begin{itemize}
        \item[$(g_1)$] g is uniquely defined and invertible. Also, $g \in C^{\infty}(\R);$
        \item[$(g_2)$] g(0)=0, \quad $0<g^{\prime}(t) \leq 1, \text{ for all } t\in \R ;$
        \item[$(g_3)$] $\left|g(t)\right| \leq \left|t\right|$; $\left|g(t)\right|^{2\alpha} \leq (2\alpha)^{\frac{1}{Q}} \left|t\right|,$ for all $t\in \R$;
        \item[$(g_4)$] $ \frac{1}{2} g(t)\leq \alpha t g^{\prime}(t) \leq \alpha g(t),  t>0, ~~ \frac{1}{2} g(t)\geq \alpha t g^{\prime}(t) \geq \alpha g(t),  t<0;$ 
        \item[$(g_5)$] $\displaystyle \lim_ {t\to \infty}\frac{g(t)}{t^{\frac{1}{2\alpha}}}=(2\alpha)^{\frac{1}{2\alpha Q}}$; 
        \item[$(g_6)$] $\left|g(t)\right| \geq \begin{cases}
            g(1)\left|t\right|, & \text{ if } \left|t\right| \leq 1,\\ g(1)\left|t\right|^\frac{1}{2\alpha}, & \text{ if } \left|t\right| \geq 1;
        \end{cases}$ 
        \item[$(g_7)$] $\left|g(t)\right|^{2\alpha -1} g^{\prime}(t) \leq \frac{1}{(2\alpha)^{\frac{Q-1}{Q}}},$ for all $t\in \R;$
        \item[$(g_{8})$] the function $\displaystyle\frac{(g(t))^\gamma g^{\prime}(t)}{t^{Q-1}}$ is strictly increasing for $\gamma \geq (2\alpha Q -1)$ and $t>0.$
        
    \end{itemize} 
\end{lemma}
\begin{proof}
Properties $(g_1)$-$(g_2)$ are immediate from the initial value problem. For $(g_3)$, $(g_5)$, $(g_6)$ and $(g_7)$, put $\varphi(t):=g(t)^{2\alpha}$ and $s:=g(t)^{2\alpha-1}$ for $t>0$ , so that
\begin{align}\label{eq:phiprime}
\varphi'(t)=2\alpha\, g(t)^{2\alpha-1}g'(t)=\frac{2\alpha\, s}{\left(1+(2\alpha)^{Q-1}s^{Q}\right)^{1/Q}} .
\end{align}
The right-hand side of \eqref{eq:phiprime} is strictly increasing in $s$ and tends to $(2\alpha)^{1/Q}$ as $s\to\infty$, this yields at once $(g_7)$ and on integrating, $\varphi(t)<(2\alpha)^{1/Q}t$, i.e.\ $g(t)^{2\alpha}\le(2\alpha)^{1/Q}t$, the second half of $(g_3)$. Moreover, $\varphi$ is convex with $\varphi(0)=0$, so $\varphi(t)/t$ is nondecreasing, giving the case $|t|\ge1$ of $(g_6)$; concavity of $g$ on $(0,\infty)$ gives the case $|t|\le1$ of $(g_6)$, the first half of $(g_3)$, and the right-hand inequality in $(g_4)$. Property $(g_5)$ follows from $\varphi'(t)\to(2\alpha)^{1/Q}$. Odd symmetry transfers all of these to $t<0$. 

For $(g_8)$, write $z(t):=(2\alpha)^{Q-1}g(t)^{Q(2\alpha-1)}$, so that $(g')^{-Q}=1+z$. Differentiating this identity,
\begin{align*}
-Q(g')^{-Q-1}g''=(2\alpha)^{Q-1}Q(2\alpha-1)g^{Q(2\alpha-1)-1}g' ,
\end{align*}
whence, for $t>0$,
\begin{align}\label{eq:gsecond}
\frac{g''}{g'}=-(2\alpha-1)\,\frac{z}{1+z}\,\frac{g'}{g} .
\end{align}
Set $h(t):=g(t)^{\gamma}g'(t)t^{-(Q-1)}$. Using \eqref{eq:gsecond},
\begin{align*}
\frac{h'(t)}{h(t)}=\gamma\frac{g'}{g}+\frac{g''}{g'}-\frac{Q-1}{t}
=\frac{g'}{g}\left[\gamma-\frac{(2\alpha-1)z}{1+z}\right]-\frac{Q-1}{t},
\end{align*}
so that
\begin{align*}
\frac{t\,h'(t)}{h(t)}=\frac{tg'(t)}{g(t)}\left[\gamma-\frac{(2\alpha-1)z}{1+z}\right]-(Q-1).
\end{align*}
By $(g_4)$, $tg'(t)/g(t)\ge\frac{1}{2\alpha}$, and $0\le z/(1+z)<1$, so the bracket is $>\gamma-(2\alpha-1)\ge 2\alpha Q-2\alpha=2\alpha(Q-1)$ whenever $\gamma\ge 2\alpha Q-1$. Therefore
\begin{align*}
\frac{t\,h'(t)}{h(t)}>\frac{1}{2\alpha}\cdot 2\alpha(Q-1)-(Q-1)=0,
\end{align*}
i.e.\ $h$ is strictly increasing on $(0,\infty)$. \qed
\end{proof}

We observe that $\left|\nah(\mou)\right|^{Q-2}\nah(\mou)=(2\alpha)^{Q-1}|u|^{(2\alpha-1)Q-2\alpha}u\left|\nah u\right|^{Q-2}\nah u$, so 
\begin{align*}
\Delta_Qu &+\Delta_Q(\mou) \left|u\right|^{2\alpha-2}u \\&= \left( 1 + (2\alpha)^{Q-1} 
\left|u\right|^{Q(2\alpha-1)}\right)  \Delta_Qu + (2\alpha)^{Q-1} (Q-1) (2\alpha-1) \left|u\right| ^{Q(2\alpha -1)-2}  \left|\nah u\right|^Q u. 
\end{align*} 

Take $v=g^{-1}(u).$ Then employing the definition of $g$, we obtain 
\begin{align*}
(g^{-1})^{\prime}(t)&= \frac{1}{g^{\prime}(g^{-1}(t))}= \left(1+(2\alpha)^{Q-1} \left|t\right|^{Q(2\alpha-1)}\right)^\frac{1}{Q}, \\(g^{-1})^{\prime \prime}(t)&=(2\alpha)^{Q-1}(2\alpha-1) \left(1+(2\alpha)^{Q-1} \left|t\right|^{Q(2\alpha-1)}\right)^{\frac{1}{Q}-1} \left|t\right|^{Q(2\alpha-1)-2} t.
\end{align*}
Also, $\nah u=g^{\prime}(v) \nah v$ or $\nah v= (g^{-1})^{\prime}(u) \nah u.$ Thus, since $(g^{-1})'(u)>0$, we have 
 \begin{align*}
     \Delta_Q v &= \nah\cdot \left(\left|(g^{-1})^{\prime}(u) \nah u \right|^{Q-2} (g^{-1})^{\prime}(u) \nah u\right)\\ 
     &= \nah\cdot\left(\left((g^{-1})^{\prime}(u)\right)^{Q-1}\left|\nah u\right|^{Q-2}\nah u\right)\\
     &=\left((g^{-1})^{\prime}(u)\right)^{Q-1}\Delta_Qu+(Q-1)\left((g^{-1})^{\prime}(u)\right)^{Q-2}(g^{-1})^{\prime\prime}(u)\left|\nah u\right|^{Q}\\
     &= \left(\Delta_Qu +\Delta_Q(\mou) \left|u\right|^{2\alpha-2}u\right)g^{\prime}(v) = -\lambda f(\xi,g(v))g^{\prime}(v).
 \end{align*}
 So, the problem \eqref{eq:P1} becomes
 \begin{equation}
 \begin{cases}
    -\Delta_Q v= \lambda f(\xi,g(v))g^{\prime}(v) & \text{in }   \Omega, \\ 
    \qquad v =0   & \text{on } \partial\Omega. 
\end{cases}
\tag{$\mathcal{P^{\prime}_{\lambda}}$}
\label{eq:P2}
\end{equation}
Thus there is energy functional $I_{\lambda}: \W \to \mathbb{R}$ associated with the problem \eqref{eq:P2} 
\begin{equation*}
    I_{\lambda}(v)= \frac{1}{Q} \int_\Omega \left|\nah v\right|^Qd\xi -\lambda \int_\Omega F(\xi,g(v))d\xi.
\end{equation*}
For the case $\lambda=1,$ we denote $\I$ by $I.$

Using the properties of $f$ and $g$, one can easily derive that $I_{\lambda}$ is well defined on $\W$ and $I_{\lambda} \in C^{1}(\W, \R)$. Moreover, for all $w\in \W,$
\begin{equation*}
\langle I_{\lambda}^{\prime}(v),w \rangle= \int_\Omega \left|\nah v \right|^{Q-2} \nah v.\nah w d\xi - \lambda \int_\Omega f(\xi,g(v))g^{\prime}(v)w d\xi. 
\end{equation*}
\begin{definition}\label{def:weak}
We say $v \in \W$ is a weak solution of \eqref{eq:P2} if 
\begin{align*}
\int_\Omega \left|\nah v \right|^{Q-2} \nah v.\nah w d\xi = \lambda \int_\Omega f(\xi,g(v))g^{\prime}(v)w d\xi, \text{ for all } w \in \W,
\end{align*}
and $v \in \W$ is said to be a nodal solution if $v$ is a solution with $v^{\pm} \not\equiv 0$ a.e. in $\Omega.$
\end{definition}

\begin{rem}
    One can easily verify that problem \eqref{eq:P1} is equivalent to problem \eqref{eq:P2}, which takes $u=g(v)$ as its solution. Consequently, it suffices to establish the existence of a solution to the problem \eqref{eq:P2}.
\end{rem}
We close this section with two preparatory results: a pointwise bound on $f$ and $F$ which will be used throughout, and a regularity statement for weak solutions of \eqref{eq:P2}.

Suppose that $f$ satisfies $(f_3)$ and $(\mathcal{A}_2)$. By $(f_3)$, for given $\epsilon>0$ there exists $\delta>0$ such that $\left|f(\xi,t)\right| \le \epsilon \left|t\right|^{2\alpha Q-1}$ for all $\xi \in \Omega$ and $\left|t\right|<\delta.$  
    Let $\beta_0\ge0$ be the constant provided by $(\mathcal{A}_2)$. Then, for every $\beta>\beta_0$ and every $q\ge1$, the critical growth condition together with the
    continuity of $f$ yields a constant $C_1=C_1(q,\delta,\beta)>0$ such that
    \begin{align*}
    \left|f(\xi,t)\right| \le C_1 \left|t\right|^{q-1} \exp\left({\beta \left|t\right|^{\frac{2\alpha Q}{Q-1}}}\right) \text{ for all } \xi \in \Omega \text{ and } \left|t\right| \geq \delta.
    \end{align*} 
    On combining both, we get 
    \begin{equation}{\label{eq:s4}}
    \left|f(\xi,t)\right| \leq \epsilon\left|t\right|^{2 \alpha Q-1} + C_1 \left|t\right|^{q-1} \exp\left({\beta \left|t\right|^{\frac{2\alpha Q}{Q-1}}}\right) \quad \forall (\xi,t) \in \Omega \times \mathbb{R},
    \end{equation} 
    and hence, integrating \eqref{eq:s4} on the segment joining $0$ to $t$ and using that the integrand is nondecreasing in $\left|t\right|$,
    \begin{equation}{\label{eq:s5}} 
    \left|F(\xi,t)\right| \leq \epsilon C_2\left|t\right|^{2 \alpha Q} + C_3 \left|t\right|^{q} \exp\left({\beta\left|t\right|^{\frac{2\alpha Q}{Q-1}}}\right) \quad \forall (\xi,t) \in \Omega \times \mathbb{R},
    \end{equation}
    with $C_2=(2\alpha Q)^{-1}$ and $C_3=C_1$.
\begin{rem}
If $f$ satisfies $(\mathcal{A}_1)$ in place of $(\mathcal{A}_2)$, then \eqref{eq:s4} and \eqref{eq:s5} hold for \emph{every} $\beta>0$, with $C_1$ depending on $\beta$; this is the case $\beta_0=0$. All the estimates below that invoke \eqref{eq:s4} or \eqref{eq:s5} are used in
this stronger form in the subcritical setting of Theorem~\ref{T2}.
\end{rem}

\begin{prop}\label{prop:reg}
Let $f$ satisfy $(f_3)$ and $(\mathcal{A}_2)$, and let $v\in\W$ be a weak solution of \eqref{eq:P2}. Then
\begin{align}\label{eq:hLr}
h:=\lambda\,f(\xi,g(v))g^{\prime}(v)\in L^{r}(\Omega)\qquad\text{for every }r\in[1,\infty),
\end{align}
and $v\in L^{\infty}_{loc}(\Omega)\cap C^{0,\gamma}_{loc}(\Omega)$ for some $\gamma\in(0,1)$, the H\"older continuity being with respect to the Carnot-Carath\'eodory distance $d$ associated with $X_1,\dots,X_N,Y_1,\dots,Y_N$. Since $d$ induces the Euclidean topology on $\H$, the sets $\{v>0\}$ and $\{v<0\}$ are open.
\end{prop}
\begin{proof}
Fix $r\in[1,\infty)$. By Lemma~\ref{lem:L3}$(g_2)$ we have $0<g^{\prime}\le1$, and by $(g_3)$,
$\left|g(t)\right|\le\left|t\right|$ together with
$\left|g(t)\right|^{\frac{2\alpha Q}{Q-1}}\le(2\alpha)^{\frac{1}{Q-1}}\left|t\right|^{\frac{Q}{Q-1}}$; hence \eqref{eq:s4} gives, almost everywhere in $\Omega$,
\begin{align*}
\left|h\right|\le\lambda\epsilon\left|v\right|^{2\alpha Q-1}
+\lambda C_1\left|v\right|^{q-1}
\exp\left(\beta(2\alpha)^{\frac{1}{Q-1}}\left|v\right|^{\frac{Q}{Q-1}}\right).
\end{align*}
Choose $s>1$ and set $s^{\prime}=s/(s-1)$. By H\"older's inequality,
\begin{align*}
\int_\Omega\left|h\right|^{r}d\xi
\le C\|v\|_{L^{r(2\alpha Q-1)}(\Omega)}^{r(2\alpha Q-1)}
+C\|v\|_{L^{r(q-1)s^{\prime}}(\Omega)}^{r(q-1)}
\left(\int_\Omega\exp\left(rs\beta(2\alpha)^{\frac{1}{Q-1}}
\left|v\right|^{\frac{Q}{Q-1}}\right)d\xi\right)^{\frac{1}{s}} .
\end{align*}
The two Lebesgue norms are finite because $\W\hookrightarrow L^{m}(\Omega)$ for every $m<\infty$, and the exponential factor is finite by Proposition~\ref{prop:A1}$(ii)$, which applies for \emph{every} positive constant in the exponent since $v$ is a fixed element of
$\W$; no bound on $\|v\|$ is required. This proves \eqref{eq:hLr}.\\
We now regard $v$ as a weak solution of the inhomogeneous equation $-\Delta_Qv=h$ in $\Omega$, with $h$ a fixed datum. In the notation of the structural conditions $(S)$ of \cite{capognadanielligarofalo1993} this corresponds to
\begin{align*}
A(\xi,u,\zeta)=\left|\zeta\right|^{Q-2}\zeta,\qquad f(\xi,u,\zeta)=h(\xi),
\end{align*}
so that $c_1=1$, $g_2=g_3=0$, $f_1=f_2=0$, $h_3=0$ and $f_3=\left|h\right|$, with $p=Q$; note that the weak formulation \cite[(1.3)]{capognadanielligarofalo1993} coincides with the one in
Definition~\ref{def:weak}.
The integrability hypotheses (i)--(iii) of \cite[\S3]{capognadanielligarofalo1993}, and the strengthened form of (i)  required there for H\"older continuity, are satisfied. Indeed, in the notation of the structural conditions $(S)$ of that paper, our choice $A(\xi,u,\zeta)=\left|\zeta\right|^{Q-2}\zeta$ and $f(\xi,u,\zeta)=h(\xi)$ makes all the coefficients vanish except the zeroth-order inhomogeneous term, which equals $\left|h\right|$; hypotheses (i) and (iii) are therefore vacuous, and (ii) requires only that $\left|h\right|\in L^{s}_{loc}(\Omega)$ for some $s>Q/p$. Since $p=Q$ here, this reads $s>1$, which holds by \eqref{eq:hLr}. As the range of exponents in
\cite[\S3]{capognadanielligarofalo1993} is $1<p\le Q$, their Theorem 3.4 applies and gives $v\in L^{\infty}_{loc}(\Omega)$, and their Theorem 3.35 then gives $v\in C^{0,\gamma}_{loc}(\Omega)$ for some $\gamma\in(0,1)$.\qed
\end{proof}

\begin{lemma}[Weak Harnack inequality]\label{lem:wharnack}
Let $U\subset\H$ be open and let $u\in S^{1,Q}_{loc}(U)$, $u\ge0$, satisfy
\begin{align}\label{eq:supersol}
\int_U\left|\nah u\right|^{Q-2}\nah u\cdot\nah\varphi\,d\xi\ \ge\ 0
\qquad\text{for all }\varphi\in C_0^\infty(U),\ \varphi\ge0 .
\end{align}
Then there exist $p_0>0$, $C>0$ and $R_0>0$, depending only on $Q$, such that for every gauge
ball $B_R$ with $B_{4R}\subset U$ and $R\le R_0$,
\begin{align}\label{eq:wharnack}
\left(\fint_{B_{2R}}u^{p_0}\,d\xi\right)^{\frac{1}{p_0}}\ \le\ C\essinf_{B_R}u .
\end{align}
\end{lemma}
\begin{proof}
For $\varepsilon>0$ the function $\bar u=u+\varepsilon$ satisfies \eqref{eq:supersol} and is
bounded below by $\varepsilon$. We follow the proof of
\cite[Theorem~3.1]{capognadanielligarofalo1993} with all lower-order coefficients equal to
zero. That proof establishes, for positive solutions, the estimate
\cite[(3.26)]{capognadanielligarofalo1993}, obtained from the test function
$\eta^{Q}\bar u^{1-Q}$ with $\eta\in C_0^\infty(B_{2R})$ together with the John-Nirenberg
theorem on $B_{2R}$, and \cite[(3.28)]{capognadanielligarofalo1993}, obtained by Moser
iteration from the test functions $\eta^{Q}\bar u^{\beta}$, $\beta\le1-Q$, on the balls
$B_{aR}$, $1\le a\le2$. Both test functions are nonnegative, and inspection of the
computations show that the weak formulation enters there only through the inequality
\eqref{eq:supersol}; hence, both estimates hold for nonnegative supersolutions, and they
combine to
\begin{align*}
\left(\fint_{B_{2R}}\bar u^{p_0}\right)^{\frac{1}{p_0}}
\le C\left(\fint_{B_{2R}}\bar u^{-p_0}\right)^{-\frac{1}{p_0}}
\le C\essinf_{B_R}\bar u .
\end{align*}
The remaining estimate \cite[(3.27)]{capognadanielligarofalo1993}, which bounds
$\esssup\bar u$ from above through \cite[Lemma~3.29]{capognadanielligarofalo1993}, is stated
there for solutions and is not needed. Letting $\varepsilon\to0$ gives \eqref{eq:wharnack}.
\end{proof}

\begin{rem}
Lemma~\ref{lem:wharnack} is the sub-Riemannian instance of a well-known fact; see
\cite[Chapters 8-9]{bjornbjorn2011} and \cite{kinnunenshanmugalingam2001} for the statement on
doubling metric measure spaces supporting a Poincar{\'e} inequality, of which $\H$ is an example
by \cite{jerison1986}. We include the argument above because we could not locate the statement
in this form in the literature on Carnot groups.
\end{rem}

\section{Convergence results}
In this section, we prove some convergence results, which will be used to prove the existence of solutions.
\begin{lemma}{\label{lem:L4}}
       Let $\{v_n\}$ in $L^1(\Omega)$ such that $v_n \to v$ in $L^1(\Omega)$ and let $f$ be a continuous function. Then, $f(\xi,v_n) \to f(\xi,v)$ in $L^1(\Omega)$ provided $f(\xi,v),f(\xi,v_n) \in L^1(\Omega)$ for all $n$ and $\int_\Omega \left|f(\xi, v_n)v_n \right|d\xi \leq C.$
   \end{lemma}
\begin{proof}
    The proof directly follows from the definitions and assertions of Lemma $5.5$ in \cite{lamlutang2012}.
\end{proof}

\begin{lemma}{\label{lem:L5}}
   Suppose $f$ satisfies the assumptions $(f_1)$, $(f_3), (f_4),$ and $(\mathcal{A}_2).$  Let $\{v_n\} \subset \W$ be a bounded $(PS)$-sequence for $I$ at level $c.$ Then,  $\{v_n\}$ has a subsequence, still denoted by  $\{v_n\}$, and $v \in \W$ such that $\nah v_n \to \nah v$ almost everywhere in $\Omega.$ Moreover, 
   \begin{align*}
   \left|\nah v_n \right|^{Q-2} \nah v_n \rightharpoonup \left|\nah v \right|^{Q-2} \nah v \text{ weakly in } \left(L^{\frac{Q}{Q-1}}(\Omega)\right)^{Q-2}. 
   \end{align*}
\end{lemma}
\begin{proof}The sequence $\{v_n\}$ is bounded in $\W$ by hypothesis. Therefore, there exists $v \in \W$ such that $v_n \rightharpoonup v$ in $\W$, $v_n \to v$ in $L^r(\Omega), r\in [1,\infty)$, $v_n(\xi) \to v(\xi)$ almost everywhere in $\Omega$ as $n\to \infty.$  Consequently, $\{\left|\nah v_n \right|^{Q-2} \nah v_n\}$ is bounded in $\left(L^{\frac{Q}{Q-1}}(\Omega)\right)^{Q-2}.$ Thus, up to a subsequence, there exists $ w\in\left(L^{\frac{Q}{Q-1}}(\Omega)\right)^{Q-2}$ such that 
\begin{align*}
\left|\nah v_n \right|^{Q-2} \nah v_n \rightharpoonup w \text{ in } \left(L^{\frac{Q}{Q-1}}(\Omega)\right)^{Q-2} \text{ as } n\to \infty.
\end{align*} 
Also, $\{\left|\nah v_n\right|^Q\}$ is bounded in $L^1(\Omega)$, it implies that there exists a non-negative radon measure $\sigma$ such that, up to a subsequence, we have 
\begin{align*}
\left|\nah v_n\right|^Q \rightharpoonup \sigma  \text{ in } C(\bar{\Omega})^* \text{ as } n\to \infty.
\end{align*}
    Now, we assert that $w= \left|\nah v \right|^{Q-2} \nah v.$ Fix once and for all a number $\tau>0$ with
    \begin{align}\label{eq:taufix}
    \tau^{\frac{1}{Q-1}}<\frac{\alpha_Q}{(2\alpha)^{\frac{1}{Q-1}}} ,
    \end{align}
    and define 
    \begin{align*}
    X_\tau=\{\xi \in \bar{\Omega}: \sigma(B_d(\xi) \cap \bar{\Omega}) \geq \tau, \text{ for all } d>0\}.
    \end{align*} 
    Then, $X_\tau$ is a finite set. Indeed, if $\xi\in X_\tau$ then, letting $d\downarrow0$, $\sigma(\{\xi\})\ge\tau$, so every point of $X_\tau$ is an atom of $\sigma$ of mass at least $\tau$. If $X_\tau$ contained $k$ distinct points, then $k\tau\le\sigma(\bar\Omega)\le\sup_n\|v_n\|^Q=:C<\infty$, whence $\#X_\tau\le C/\tau$. Therefore, $X_\tau= \{\xi_1, \xi_2,\dots,\xi_k\}.$\\ 

    \textbf{Claim $1.$} With $\tau$ as in \eqref{eq:taufix}, 
    \begin{equation}{\label{eq:s6}}
    \lim_{n\to\infty} \int_K f(\xi,g(v_n))g^{\prime}(v_n)v_n d\xi=  \int_K f(\xi,g(v))g^{\prime}(v)v d\xi,
    \end{equation} 
    where $K$ is any compact subset of $\bar{\Omega} \setminus X_\tau.$ Let $\xi_0 \in K$ and $d_0>0$ such that $\sigma(B_{d_0}(\xi_0) \cap \bar{\Omega})< \tau.$ Choose $\phi \in C_c^\infty(\Omega)$ satisfying $0\leq \phi(\xi)\leq 1, \forall \xi \in \Omega,$ $\phi\equiv 1$ in $B_{\frac{d_0}{2}}(\xi_0)\cap \bar{\Omega}$ and $\phi \equiv 0$ in $\bar{\Omega}\setminus (B_{d_0}(\xi_0)\cap \bar{\Omega}).$ Then, 
    \begin{align*}
    \lim_{n\to\infty} \int_{B_{\frac{d_0}{2}}(\xi_0)\cap \bar{\Omega}} \left|\nah v_n\right|^Q d\xi \leq \lim_{n\to\infty} \int_{B_{d_0}(\xi_0)\cap \bar{\Omega}} \phi\left|\nah v_n\right|^Q d\xi \leq \sigma(B_{d_0}(\xi_0) \cap \bar{\Omega})<\tau.
    \end{align*} 
    Thus, for small enough $\epsilon>0$ and sufficiently large $n$, we get 
    \begin{equation}{\label{eq:s7}}
        \int_{B_{\frac{d_0}{2}}(\xi_0)\cap \bar{\Omega}} \left|\nah v_n\right|^Q d\xi \leq (1-\epsilon)\tau.
    \end{equation}
    Under the hypothesis that $f$ is continuous and has critical growth, for any $\beta > \beta_0$, there exists $C>0$ such that 
    \begin{align*} 
    \left|f(\xi,g(v_n))\right| \leq C \exp\left(\beta \left|g(v_n)\right|^\frac{2\alpha Q}{Q-1}\right) \le C\exp\left(\beta(2\alpha)^{\frac{1}{Q-1}}\left|v_n\right|^{\frac{Q}{Q-1}}\right),
    \end{align*}
   where the last step uses Lemma~\ref{lem:L3}$(g_3)$.
    Choose $\beta >\beta_0$(close to $\beta_0$), $q>1$ (close to $1$) and $\epsilon >0$ such that 
    \begin{equation}{\label{eq:s8}}
        q\beta (2\alpha)^\frac{1}{Q-1} \tau^\frac{1}{Q-1}(1-\epsilon)^\frac{1}{Q-1}< \alpha_Q.
    \end{equation}
    Taking into account Trudinger-Moser inequality, Lemma~\ref{lem:L3}$(g_3)$, and \eqref{eq:s7}-\eqref{eq:s8}, we deduce
    \begin{small}
    \begin{align}
        &\int_{B_{\frac{d_0}{2}}(\xi_0)\cap \bar{\Omega}} \left|f(\xi, g(v_n))\right|^q d\xi \notag\\&  \leq C\int_{B_{\frac{d_0}{2}}(\xi_0)\cap \bar{\Omega}} \exp \left(q\beta (2\alpha)^\frac{1}{Q-1}\left(\int_{B_{\frac{d_0}{2}}(\xi_0)\cap \bar{\Omega}} \left|\nah v_n\right|^Q d\xi \right)^\frac{1}{Q-1} \left(\frac{\left|v_n\right|^Q}{\int_{B_{\frac{d_0}{2}}(\xi_0)\cap \bar{\Omega}} \left|\nah v_n\right|^Q d\xi }\right)^\frac{1}{Q-1}\right)d\xi \notag \\& \leq C \int_{B_{\frac{d_0}{2}}(\xi_0)\cap \bar{\Omega}} \exp \left(q\beta (2\alpha)^\frac{1}{Q-1}\tau^\frac{1}{Q-1} (1-\epsilon)^\frac{1}{Q-1} \left(\frac{\left|v_n\right|^Q}{\int_{B_{\frac{d_0}{2}}(\xi_0)\cap \bar{\Omega}} \left|\nah v_n\right|^Q d\xi }\right)^\frac{1}{Q-1}\right)d\xi \leq C_1, {\label{eq:s9}}
    \end{align}
    \end{small}
    for some $ C_1>0.$
    Using the asymptotic growth assumptions on $f$, we get
    \begin{equation}{\label{eq:s10}}
        \lim_{t\to \infty} \frac{f(\xi,t)t}{(f(\xi,t))^s}=0 \text{ uniformly in } \Omega, \forall s>1.
    \end{equation}
    By \eqref{eq:s9} and \eqref{eq:s10} together with Lemma~\ref{lem:L3}$(g_4)$, we have that $\{f(\xi,g(v_n))g^{\prime}(v_n)v_n:n\in \mathbb{N}\}$ is an equi-integrable family of functions over $B_{\frac{d_0}{2}}(\xi_0)\cap \bar{\Omega}.$ The continuity of $f$ and $g$ imply that 
    \begin{align*}
    f(\xi,g(v_n))g^{\prime}(v_n)v_n \to f(\xi,g(v))g^{\prime}(v)v \text{ almost everywhere in } \Omega \text{ as } n\to\infty.
    \end{align*} 
    Thus, applying Vitali's convergence theorem, we obtain
    \begin{equation*}
        \lim_{n\to\infty}\int_{B_{\frac{d_0}{2}}(\xi_0)\cap \bar{\Omega}} \left|f(\xi,g(v_n))g^{\prime}(v_n)v_n- f(\xi,g(v))g^{\prime}(v)v\right|d\xi=0,
    \end{equation*}
   and the claim 1 follows by compactness of $K.$ Now, choose $\sigma_0>0$ small enough such that $B_{\sigma_0}(\xi_i)\cap B_{\sigma_0}(\xi_j)=\emptyset \text{ for } i\neq j$ and 
   \begin{align*}
   \Omega_{\sigma_0}=\{\xi \in \bar{\Omega}: \rho(\xi_i^{-1}\circ\xi) \geq \sigma_0,\ i=1,2,\dots,k\}.
   \end{align*} 
   \textbf{Claim $2.$} $\displaystyle{\lim_{n\to \infty}} \int_{\Omega_{\sigma_0}} \left(\left|\nah v_n \right|^{Q-2} \nah v_n -\left|\nah v \right|^{Q-2} \nah v\right).(\nah v_n -\nah v)d\xi=0.$ \\ 
   Choose $\phi \in C^\infty(\H,[0,1])$ such that $\phi \equiv 1$ in $B_{\frac{1}{2}}(0)$ and $\phi \equiv 0$ in $\H \setminus B_1(0).$ For $0<\sigma <\sigma_0,$ define 
   \begin{equation*} 
   \phi_\sigma(\xi)=1-\sum_{i=1}^k \phi \left(\delta_{1/\sigma}\left(\xi_i^{-1}\circ\xi\right)\right).
   \end{equation*}
  Since $\rho$ is $\delta_\theta$-homogeneous of degree one and $\nah$ is left-invariant, $\left|\nah\phi_\sigma\right|\le C\sigma^{-1}$ with $C$ independent of $\sigma$.
   Clearly, $0 \le \phi_\sigma \le 1$, $\phi_\sigma \equiv 1$ in $\bar{\Omega}_\sigma=\bar{\Omega} \setminus \cup_{i=1}^k B_\sigma(\xi_i)$, $\phi_\sigma \equiv 0$ in $\cup_{i=1}^k B_{\frac{\sigma}{2}}(\xi_i).$ Moreover, for each $\sigma,$ $\{\phi_\sigma v_n\}$ is a bounded sequence in $\W.$ Using the fact that $\{v_n\} \subset \W$ is a $(PS)-$sequence for $I$, we have 
   \begin{align*}
   \frac{1}{Q}\int_\Omega \left|\nah v_n\right|^Q d\xi - \int_\Omega F(\xi, g(v_n))d\xi \to c \text{ as } n\to \infty
   \end{align*} 
   \begin{equation}{\label{eq:s11}}
     \text{ and }  \left|\int_\Omega \left|\nah v_n \right|^{Q-2} \nah v_n.\nah w d\xi- \int_\Omega f(\xi,g(v_n))g^{\prime}(v_n)wd\xi\right| \leq \epsilon_n \|w\| \quad \forall w\in \W.
   \end{equation}
   Taking $w=\phi_\sigma v_n$ and $w= \phi_\sigma v$ in \eqref{eq:s11}, we get 
   \begin{small} 
   \begin{equation}{\label{eq:s12}}
       \int_\Omega \left(\phi_\sigma \left|\nah v_n\right|^Q + v_n \left|\nah v_n \right|^{Q-2} \nah v_n.\nah \phi_\sigma -  \phi_\sigma f(\xi,g(v_n))g^{\prime}(v_n)v_n \right)d\xi \leq \epsilon_n \|\phi_\sigma v_n\|,
   \end{equation}
   \end{small}
   \begin{small}
   \begin{equation}{\label{eq:s13}}
       \int_\Omega \left(-\phi_\sigma \left|\nah v_n \right|^{Q-2} \nah v_n.\nah v - v \left|\nah v_n \right|^{Q-2} \nah v_n.\nah \phi_\sigma + \phi_\sigma f(\xi,g(v_n))g^{\prime}(v_n)v \right)d\xi \leq \epsilon_n \|\phi_\sigma v\|.
   \end{equation}
   \end{small}
   Since the map $\psi:\R^{2N} \to \mathbb{R}: \zeta\mapsto \left| \zeta\right|^Q$ is strictly convex, 
   \begin{align*}
   0 \leq \left(\left|\nah v_n \right|^{Q-2} \nah v_n -\left|\nah v \right|^{Q-2} \nah v\right).(\nah v_n -\nah v).
   \end{align*}  
   Using \eqref{eq:s12} and \eqref{eq:s13}, and the fact that $\phi_\sigma\equiv1$ on $\bar\Omega_{\sigma_0}$ together with the nonnegativity of the integrand, we have 
   \begin{small} 
   \begin{align} 0 &\leq \int_{\bar{\Omega}_{\sigma_0}} \left(\left|\nah v_n \right|^{Q-2} \nah v_n -\left|\nah v \right|^{Q-2} \nah v\right).(\nah v_n -\nah v)d\xi \notag \\
   &\leq \int_{\Omega}\phi_\sigma\left(\left|\nah v_n \right|^{Q-2} \nah v_n -\left|\nah v \right|^{Q-2} \nah v\right).(\nah v_n -\nah v)d\xi \notag \\
   &= \int_{\Omega} \left(\left|\nah v_n \right|^Q \phi_\sigma -\left|\nah v_n \right|^{Q-2} (\nah v_n.\nah v)\phi_\sigma - \left|\nah v \right|^{Q-2} (\nah v_n.\nah v)\phi_\sigma + \left|\nah v \right|^Q \phi_\sigma\right)d\xi \notag \\
   & \le \int_\Omega \left|\nah v_n \right|^{Q-2} \nah v_n.\nah \phi_\sigma (v-v_n)d\xi + \int_\Omega \phi_\sigma \left|\nah v\right|^{Q-2} \nah v .(\nah v- \nah v_n)d\xi \notag \\
   & \quad + \int_\Omega \phi_\sigma f(\xi,g(v_n))g^{\prime}(v_n)(v_n-v)d\xi + \epsilon_n \|\phi_\sigma v_n\| +  \epsilon_n \|\phi_\sigma v\|. {\label{eq:s14}}
    \end{align}
    \end{small}
       The interpolation inequality, for any $r>0$, $ab \leq r a^{\frac{Q}{Q-1}} +r^{1-Q} b^Q$, gives us
       \begin{small}
       \begin{align}
           \int_\Omega \left|\nah v_n \right|^{Q-2} \nah v_n.\nah \phi_\sigma (v-v_n)d\xi &\leq r \int_\Omega \left|\nah v_n\right|^Q d\xi + r^{1-Q} \int_\Omega \left| \nah \phi_\sigma\right|^Q \left|v_n-v\right|^Q d\xi \notag \\& \leq r \int_\Omega \left|\nah v_n\right|^Q d\xi + r^{1-Q} \left(\int_\Omega \left|\nah \phi_\sigma\right|^{tQ} d\xi\right)^{\frac{1}{t}} \left(\int_\Omega \left|v-v_n \right|^{t^{\prime}Q} d\xi\right)^{\frac{1}{t^{\prime}}}. \notag
       \end{align}
       \end{small}
       Since $v_n \to v$ in $L^{t^{\prime}Q}(\Omega)$ and $r$ is arbitrary, it follows that 
       \begin{equation*}
           \limsup_{n \to \infty} \int_\Omega \left|\nah v_n \right|^{Q-2} \nah v_n.\nah \phi_\sigma (v-v_n)d\xi \leq 0.
       \end{equation*}
       Also, $v_n \rightharpoonup v$ in $\W$ implies that 
       \begin{equation}{\label{eq:s15}}
           \lim_{n \to \infty} \int_\Omega \phi_\sigma \left|\nah v\right|^{Q-2} \nah v .(\nah v- \nah v_n)d\xi=0.
       \end{equation}
      We can write 
      \begin{align}
          \int_\Omega \phi_\sigma f(\xi,g(v_n))g^{\prime}(v_n)(v_n-v)d\xi&=\int_\Omega \phi_\sigma f(\xi,g(v_n))g^{\prime}(v_n)v_nd\xi-\int_\Omega \phi_\sigma f(\xi,g(v))g^{\prime}(v)vd\xi \notag \\& \quad+ \int_\Omega \phi_\sigma f(\xi,g(v))g^{\prime}(v)vd\xi -\int_\Omega \phi_\sigma f(\xi,g(v_n))g^{\prime}(v_n)vd\xi.\label{eq:s16}
      \end{align}
      Using the generalized dominated convergence theorem together with Lemma~\ref{lem:L4}, we deduce that 
      \begin{align*}
      \int_\Omega \phi_\sigma f(\xi,g(v_n))g^{\prime}(v_n)vd\xi \to \int_\Omega \phi_\sigma f(\xi,g(v))g^{\prime}(v)vd\xi \text{ as } n \to \infty.
      \end{align*} 
      For $K=\bar{\Omega}_\frac{\sigma}{2}$ in \eqref{eq:s6},which is a compact subset of $\bar\Omega\setminus X_\tau$ on which $\phi_\sigma$ is supported, it follows that 
      \begin{align*}
      \int_\Omega \phi_\sigma f(\xi,g(v_n))g^{\prime}(v_n)v_n d\xi &= \int_K \phi_\sigma f(\xi,g(v_n))g^{\prime}(v_n)v_nd\xi \\& \to \int_K \phi_\sigma f(\xi,g(v))g^{\prime}(v)v\,d\xi = \int_\Omega \phi_\sigma f(\xi,g(v))g^{\prime}(v)v d\xi \text{ as } n \to \infty. 
      \end{align*} 
      Thus, by \eqref{eq:s16}, we get 
      \begin{equation}{\label{eq:s17}}
      \int_\Omega \phi_\sigma f(\xi,g(v_n))g^{\prime}(v_n)(v_n-v)d\xi \to 0 \text{ as } n \to \infty.
      \end{equation}
      Taking into account \eqref{eq:s14}-\eqref{eq:s15} and \eqref{eq:s17}, claim $2$ holds.  
      By the classical Simon inequality on $\R^{2N}$ \cite{simon1978}, the integrand in Claim $2$ dominates $c_Q\left|\nah v_n-\nah v\right|^Q$ for $Q\ge2$; hence $\nah v_n\to\nah v$ in $L^Q(\bar\Omega_{\sigma_0})$ and, along a further subsequence, almost everywhere on $\bar\Omega_{\sigma_0}$. Since $\sigma_0>0$ was arbitrary and $X_\tau$ is finite, hence Lebesgue-null, a diagonal argument gives
      \begin{align*}
      \nah v_n \to \nah v \text{ a.e. in } \Omega.
      \end{align*} 
      This immediately implies, up to a subsequence, that 
      \begin{align*}
      \left|\nah v_n \right|^{Q-2} \nah v_n \rightharpoonup \left|\nah v \right|^{Q-2} \nah v \text{ in } \left(L^{\frac{Q}{Q-1}}(\Omega)\right)^{Q-2} \text{ as } n\to\infty. 
      \end{align*} 
      \end{proof}

        \begin{lemma}{\label{lem:L6}}
        Suppose $f$ satisfies the assumptions $(f_1)$, $(f_3), (f_4),$  and $(\mathcal{A}_2).$ Let $\{v_n\} \subset \W$ be a bounded $(PS)$-sequence for $I$ at level $c.$  Then 
        \begin{enumerate}[label=\text{(\roman*)}]
           \item $\lim\limits_{n \to \infty} \int_\Omega f(\xi,g(v_n))g^{\prime}(v_n) \psi d\xi=\int_\Omega f(\xi,g(v))g^{\prime}(v) \psi d\xi,$ for all $\psi \in \W,$
           \item $\lim\limits_{n\to\infty}\int_\Omega F(\xi, g(v_n))d\xi = \int_\Omega F(\xi, g(v))d\xi.$
       \end{enumerate}
        \end{lemma}
      \begin{proof}
          Let $\Omega^{\prime} \subset\subset \Omega.$ Choose $\eta \in C^\infty_c(\Omega)$ such that $0 \leq \eta \leq 1$ and $\eta \equiv 1$ in $\Omega^{\prime}.$ An easy computation shows that 
          \begin{equation}{\label{eq:s18}} 
          \int_\Omega \left|\nah\left(\frac{\eta}{1+\left|v_n\right|}\right)\right|^Qd\xi=\int_\Omega \left|\frac{\nah \eta}{1+\left|v_n\right|}-\eta\,\operatorname{sgn}(v_n) \frac{\nah v_n}{(1+\left|v_n\right|)^2}\right|^Qd\xi \leq 2^{Q-1}(\|\eta\|^Q + \|v_n\|^Q) 
          \end{equation} 
          and thus, $\frac{\eta}{1+\left|v_n\right|} \in \W.$ 
          Since $\{v_n\}$ is a $(PS)$ sequence for $I,$ we have 
          \begin{equation}{\label{eq:s19}}
              \left|\int_\Omega \left|\nah v_n \right|^{Q-2} \nah v_n.\nah \phi d\xi-\int_\Omega f(\xi,g(v_n))g^{\prime}(v_n) \phi d\xi\right| \leq \epsilon_n \|\phi\|
          \end{equation} 
          where $\epsilon_n \to 0 $ as $n\to \infty.$ Taking into account \eqref{eq:s18} and \eqref{eq:s19} with $\phi=\frac{\eta}{1+\left|v_n\right|}$, and recalling that $f(\xi,g(v_n))g'(v_n)\ge0$ by $(f_1)$ and $(g_2)$, we obtain 
          \begin{align}
          \int_{\Omega^{\prime}} f(\xi, g(v_n)) \frac{g^{\prime}(v_n)}{1+\left|v_n\right|} d\xi &\leq \int_\Omega f(\xi,g(v_n)) g^{\prime}(v_n) \frac{\eta}{1+\left|v_n\right|} d\xi \notag \\& \leq \epsilon_n \left\|\frac{\eta}{1+\left|v_n\right|}\right\| + \int_\Omega \left|\nah v_n \right|^{Q-2} \nah v_n.\nah \left(\frac{\eta}{1+\left|v_n\right|}\right)d\xi \notag \\& \leq \epsilon_n 2^{\frac{Q-1}{Q}}( \|\eta\| + \|v_n\|) + \int_\Omega \left|\nah v_n \right|^{Q-1} (\left|\nah \eta \right| + \left| \nah v_n \right|) d\xi \notag \\& \leq \epsilon_n 2^{\frac{Q-1}{Q}}(\|\eta\| + \|v_n\|) + (\|\eta\| \|v_n\|^{Q-1} +\|v_n\|^Q) \leq C_1. {\label{eq:s20}}
        \end{align}
         Take $\phi=v_n$ in \eqref{eq:s19}, we get 
         \begin{align}
             \int_{\Omega^{\prime}} f(\xi, g(v_n)) g^{\prime}(v_n) v_nd\xi &\leq \int_\Omega f(\xi,g(v_n)) g^{\prime}(v_n)v_n d\xi \leq \epsilon_n \|v_n\| + \|v_n\|^Q \leq C_2. \label{eq:s21}
        \end{align}
       Note that by $(f_1)$ one has $f(\xi,g(v_n))=0$ on $\{v_n\le0\}$, so that in fact $f(\xi,g(v_n))g'(v_n)v_n=f(\xi,g(v_n))g'(v_n)\left|v_n\right|$ pointwise.
        Resuming information from \eqref{eq:s20} and \eqref{eq:s21}, we deduce 
        \begin{align}
            \int_{\Omega^{\prime}} f(\xi, g(v_n)) g^{\prime}(v_n)d\xi &\le 2 \int_{\Omega^{\prime} \cap \{\left|v_n\right|<1\}} \frac{f(\xi,g(v_n))g^{\prime}(v_n)}{1+\left|v_n\right|}d\xi + \int_{\Omega^{\prime} \cap \{\left|v_n\right|\geq1\}} f(\xi,g(v_n))g^{\prime}(v_n)\left|v_n\right| d\xi\ \notag \\ & \leq 2 \int_{\Omega^{\prime}} \frac{f(\xi,g(v_n))g^{\prime}(v_n)}{1+\left|v_n\right|}d\xi + \int_{\Omega^{\prime}}  f(\xi,g(v_n))g^{\prime}(v_n)v_n d\xi \leq 2C_1 +C_2 \notag.
        \end{align}
        This shows that the sequence $\{f(\xi,g(v_n))g^{\prime}(v_n)\}_{n\in \mathbb{N}}$ is bounded in $L^1_{loc}(\Omega).$ 
        
       We now identify the limit. Set $\tilde f(\xi,s):=f(\xi,g(s))g'(s)$, which is continuous on $\bar\Omega\times\R$. By the above, $\tilde f(\xi,v_n)\in L^1(\Omega)$ for every $n$ and $\int_\Omega\left|\tilde f(\xi,v_n)v_n\right|d\xi\le C_2$ by \eqref{eq:s21}; moreover $v_n\to v$ in $L^1(\Omega)$ by the compact embedding. Lemma~\ref{lem:L4} therefore yields
        \begin{align}\label{eq:L1conv}
        \tilde f(\xi,v_n)\longrightarrow \tilde f(\xi,v)\quad\text{in }L^1(\Omega)\text{ as }n\to\infty .
        \end{align}
        In particular $(i)$ holds for every $\psi\in C^\infty_c(\Omega)$, since such $\psi$ is bounded. On the other hand, \eqref{eq:s19} gives, for all $\phi\in\W$,
        \begin{align}\label{eq:dualbound}
        \left|\int_\Omega \tilde f(\xi,v_n)\phi\,d\xi\right|\le \epsilon_n\|\phi\|+\|v_n\|^{Q-1}\|\phi\|\le C\|\phi\| ,
        \end{align}
        so that $\{\tilde f(\cdot,v_n)\}$ is bounded in ${\overset{\scriptscriptstyle\circ}{S}{}^Q_1(\Omega)}^{*}$. Given $\psi\in\W$ and $\epsilon>0$, choose $\psi_\epsilon\in C^\infty_c(\Omega)$ with $\|\psi-\psi_\epsilon\|<\epsilon$; combining \eqref{eq:L1conv} for $\psi_\epsilon$ with \eqref{eq:dualbound} for $\psi-\psi_\epsilon$ and letting $\epsilon\downarrow0$ gives $(i)$ for all $\psi\in\W$.
        
        For $(ii)$, the boundedness hypothesis gives $v \in \W$ such that $v_n \rightharpoonup v$ in $\W$, $v_n \to v$ in $L^r(\Omega), r\in [1,\infty)$, $v_n(\xi) \to v(\xi)$ almost everywhere in $\Omega$ as $n\to \infty.$ Since $f, F,$ and $g$ are continuous, we have $F(\xi,g(v_n(\xi))) \to F(\xi,g(v(\xi)))$ and $f(\xi,g(v_n(\xi))) g(v_n(\xi)) \to f(\xi, g(v(\xi)))g(v(\xi))$ a.e. in $\Omega$ as $n\to\infty.$ With the help of H\"older inequality, Moser-Trudinger inequality, and \eqref{eq:s4}, one can easily see that the sequence $\{f(\xi,g(v_n(\xi))) g(v_n(\xi))\}$ is uniformly integrable in $L^1(\Omega).$ Thanks to Vitali convergence theorem, $f(\xi, g(v_n))g(v_n)$  converges to $f(\xi, g(v))g(v)$ in $L^1(\Omega)$ as $n\to\infty$. Finally, since $0\le\mu F(\xi,g(v_n))\le f(\xi,g(v_n))g(v_n)$ by $(f_4)$, the generalized dominated convergence theorem yields $(ii)$. 
      \end{proof} 

      \begin{rem}
    Lemma~\ref{lem:L6}$(ii)$ holds verbatim under the assumptions
    $(f_1^{\prime}),(f_2)-(f_5)$, since its proof uses only the growth bound \eqref{eq:s4},
    which is sign-free, together with $(f_4)$ and the nonnegativity of $F$ -- and the latter
    follows from $(f_1^{\prime})$ because $tf(\xi,t)>0$ for $t\ne0$ gives
    $F(\xi,s)=\int_0^{s}f(\xi,\tau)\,d\tau\ge0$ on both half-lines. In particular, $(ii)$ is
    available for the sign-changing sequences of Section~5. Part $(i)$ also holds
    under $(f_1^{\prime})$, but its proof must be modified, since
    $f(\xi,g(v_n))g^{\prime}(v_n)$ is then no longer of one sign and the integrals over
    $\Omega^{\prime}$ in \eqref{eq:s20} and \eqref{eq:s21} cannot be enlarged to $\Omega$; this is
    carried out in Lemma~\ref{lem:L6prime} below.
    Also, we can replace the growth condition $(\mathcal{A}_2)$ by $(\mathcal{A}_1).$
\end{rem}

\begin{lemma}\label{lem:L6prime}
Suppose $f$ satisfies $(f_1^{\prime})$, and let $\{v_n\}\subset\W$ be a bounded
$(PS)$-sequence for $\I$ with $v_n\rightharpoonup v$ in $\W$. Then
\begin{align*}
\lim_{n\to\infty}\int_\Omega f(\xi,g(v_n))g^{\prime}(v_n)\,\psi\,d\xi
=\int_\Omega f(\xi,g(v))g^{\prime}(v)\,\psi\,d\xi
\quad\text{for every }\psi\in\W .
\end{align*}
\end{lemma}
\begin{proof}
Write $\tilde f(\xi,s):=f(\xi,g(s))g^{\prime}(s)$; by $(f_1^{\prime})$ and
Lemma~\ref{lem:L3}$(g_2)$, together with the fact that $g$ is odd and increasing,
\begin{align}\label{eq:signf}
s\,\tilde f(\xi,s)>0\quad\text{for every }s\ne0 .
\end{align}
Only the local $L^{1}$ bound on $\{\tilde f(\cdot,v_n)\}$ has to be reproved, the remainder of
the proof of Lemma~\ref{lem:L6}$(i)$ using no sign condition. Let
$\Omega^{\prime}\subset\subset\Omega$ and let $\eta\in C^\infty_c(\Omega)$ satisfy
$0\le\eta\le1$ and $\eta\equiv1$ on $\Omega^{\prime}$. Testing \eqref{eq:s19} with
$\varphi_n:=\eta\,v_n/(1+\left|v_n\right|)$, which lies in $\W$ because
$\left|\nah\varphi_n\right|\le\left|\nah\eta\right|+\left|\nah v_n\right|$ and whose integrand
$\tilde f(\xi,v_n)\varphi_n=\eta\left|\tilde f(\xi,v_n)\right|\left|v_n\right|/
(1+\left|v_n\right|)$ is nonnegative on all of $\Omega$ by \eqref{eq:signf}, we obtain
\begin{align*}
\int_{\Omega^{\prime}\cap\{\left|v_n\right|\ge1\}}\left|\tilde f(\xi,v_n)\right|d\xi
\;\le\;2\int_\Omega \tilde f(\xi,v_n)\varphi_n\,d\xi
\;\le\;2\left(\epsilon_n+\|v_n\|^{Q-1}\right)\left(\|\eta\|+\|v_n\|\right) ,
\end{align*}
since $\left|v_n\right|/(1+\left|v_n\right|)\ge\frac12$ there. On
$\{\left|v_n\right|<1\}$ we have $\left|g(v_n)\right|\le g(1)$ and $0<g^{\prime}\le1$, so
$\left|\tilde f(\xi,v_n)\right|$ is bounded by the maximum of $\left|f\right|$ on the compact
set $\bar\Omega\times[-g(1),g(1)]$. Adding the two bounds shows that
$\{\tilde f(\cdot,v_n)\}$ is bounded in $L^{1}_{loc}(\Omega)$, and the proof of
Lemma~\ref{lem:L6}$(i)$ applies unchanged from that point on. 
\end{proof}

\section{Existence of a Positive Solution}
In this section, we prove the existence of a positive solution for the problem
\eqref{eq:P2}, by applying the Mountain Pass lemma due to Ambrosetti and Rabinowitz. Throughout
this section $\lambda>0$ is arbitrary; the range of $\lambda$ for which the conclusion holds is
determined in Lemma~\ref{lem:L9} and recorded in Theorem~\ref{T1}.

\begin{lemma}{\label{lem:L7}}
If $\{v_n\}$ be a $(PS)$-sequence of $\I$ at level $c$, then $\{v_n\}$ is bounded in $\W.$
\end{lemma}
\begin{proof}
    Let $\{v_n\} \subset \W$ be a $(PS)$-sequence of $\I$ at level $c.$ Then, for $n\to \infty,$ 
    \begin{align*}
    \frac{1}{Q}\|v_n\|^Q- \lambda\int_\Omega F(\xi,g(v_n))d\xi \to c,
    \end{align*} 
    and for all $w\in \W,$ 
    \begin{align*}
    \left|\langle \I^{\prime}(v_n),w \rangle\right|=\left|\int_\Omega \left|\nah v_n\right|^{Q-2}\nah v_n.\nah w d\xi - \lambda\int_\Omega f(\xi,g(v_n))g^{\prime}(v_n)wd\xi \right| \le \epsilon_n \|w\|.
    \end{align*} 
    Taking $w=v_n,$
    \begin{equation*}
        \left|\|v_n\|^Q- \lambda\int_\Omega f(\xi,g(v_n))g^{\prime}(v_n)v_nd\xi\right| \leq \epsilon_n \|v_n\|.
    \end{equation*}
    Let $w_n=: \frac{g(v_n)}{g^{\prime}(v_n)}.$ Differentiating the identity $(g')^{-Q}=1+(2\alpha)^{Q-1}\left|g\right|^{Q(2\alpha-1)}$ gives
    \begin{align*}
    g^{\prime\prime}=-(2\alpha)^{Q-1}(2\alpha-1)\left|g\right|^{Q(2\alpha-1)-2}g\,(g^{\prime})^{Q+2},
    \end{align*}
    whence, writing $\zeta:=\left(\tfrac{g}{g'}\right)'=1-\frac{g\,g^{\prime\prime}}{(g^{\prime})^{2}}$,
    \begin{align*}
    \nah w_n= \left(1+(2\alpha-1)\frac{(2\alpha)^{Q-1}  \left|g(v_n)\right|^{Q(2\alpha-1)}}{1+(2\alpha)^{Q-1} \left|g(v_n)\right|^{Q(2\alpha-1)}}\right) \nah v_n,
    \end{align*} 
    and since the fraction lies in $[0,2\alpha-1)$ we obtain $1\le\zeta<2\alpha$ and hence, $\left|\nah w_n\right| \leq 2\alpha \left|\nah v_n \right|.$ Consequently, $\|w_n\|\le 2\alpha\|v_n\|$. Moreover, using Lemma~\ref{lem:L3}$(g_4)$, $\left|w_n\right|\le 2\alpha\left|v_n\right|$, so that $w_n\in\W$. Thus, 
    \begin{equation}{\label{eq:s23}}
    \left|\langle \I^{\prime}(v_n),w_n \rangle\right| \leq \epsilon_n \|w_n\| \le 2\alpha \epsilon_n \|v_n\|.
    \end{equation}
  Using     $(f_4)$, we get 
    \begin{align}
        \I(v_n)-\frac{1}{\mu} \langle \I^{\prime}(v_n),w_n \rangle & \notag\geq \left(\frac{1}{Q}- \frac{2\alpha}{\mu}\right) \|v_n\|^Q + \frac{\lambda}{\mu} \int_\Omega \left(f(\xi, g(v_n))g(v_n)-\mu F(\xi, g(v_n))\right)d\xi \\& \geq \left(\frac{1}{Q}- \frac{2\alpha}{\mu}\right) \|v_n\|^Q \notag. 
    \end{align}
   
    Taking into account \eqref{eq:s23} and $\mu> 2\alpha Q,$ we obtain 
    \begin{align*}
    0< \left(\frac{1}{Q}- \frac{2\alpha}{\mu}\right) \|v_n\|^Q \leq \I(v_n)-\frac{1}{\mu} \langle \I^{\prime}(v_n),w_n \rangle \leq C + \frac{2\alpha}{\mu}\epsilon_n \|v_n\|,
    \end{align*} 
    where $\epsilon_n \to 0$ as $n \to \infty.$ Hence, $\{v_n\}$ is bounded in $\W.$
\end{proof}
\begin{lemma}{\label{lem:L8}}
    Suppose $f$ satisfies assumptions $(\mathcal{A}_2)$, $(f_1),(f_3) \text{ and } (f_4).$ Then 
     \begin{enumerate}[label=\text{(\roman*)}]
           \item  there exist $R, \rho >0$ such that $\I(v) \geq R \text{ for all $v$ satisfying } \|v\|= \rho.$
           \item for any $v \in \W$ with $v^{+}\not\equiv0$, $\I(tv) \to -\infty$ as $t \to \infty.$
     \end{enumerate}
\end{lemma}
\begin{proof} $(i)$  Taking into account H\"older inequality, Lemma~\ref{lem:L3}$(g_3)$, and \eqref{eq:s5}, we deduce
    \begin{align}
       \lambda \int_\Omega F(\xi,g(v))d\xi & \leq  \lambda \epsilon C_2 \int_\Omega \left|g(v)\right|^{2 \alpha Q}d\xi + C_3 \int_\Omega \left|g(v)\right|^{q}\exp\left({\beta \left|g(v)\right|^{\frac{2\alpha Q}{Q-1}}}\right)d\xi \notag\\
       &\leq \lambda \epsilon C_2 \|v\|^{2 \alpha Q}_{L^{2 \alpha Q}(\Omega)} + \lambda C_3 \|v\|^{q}_{L^{2 q}(\Omega)} \left(\int_\Omega\exp\left({2 \beta (2\alpha)^{\frac{1}{Q-1}} \|v\|^{\frac{Q}{Q-1}}\left(\frac{\left|v\right|}{\|v\|}\right)^{\frac{Q}{Q-1}}}\right)d\xi\right)^\frac{1}{2}. \label{eq:s24}
    \end{align}
     If we choose $v$ such that $\|v\|$ is sufficiently small so that $2 \beta (2\alpha)^{\frac{1}{Q-1}}\|v\|^{\frac{Q}{Q-1}} \le \alpha_Q$, then, using Moser-Trudinger inequality \ref{prop:A1} and Sobolev embedding in \eqref{eq:s24}, we obtain 
     \begin{align*}
    \lambda \int_\Omega F(\xi,g(v))d\xi \leq  \lambda C \|v\|^{2 \alpha Q} + \lambda C^{\prime} \|v\|^{q}.
     \end{align*} 
     Thus, we have 
     \begin{align*}
     \I(v) \geq \frac{1}{Q}\|v\|^Q- \lambda C \|v\|^{2 \alpha Q}-\lambda C^{\prime} \|v\|^{q}. 
     \end{align*} 
     Choosing $q>Q$ and recalling that $2\alpha Q>Q$ because $\alpha>\frac12$, we can conclude that there exists $\rho>0$ (small enough) such that $\I(v) \geq R>0$ for all $\|v\|=\rho$ and some $R>0$ (both depending on $\lambda$).
     
      $(ii)$ By $(f_4),$ there exist positive constants $d_1, d_2,$ and $\mu>2\alpha Q$ such that 
     \begin{equation}{\label{eq:s25}} 
     F(\xi,t) \geq d_1 \left|t\right|^\mu -d_2 \text{ for all } (\xi,t) \in \Omega \times (0,\infty).
     \end{equation} 
    Set $\Omega^{+}:=\{\xi\in\Omega: v(\xi)>0\}$, which has positive measure by hypothesis, and $\Omega^{+}_{t}:=\{\xi\in\Omega^{+}: t v(\xi)\ge1\}$. Since $\chi_{\Omega^{+}_{t}}\uparrow\chi_{\Omega^{+}}$ as $t\to\infty$, monotone convergence gives $\int_{\Omega^{+}_{t}}\left|v\right|^{\frac{\mu}{2\alpha}}d\xi\to\int_{\Omega^{+}}\left|v\right|^{\frac{\mu}{2\alpha}}d\xi=:\kappa>0$.
    Taking into account Lemma~\ref{lem:L3}-$(g_6)$, \eqref{eq:s25}, and $\mu> 2\alpha Q,$ for large $t>1,$ we deduce
    \begin{align*}
        \I(tv)&= \frac{t^Q}{Q}\|v\|^Q-\lambda \int_\Om F(\xi, g(tv))d\xi \\
        &= \frac{t^Q}{Q}\|v\|^Q-\lambda \int_{\Omega^{+}} F(\xi, g(tv))d\xi\\
        &\le \frac{t^Q}{Q}\|v\|^Q- \lambda d_1 \int_{\Omega^{+}_{t}} \left|g(tv)\right|^\mu d\xi + \lambda d_2 \left|\Omega\right| \\
        & \leq \frac{t^Q}{Q}\|v\|^Q -\lambda d_1 g(1)^\mu t^{\frac{\mu}{2\alpha}} \int_{\Omega^{+}_{t}} \left|v\right|^{\frac{\mu}{2\alpha}}d\xi + \lambda d_2 \left|\Omega\right| \to -\infty \text{ as } t \to \infty,
    \end{align*} 
    since $\frac{\mu}{2\alpha}>Q$ and the last integral tends to $\kappa>0$. This completes the proof. 
\end{proof}
\begin{rem}
By the last two lemmas, $\I$ satisfies the mountain pass geometry at the origin. Define the mountain pass critical level 
\begin{equation*}
c_\lambda= \inf_{\gamma \in \Gamma} \max_{t \in [0,1]} \I(\gamma(t)),
\end{equation*} 
where $\Gamma=\left\{\gamma \in C([0,1], \W): \gamma(0)=0, \I(\gamma(1))<0 \right\}.$ By Lemma~\ref{lem:L8}$(ii)$, applied to any $v\in\W$ with $v^{+}\not\equiv0$, the set $\Gamma$ is nonempty. Taking into account Lemma~\ref{lem:L8} together with the mountain pass theorem, there exists a Palais-Smale sequence $\{v_n\} \subset \W$ for $\I$ at level $c_\lambda$, that is, $\I(v_n) \to c_\lambda$ and $\I^{\prime}(v_n) \to 0$ in $\left(\W\right)^*.$ Moreover, Lemma~\ref{lem:L8} guarantees that $c_\lambda\ge R>0$ , since every $\gamma\in\Gamma$ meets the sphere $\|v\|=\rho$.
\end{rem}

\begin{lemma}{\label{lem:L9}}
    Let $\lambda_1$ be the constant defined in \eqref{eq:lambda1}. Then, for
    every $\lambda>\lambda_1$,
    $c_\lambda < \frac{1}{2\alpha Q}\left(\frac{\alpha_Q}{\beta_0}\right)^{Q-1}.$
\end{lemma}
\begin{proof}
     It is enough to show that 
     \begin{align*}
     \max_{t \geq 0} \I(tw) < \frac{1}{2\alpha Q}\left(\frac{\alpha_Q}{\beta_0}\right)^{Q-1}, \text{ for some } w \in \W.
     \end{align*} 
     Let us consider the sequence of Moser functions $\{\mathcal{W}_n\}$ defined by 
     \begin{align*}
     \M(\xi)=\frac{1}{\sigma_Q^{\frac{1}{Q}}}\begin{cases} (\log n)^{\frac{Q-1}{Q}},& \text{ if } \rho(\xi)\leq \frac{r}{n}, \\  \displaystyle{\frac{\log \frac{r}{\rho(\xi)}}{(\log n)^{\frac{1}{Q}}}},& \text{ if } \frac{r}{n} \leq \rho(\xi) \leq r, \\ 0 ,& \text{ if } \rho(\xi) \geq r,
    \end{cases}
    \end{align*} 
    where $\rho$ is the homogeneous norm on $\H$ and $\sigma_Q$ is the constant defined in \eqref{eq:constants}.
    The function $\mathcal{W}_n$ is continuous on $\H$, vanishes for $\rho(\xi)\ge r$, and is
smooth on each of the two open regions $\{\rho<r/n\}$ and $\{r/n<\rho<r\}$, where it is
constant and a smooth function of $\rho$ respectively. Since $\rho$ is smooth away from the
origin and $\left|\nah\rho\right|\le1$, $\mathcal{W}_n$ is Lipschitz with respect to the
Carnot-Carath\'eodory distance, with support contained in the closed gauge ball
$\overline{B_r(0)}\subset\Omega$; hence $\mathcal{W}_n\in\W$. Its horizontal gradient vanishes a.e.\
outside the annulus $\{r/n<\rho<r\}$, where
$\nah w_n=-\sigma_Q^{-1/Q}(\log n)^{-1/Q}\rho^{-1}\nah\rho$, so that by the polar coordinate
formula \eqref{eq:polar} and the $\delta_\theta$-homogeneity of degree $0$ of
$\left|\nah\rho\right|$,
\begin{align*}
\|\mathcal{W}_n\|^{Q}=\frac{1}{\sigma_Q\log n}\int_{\frac rn\le\rho\le r}
\frac{\left|\nah\rho\right|^{Q}}{\rho^{Q}}\,d\xi
=\frac{1}{\sigma_Q\log n}\int_{\frac rn}^{r}\frac{ds}{s}
\int_{\Sigma}\left|\nah\rho\right|^{Q}d\varsigma=1 .
\end{align*}
    \textbf{Assertion.} There exists $n\in \mathbb{N}$ such that
    \begin{equation*}
    \max_{t \geq 0} \I(t \mathcal{W}_n)< \frac{1}{2\alpha Q}\left(\frac{\alpha_Q}{\beta_0}\right)^{Q-1}.
    \end{equation*} 
    Let, if possible, for all $n\in\mathbb{N},$ there exists $t_n>0$ such that 
    \begin{equation}{\label{eq:s27}}
    \max_{t \geq 0} \I(t \mathcal{W}_n)= \I(t_n \M) \geq\frac{1}{2\alpha Q}\left(\frac{\alpha_Q}{\beta_0}\right)^{Q-1}.
    \end{equation} 
    Using $(f_1)$ and \eqref{eq:s27}, we deduce that 
    \begin{align*}
    \frac{t_n^Q}{Q}\|\M\|^Q= \I(t_n\M)+ \lambda \int_\Omega F(\xi, g(t_n \M))d\xi \geq   \I(t_n\M) \geq \frac{1}{2\alpha Q}\left(\frac{\alpha_Q}{\beta_0}\right)^{Q-1},
    \end{align*} 
    it implies 
    \begin{equation}{\label{eq:s28}} 
    t_n^Q \geq \frac{1}{2\alpha}\left(\frac{\alpha_Q}{\beta_0}\right)^{Q-1}.        
    \end{equation}
    By \eqref{eq:s27}, $\frac{d}{dt} \I(t\M)|_{t=t_n}=0$ that is, $t_n^Q \|\M\|^Q= \lambda \int_\Omega f(\xi,g(t_n\M))g^{\prime}(t_n\M)t_n\M d\xi.$ Also, using Lemma~\ref{lem:L3}$(g_4)$ in the form $s\,g'(s)\ge\frac{1}{2\alpha}g(s)$ for $s>0$, the nonnegativity of the integrand granted by $(f_1)$, and $\|\M\|=1,$ we have 
    \begin{equation}{\label{eq:s29}}
        t_n^Q \geq \frac{\lambda}{2\alpha} \int_{B_{\frac{r}{n}}(0)} f(\xi,g(t_n\M))g(t_n\M)d\xi.
    \end{equation}
    We observe that $t_n \M \to \infty$ as $n \to \infty$ uniformly on $B_{\frac{r}{n}}(0)$, by \eqref{eq:s28} and the definition of $\M$. Now by Lemma~\ref{lem:L3}$(g_6)$, we get $g(t_n \M) \to \infty$ as $n \to \infty,$ uniformly in $B_{\frac{r}{n}}(0)$. So, we can choose $n_\epsilon \in \mathbb{N}$ such that $g(t_n \M) \geq R_\epsilon,$ for all $n \geq n_\epsilon$ in $B_{\frac{r}{n}}(0).$ Then, given $\epsilon>0,$ there exists $n_0 \in \mathbb{N}$ such that for all $n \geq n_0$, 
    \begin{equation}{\label{eq:s31}}
        g(t_n \M) f(\xi,g(t_n \M)) \geq (\alpha_0-\epsilon)\exp\left(\beta_0 \left| g(t_n \M) \right|^{\frac{2\alpha Q}{Q-1}}\right). 
    \end{equation}
    Furthermore, by Lemma~\ref{lem:L3}$(g_5)$, for any $\eta>0$, there exists $n_1 \in \mathbb{N}$ such that for all $n \geq n_1$, 
    \begin{equation}{\label{eq:s32}}
        \left|g(t_n \M)\right|^{\frac{2\alpha Q}{Q-1}} \geq \left((2\alpha)^{\frac{1}{Q-1}}-\eta\right)\left|t_n \M\right|^{\frac{Q}{Q-1}} .
    \end{equation}
   Here $n_1=n_1(\eta)$; more precisely, \eqref{eq:s32} holds wherever $t_n\M\ge S_\eta$ for a threshold $S_\eta$ depending only on $\eta$, and we enlarge $R_\epsilon$ so that $R_\epsilon\ge S_\eta$.
    Using \eqref{eq:s31} and \eqref{eq:s32} in \eqref{eq:s29} for sufficiently large $n$, we get
    \begin{align*}
        t_n^Q &\geq \frac{\lambda (\alpha_0-\epsilon)}{2\alpha} \int_{B_{\frac{r}{n}}(0)} \exp\left(\beta_0\left((2\alpha)^\frac{1}{Q-1}-\eta\right) (t_n \M)^{\frac{Q}{Q-1}}\right)d\xi \\
        & = \frac{\lambda (\alpha_0-\epsilon)}{2\alpha} \int_{B_{\frac{r}{n}}(0)} \exp\left(\beta_0\left((2\alpha)^\frac{1}{Q-1}-\eta\right) \frac{Q}{\alpha_Q}t_n^{\frac{Q}{Q-1}}\log n\right)d\xi\\
        & =\frac{\lambda (\alpha_0-\epsilon)}{2\alpha}\frac{\omega_{Q-1}}{Q}\left(\frac{r}{n}\right)^Q \exp\left(\beta_0\left((2\alpha)^\frac{1}{Q-1}-\eta\right) \frac{Q}{\alpha_Q}t_n^{\frac{Q}{Q-1}}\log n\right),
    \end{align*}
    or, 
    \begin{align*}
    1 \geq \frac{\lambda}{2\alpha}(\alpha_0-\epsilon)\frac{\omega_{Q-1}}{Q}r^Q \exp\left(\left(\beta_0\left((2\alpha)^\frac{1}{Q-1}-\eta\right) \frac{1}{\alpha_Q}t_n^{\frac{Q}{Q-1}}-1\right)Q \log n - Q\log t_n\right),
    \end{align*}  
    it implies that $\{t_n\}$ is bounded. Indeed, if $\frac{A}{Q}t_n^{Q'}>1$ along a subsequence then the right-hand side tends to $+\infty$, a contradiction; hence
    \begin{align*}
    \limsup_{n\to\infty}\;\beta_0\left((2\alpha)^{\frac{1}{Q-1}}-\eta\right)\frac{t_n^{\frac{Q}{Q-1}}}{\alpha_Q}\le1 ,
    \end{align*}
    and letting $\eta\to0^{+}$ gives $\limsup_n t_n^{Q}\le\frac{1}{2\alpha}\left(\frac{\alpha_Q}{\beta_0}\right)^{Q-1}$. Moreover, using \eqref{eq:s28}, we obtain $t_n^Q \to \frac{1}{2\alpha}\left(\frac{\alpha_Q}{\beta_0}\right)^{Q-1}$ as $n \to \infty.$
    
    Let $ C_n=\{\xi \in B_r: t_n \M(\xi) \geq R_\epsilon\}$ and $D_n= \{\xi \in B_r: t_n \M(\xi)< R_\epsilon\}= B_r \setminus C_n.$ Using Lemma~\ref{lem:L3}$(g_4),$ we deduce 
    \begin{align}
        t_n^Q &\geq \frac{\lambda}{2\alpha} \int_\Omega f(\xi,g(t_n \M))g(t_n \M) d\xi \notag \\& \geq\frac{\lambda (\alpha_0-\epsilon)}{2\alpha} \int_{B_r} \exp\left(\beta_0\left((2\alpha)^\frac{1}{Q-1}-\eta\right) (t_n \M)^{\frac{Q}{Q-1}}\right)d\xi +\frac{\lambda}{2\alpha} \int_{D_n} f(\xi,g(t_n \M))g(t_n \M) d\xi \notag \\& \quad -\frac{\lambda (\alpha_0-\epsilon)}{2\alpha} \int_{D_n} \exp\left(\beta_0\left((2\alpha)^\frac{1}{Q-1}-\eta\right) (t_n \M)^{\frac{Q}{Q-1}}\right)d\xi. \label{eq:s33}
    \end{align}
    Taking into account $\M(\xi) \to 0$ and $\chi_{D_n}(\xi) \to 1$ almost everywhere in $B_r$ as $n\to \infty$, and noting that on $D_n$ both integrands are bounded by constants depending only on $R_\epsilon$, the Lebesgue dominated convergence theorem gives us 
    \begin{align} 
    &\lim_{n\to \infty} \int_{D_n} f(\xi,g(t_n \M))g(t_n \M) d\xi =0, \notag \\
    & \lim_{n\to \infty} \int_{D_n} \exp\left(\beta_0\left((2\alpha)^\frac{1}{Q-1}-\eta\right) (t_n \M)^{\frac{Q}{Q-1}}\right)d\xi=\frac{\omega_{Q-1}}{Q}r^Q.{\label{eq:s34}} 
    \end{align} 
    By \eqref{eq:s28}, we also have, 
    \begin{small}
    \begin{align}
        &\int_{B_r} \exp\left(\beta_0\left((2\alpha)^\frac{1}{Q-1}-\eta\right) (t_n \M)^{\frac{Q}{Q-1}}\right)d\xi \notag \\& \geq \int_{B_r} \exp\left(\left((2\alpha)^\frac{1}{Q-1}-\eta\right) \frac{\alpha_Q \M^{\frac{Q}{Q-1}}}{(2\alpha)^{\frac{1}{Q-1}}}\right)d\xi \notag \\&= \int_{\rho(\xi) \leq \frac{r}{n}} \exp\left(\left((2\alpha)^\frac{1}{Q-1}-\eta\right) \frac{\alpha_Q \M^{\frac{Q}{Q-1}}}{(2\alpha)^{\frac{1}{Q-1}}}\right)d\xi + \int_{\frac{r}{n} \leq \rho(\xi) \leq r} \exp\left(\left((2\alpha)^\frac{1}{Q-1}-\eta\right) \frac{\alpha_Q \M^{\frac{Q}{Q-1}}}{(2\alpha)^{\frac{1}{Q-1}}}\right)d\xi {\label{eq:s35}} \\&= \mathrm{I} + \mathrm{II}.\notag
    \end{align}
    \end{small}
    From the definition of $\M,$ we obtain 
    \begin{small}
    \begin{align}
        \mathrm{I}&= \int_{\rho(\xi) \leq \frac{r}{n}}   \exp\left(\left((2\alpha)^\frac{1}{Q-1}-\eta\right) \frac{Q \log n}{(2\alpha)^{\frac{1}{Q-1}}}\right)d\xi \notag \\&= \frac{\omega_{Q-1}}{Q} \left(\frac{r}{n}\right)^Q \exp \left(\left((2\alpha)^\frac{1}{Q-1}-\eta\right) \frac{Q \log n}{(2\alpha)^{\frac{1}{Q-1}}}\right) =\frac{\omega_{Q-1}}{Q}r^{Q}\,n^{-(1-c_\eta)Q}, {\label{eq:s36}} \\
    \mathrm{II} &=\int_{\frac{r}{n} \leq \rho(\xi) \leq r} \exp\left(\left((2\alpha)^\frac{1}{Q-1}-\eta\right) \frac{Q }{(2\alpha)^{\frac{1}{Q-1}}} \frac{\left(\log \frac{r}{\rho(\xi)}\right)^{\frac{Q}{Q-1}}}{(\log n)^{\frac{1}{Q-1}}}\right)d\xi \notag \\&= \omega_{Q-1}\int_{\frac{r}{n}}^r s^{Q-1}\exp\left(\left((2\alpha)^\frac{1}{Q-1}-\eta\right) \frac{Q \left(\log \frac{r}{s}\right)^{\frac{Q}{Q-1}} }{(2\alpha)^{\frac{1}{Q-1}} (\log n)^{\frac{1}{Q-1}}} \right)ds \notag \\& =\frac{\omega_{Q-1}}{Q} r^Q \int_0^1 Q \log n \exp\left(c_\eta\,Q \log n\left(\tau^\frac{Q}{Q-1}-\tau\right)\right)d\tau, {\label{eq:s37}}
    \end{align}
    \end{small}
    
    where in the last integral we have used the change of variable $\tau= \frac{\log (\frac{r}{s})}{\log n}.$ 
    
    We now let $n\to\infty$ with $\eta>0$ held fixed. By \eqref{eq:s36}, $\mathrm{I}\to0$. In \eqref{eq:s37}, put $a:=c_\eta Q\log n\to\infty$, so that the integral equals $\frac{1}{c_\eta}\int_0^1 a\,e^{a(\tau^{Q'}-\tau)}d\tau$; the mass concentrates at the two endpoints, contributing $1$ from $\tau\approx0$ and $Q-1$ from $\tau\approx1$, whence
    \begin{align}\label{eq:IIlimit}
    \lim_{n\to\infty}\mathrm{II}=\frac{\omega_{Q-1}}{Q}r^{Q}\,\frac{Q}{c_\eta}.
    \end{align}
    Combining \eqref{eq:s33}, \eqref{eq:s34}, \eqref{eq:s35}, \eqref{eq:s36} and \eqref{eq:IIlimit},
    \begin{align*}
    \lim_{n\to\infty}t_n^{Q}\;\ge\;\frac{\lambda (\alpha_0-\epsilon)}{2\alpha}\,\frac{\omega_{Q-1}}{Q}r^{Q}\left(\frac{Q}{c_\eta}-1\right).
    \end{align*}
    Letting now $\eta\to0^{+}$, so that $c_\eta\to1$, and then $\epsilon\to0^{+}$, and recalling $t_n^{Q}\to\frac{1}{2\alpha}\left(\frac{\alpha_Q}{\beta_0}\right)^{Q-1}$, we obtain
    \begin{align*}
    \frac{1}{2\alpha}\left(\frac{\alpha_Q}{\beta_0}\right)^{Q-1}\;\ge\;\frac{\alpha_0}{2\alpha}\,\frac{\omega_{Q-1}}{Q}r^{Q}(Q-1),
    \end{align*}
    that is,
    \begin{align*}
    \alpha_0\;\le\;\frac{Q}{(Q-1)\,\omega_{Q-1}\,r^{Q}}\left(\frac{\alpha_Q}{\beta_0}\right)^{Q-1},
    \end{align*}
    which contradicts $(f_6)$. Hence, the proof follows.
\end{proof}

\textbf{Proof of Theorem~\ref{T1}:} Fix $\lambda>\lambda_1$ and let $\{v_n\}$ be a $(PS)$ sequence at the level $c_\lambda.$ By Lemma~\ref{lem:L7}, the sequence $\{v_n\}$ is bounded in $\W$, so by reflexivity and the compactness of the embedding $\W\hookrightarrow L^{r}(\Omega)$, there exists $v\in \W$ such that, up to subsequence, $v_n \rightharpoonup v$ in $\W$, $v_n \to v$ in $L^r(\Omega), r\in [1,\infty)$ as $n\to \infty.$ Thanks to Lemma~\ref{lem:L5} and Lemma~\ref{lem:L6} $(i)$, $v$ is a weak solution of the problem~\eqref{eq:P2}. 
      We claim that $v \not\equiv 0.$ We suppose, by contradiction, that $v \equiv 0.$ Then, by Lemma~\ref{lem:L6} $(ii)$, we have 
      \begin{align*}
      \int_\Omega F(\xi, g(v_n))d\xi \to 0 \text{ as } n \to \infty, 
      \end{align*} 
      which gives that $\lim\limits_{n\to \infty} \|v_n\|^Q=c_\lambda Q.$ Thus, by Lemma~\ref{lem:L9}, we may fix $\epsilon>0$ so small that $c_\lambda Q+\epsilon<\frac{1}{2\alpha}\left(\frac{\alpha_Q}{\beta_0}\right)^{Q-1}$, and then for sufficiently large $n,$ $\|v_n\|^Q \leq c_\lambda Q +\epsilon <\frac{1}{2\alpha}\left(\frac{\alpha_Q}{\beta_0}\right)^{Q-1}$, so that $\beta_0 (2\alpha)^{\frac{1}{Q-1}} \|v_n\|^{\frac{Q}{Q-1}}< \alpha_Q.$ Choose $b>1$ (sufficiently close to $1$, with $b^{\prime}=b/(b-1)$) and $\beta >\beta_0$ (sufficiently close to $\beta_0$) such that $b \beta (2\alpha)^{\frac{1}{Q-1}} \|v_n\|^{\frac{Q}{Q-1}} \le \alpha_Q.$ Taking into account H\"older inequality, Trudinger-Moser inequality, Lemma~\ref{lem:L3}, and \eqref{eq:s4}, we deduce that
      \begin{align}
          \left|\lambda \int_\Omega f(\xi,g(v_n))g^{\prime}(v_n)v_nd\xi\right| &\leq \lambda \epsilon \|v_n\|^{2 \alpha Q}_{L^{2 \alpha Q}(\Omega)} + C \int_\Omega \left|v_n\right|^{q} \exp\left(\beta (2\alpha)^{\frac{1}{Q-1}} \left|v_n\right|^{\frac{Q}{Q-1}}\right)d\xi \notag \\& \leq \epsilon \|v_n\|^{2 \alpha Q}_{L^{2 \alpha Q}(\Omega)} + C \|v_n\|^{q}_{L^{b^{\prime}q}(\Omega)} \left(\int_\Omega \exp \left(b\beta (2\alpha)^{\frac{1}{Q-1}} \left|v_n\right|
          ^{\frac{Q}{Q-1}} \right)d\xi\right)^\frac{1}{b} \notag \\& \leq \epsilon \|v_n\|^{2 \alpha  Q}_{L^{2 \alpha Q}(\Omega)} + C_1 \|v_n\|^{q}_{L^{b^{\prime}q}(\Omega)} \to 0 \text{ as } n \to \infty.\notag 
      \end{align}
       Since $\{v_n\}$ is a (PS) sequence for $\I$, we have $\lim\limits_{n\to \infty} \langle \I^{\prime}(v_n),v_n \rangle=0.$ It implies that $\lim\limits_{n\to \infty} \|v_n\|^Q=0.$ Thus, we must have $\lim\limits_{n\to \infty} \I(v_n)=0=c_\lambda,$ which is a contradiction to the fact that $c_\lambda>0.$ Hence, $v$ is nontrivial.
       Next, we prove that $v>0$ in $\Omega.$ Since $v$ is a weak solution of \eqref{eq:P2}, we have 
       \begin{align*}
       \int_\Omega \left|\nah v \right|^{Q-2} \nah v.\nah w d\xi =  \lambda \int_\Omega f(\xi,g(v))g^{\prime}(v)w d\xi,   \text{ for all } w \in \W.
       \end{align*} 
       In particular, taking $w=v^-$ in the above equation, and observing that $f(\xi,g(v))=0$ on $\{v\le0\}$ by $(f_1)$ so that the right-hand side vanishes, we get $\|v^-\|=0$, that is, $v^-=0$ almost everywhere in $\Omega.$  
       Thus, we have $v\geq 0$ a.e. in $\Omega.$ 
Consequently, by $(f_1)$ and Lemma~\ref{lem:L3}$(g_2)$, the right-hand side
$\lambda f(\xi,g(v))g^{\prime}(v)$ is nonnegative a.e.\ in $\Omega$, so that for every
$\varphi\in C_0^\infty(\Omega)$ with $\varphi\ge0$,
\begin{align*}
\int_\Omega\left|\nah v\right|^{Q-2}\nah v\cdot\nah\varphi\,d\xi
=\lambda\int_\Omega f(\xi,g(v))g^{\prime}(v)\varphi\,d\xi\ \ge\ 0 ,
\end{align*}
that is, $v$ is a nonnegative weak supersolution of $\Delta_Qv=0$ in $\Omega$. \textcolor{red}{Suppose
$\essinf_{B_R}v=0$ for some gauge ball with $B_{4R}\subset\Omega$ and $R\le R_0$.} Then
\eqref{eq:wharnack} forces $v\equiv0$ on $B_{2R}$. Hence, the zero set of $v$ is open in
$\Omega$; it is closed by the continuity supplied by Proposition~\ref{prop:reg}; and $\Omega$
is connected with $v\not\equiv0$. Therefore $v>0$ in $\Omega$, and $u=g(v)>0$ is a positive
weak solution of \eqref{eq:P1}, since $g$ is increasing with $g(0)=0$.\qed
\section{Existence of Nodal Solutions}
In this section, we define the Nehari set to study the properties of the functional $I_\la$ constraint on the Nehari set and prove the existence of nodal solutions for \eqref{eq:P2}.  

Define 
\begin{align*}
\N:=\{v\in \W: \langle \I^{\prime}(v),v^{+}\rangle = \langle \I^{\prime}(v),v^{-}\rangle=0 \text{ and } v^{\pm} \not\equiv 0\}
\end{align*}
where $v^+(\xi)= \max\{v(\xi),0\}$, $v^-(\xi)=\min\{v(\xi),0\}$ and $v=v^+ + v^-.$ The set $\N$ is known as the Nehari set. By definition of $\I$, and since $\nah v^{+}$ and $\nah v^{-}$ have disjoint supports while $F(\xi,g(0))=0$, one can easily see that 
\begin{align}\label{eq:split}
\I(v)=\I(v^+) + \I(v^-), \quad \langle \I^{\prime}(v),v^{+}\rangle = \langle \I^{\prime}(v^+),v^{+}\rangle,   \text{ and } \langle \I^{\prime}(v),v^{-}\rangle = \langle \I^{\prime}(v^-),v^{-}\rangle.
\end{align}
 For any $v\in \W$ with $v^{\pm}\neq0,$ define $\Phi_v:[0,\infty)\times[0,\infty)\to\R$ by 
\begin{align*}
\Phi_v(r,s)=\I(rv^++sv^-)
\end{align*} 
and $\Psi_v:[0,\infty)\times[0,\infty)\to\R^2$ by
\begin{align*}
\Psi_v(r,s)=\left(\langle \I^{\prime}(rv^++sv^-),rv^+\rangle, \langle\I^{\prime}(rv^+ +sv^-),sv^-\rangle \right).
\end{align*} 
Also,
\begin{align*}
 \nabla\Phi_v(r,s)&=\left(\langle\I^{\prime}(r v^++s v^-), v^+ \rangle , \langle \I^{\prime}(r v^++s v^-),v^- \rangle \right)\\&= \left(\frac{1}{r}\langle\I^{\prime}(rv^++sv^-), rv^+ \rangle , \frac{1}{s}\langle \I^{\prime}(rv^++sv^-),sv^- \rangle \right).
\end{align*}
Thus, we deduce that $(r,s) \in(0,\infty)\times(0,\infty)$ is a critical point of $\Phi_v$ if and only if $rv^++sv^-\in N_{\lambda}.$
Moreover, by \eqref{eq:split},
\begin{align}\label{eq:decouple}
\langle \I^{\prime}(rv^++sv^-),rv^+\rangle=\langle \I^{\prime}(rv^+),rv^+\rangle,\quad
\langle \I^{\prime}(rv^++sv^-),sv^-\rangle=\langle \I^{\prime}(sv^-),sv^-\rangle,
\end{align}
so that the first component of $\Psi_v$ depends on $r$ alone and the second on $s$ alone.
The next lemma will establish the fact that $N_{\lambda} \ne\emptyset .$

\begin{lemma}{\label{lem:L10}}
   Assume that f satisfies $(\mathcal{A}_1)$ or $(\mathcal{A}_2)$ along with the conditions $(f_1^{\prime}), (f_2)-(f_5).$ Let $v\in \W \text{ such that } v^{\pm}\neq0.$ Then, 
   \begin{enumerate}[label=\textnormal{(\roman*)}]
       \item there exists a unique element $(r_v,s_v)\in (0,\infty)\times(0,\infty) \text{ such that } r_v v^++s_v v^-\in \N,$
       \item for all $r,s>0 \text{ with } (r,s)\neq (r_v,s_v),$ 
       \begin{align*}
       \I(rv^++sv^-)< \I(r_v v^++s_v v^-),
       \end{align*}
       \item if $\langle \I^{\prime}(v),v^{\pm} \rangle \leq 0$ then $0<r_v,s_v \leq1.$ 
   \end{enumerate}
   \end{lemma}
   \begin{proof}
      Using assumptions on $f$, there exist positive constants $C_1, \beta_0$ such that 
      \begin{align*}
      f(\xi,t)t\leq \epsilon\left|t\right|^{2\alpha Q}+C_1\left|t\right|^q \exp(\beta \left|t\right|^{\frac{2\alpha Q}{Q-1}}), \text{ for } \beta>\beta_0 \text{ and }q>Q.
      \end{align*}
     Composing with $g$ and using Lemma~\ref{lem:L3}$(g_3)$, which gives
      $\left|g(t)\right|^{2\alpha Q}\le 2\alpha\left|t\right|^{Q}$, $\left|g(t)\right|\le\left|t\right|$ and
      $\left|g(t)\right|^{\frac{2\alpha Q}{Q-1}}\le(2\alpha)^{\frac{1}{Q-1}}\left|t\right|^{\frac{Q}{Q-1}}$, we obtain
      \begin{align*}
      f(\xi,g(t))g(t)\le 2\alpha\epsilon\left|t\right|^{Q}+C_1\left|t\right|^{q}
      \exp\left(\beta(2\alpha)^{\frac{1}{Q-1}}\left|t\right|^{\frac{Q}{Q-1}}\right).
      \end{align*}
     
      Let $v\in\W \text{ with } v^{\pm}\neq 0.$
      Choose $t>1$ with $\frac{1}{t}+\frac{1}{t'}=1$ and sufficiently small $r>0$ such that 
      \begin{align*}
      \beta t \|rv^+\|^{Q'} \leq \frac{\alpha_Q}{(2\alpha)^{\frac{1}{Q-1}}}, \text{ where } Q'=\frac{Q}{Q-1}.
      \end{align*} 
      Taking into account Lemma~\ref{lem:L3} $(g_3)$, $(g_4)$, H\"older and Trudinger-Moser inequalities, we deduce
      \begin{small}
      \begin{align*}
        \langle \I^{\prime}(rv^+), rv^+ \rangle  
       &= \langle \I^{\prime}(rv^++sv^-), rv^+\rangle\\&= \|rv^+\|^Q-\lambda \int_\Omega f(\xi,g(rv^+))g^{\prime}(rv^+)rv^+ d\xi \\ & \geq \|rv^+\|^Q-  2\alpha \epsilon \lambda \int_\Omega \left|rv^+\right|^Qd\xi - C_1 \lambda \int_\Omega \left|rv^+\right|^q \exp(\beta (2\alpha)^\frac{1}{Q-1}\left|rv^+\right|^{Q'})d\xi \\ & \geq \|rv^+\|^Q-\epsilon \lambda C_2 \|rv^+\|^Q-C_1 \lambda \left(\int_\Omega \left|rv^+\right|^{qt'}d\xi\right)^\frac{1}{t'} \left(\int_\Omega \exp(\beta t(2\alpha)^\frac{1}{Q-1} \left|rv^+\right|^{Q'})d\xi\right)^\frac{1}{t} \\ & \geq (1-\epsilon \lambda C_2) \|rv^+\|^Q-C_3 \lambda \|rv^+\|^q. 
   \end{align*}
   \end{small}
  For $\epsilon$ small enough, $(1-\epsilon \lambda C_2)>0.$ Thus,  $\langle \I^{\prime}(rv^+ + sv^-),rv^+\rangle >0,$ for small $r>0$ and all $s>0.$ In the same manner, we can prove that $\langle \I^{\prime}(rv^+ + sv^-),sv^-\rangle >0,$  for small $s>0$ and all $r>0.$ So, there exists $\delta>0$ such that  
  \begin{align*}
  \langle \I^{\prime}(\delta v^+ + sv^-),\delta v^+\rangle >0 \text{ and } \langle \I^{\prime}(r v^+ + \delta v^-),\delta v^-\rangle >0, \forall r,s>0.
  \end{align*} 
  At the same time, it follows from $(f_4)$ that 
   \begin{equation}{\label{eq:n1}}
   F(\xi,t) \geq d_1 \left|t\right|^\mu-d_2, \text{ for some } d_1, d_2 >0.
   \end{equation}
  For $\rho>0$, put $\Omega^{+}_{\rho}:=\{\xi\in\Omega: v^{+}(\xi)>0,\ \rho v^{+}(\xi)\ge1\}$; since $v^{+}\not\equiv0$, monotone convergence gives $\int_{\Omega^{+}_{\rho}}\left|v^{+}\right|^{\frac{\mu}{2\alpha}}d\xi\to\kappa^{+}:=\int_{\{v^{+}>0\}}\left|v^{+}\right|^{\frac{\mu}{2\alpha}}d\xi>0$ as $\rho\to\infty$.
   Take $r=\delta_1>\delta$ with large enough $\delta_1$. Then, using Lemma~\ref{lem:L3} $(g_4)$, $(g_6)$ and \eqref{eq:n1}, we deduce
    \begin{align*}
       \langle \I^{\prime}(\delta_1v^++sv^-), \delta_1v^+\rangle &= \langle \I^{\prime}(\delta_1v^+),\delta_1v^+\rangle \\ &= \|\delta_1v^+\|^Q-\lambda \int_\Omega f(\xi,g(\delta_1v^+))g^{\prime}(\delta_1v^+)\delta_1v^+d\xi \\ &\leq \|\delta_1v^+\|^Q- \frac{\lambda \mu}{2\alpha} \int_\Omega F(\xi,g(\delta_1 v^+))d\xi   \\ &\leq \|\delta_1v^+\|^Q- \frac{\lambda \mu d_1}{2\alpha} \int_{\Omega^{+}_{\delta_1}} \left|g(\delta_1v^+)\right|^\mu d\xi + \frac{\lambda \mu}{2\alpha} d_2 \left|\Omega\right| \\ &\leq \|\delta_1v^+\|^Q- {\frac{\lambda \mu d_1 }{2\alpha}}g(1)^{\mu}\, \delta_1^{\frac{\mu}{2\alpha}} \int_{\Omega^{+}_{\delta_1}} \left|v^+ \right|^{\frac{\mu}{2\alpha}} d\xi+\frac{\lambda \mu}{2\alpha} d_2 \left|\Omega\right|   \\ &\leq 0, \text{ for all $s>0$, provided $\delta_1$ is large enough}, \text{ since }  \mu>2\alpha Q.
   \end{align*}
   Similarly, for $s=\delta_1 > \delta$ with large enough $\delta_1,$ 
   $\langle \I^{\prime}(rv^+ + \delta_1v^-),\delta_1 v^- \rangle \leq 0,$ for all $r>0$ and $\mu> 2\alpha Q.$ 
   Since, by \eqref{eq:decouple}, $r\mapsto\langle \I^{\prime}(rv^{+}),rv^{+}\rangle$ is a continuous real function which is positive at $r=\delta$ and non-positive at $r=\delta_1$, the intermediate value theorem provides $r_v\in(\delta,\delta_1]$ with $\langle \I^{\prime}(r_vv^{+}),r_vv^{+}\rangle=0$; likewise there is $s_v\in(\delta,\delta_1]$ with $\langle \I^{\prime}(s_vv^{-}),s_vv^{-}\rangle=0$. Hence $r_vv^++s_vv^-\in\N$.
   Now, we will prove the uniqueness of the point $(r_v,s_v).$ Precisely, we will show that if $v \in \N$ and $r_0v^++s_0v^- \in \N \text{ with } r_0,s_0>0 \text{ then }(r_0,s_0)=(1,1).$ Suppose that $v \in \N$ and $r_0v^++s_0v^- \in \N $. Then, by definition, 
   \begin{align*}
   \langle \I^{\prime}(r_0v^++s_0v^- ), r_0v^+\rangle=0,\ \langle \I^{\prime}(r_0v^++s_0v^- ), s_0v^-\rangle=0 \text{ and } \langle \I^{\prime}(v),v^{\pm}\rangle=0.
   \end{align*} 
   It gives us
   \begin{align}
       & \|r_0 v^+\|^Q=  \lambda \int_\Omega f(\xi,g(r_0v^+))g^{\prime}(r_0v^+)r_0v^+d\xi  \label{eq:n2}\\ & \|s_0v^-\|^Q= \lambda \int_\Omega f(\xi,g(s_0v^-))g^{\prime}(s_0v^-)s_0v^-d\xi \label{eq:n3}\\ & \|v^+\|^Q= \lambda \int_\Omega f(\xi,g(v^+))g^{\prime}(v^+)v^+d\xi \label{eq:n4} \\ & \|v^-\|^Q= \lambda \int_\Omega f(\xi,g(v^-))g^{\prime}(v^-)v^-d\xi \label{eq:n5}
   \end{align}
Using \eqref{eq:n2} and \eqref{eq:n4}, we get, 
\begin{equation}{\label{eq:n6}}
\lambda \int_\Omega f(\xi,g(v^+))g^{\prime}(v^+)v^+d\xi=\lambda \int_\Omega \frac{f(\xi,g(r_0v^+))g^{\prime}(r_0v^+)r_0v^+}{r_0^Q} d\xi.
\end{equation}
Introduce
\begin{align}\label{eq:Lambda}
\Lambda(\xi,t):=\frac{f(\xi,g(t))g^{\prime}(t)t}{t^{Q}}
=\frac{f(\xi,g(t))}{g(t)^{2\alpha Q-1}}\cdot\frac{g(t)^{2\alpha Q-1}g^{\prime}(t)}{t^{Q-1}},
\qquad t>0 .
\end{align}
The first factor is nondecreasing in $t$ by $(f_2)$, since $g$ is increasing, and the second is strictly increasing by Lemma~\ref{lem:L3}$(g_8)$ with $\gamma=2\alpha Q-1$; hence $t\mapsto\Lambda(\xi,t)$ is strictly increasing on $(0,\infty)$. Rewriting \eqref{eq:n6} as
\begin{align*}
\int_{\{v^{+}>0\}}\Lambda(\xi,v^{+})\left(v^{+}\right)^{Q}d\xi
=\int_{\{v^{+}>0\}}\Lambda(\xi,r_0v^{+})\left(v^{+}\right)^{Q}d\xi ,
\end{align*}
we see that $r_0>1$ forces the right-hand side to be strictly larger, and $r_0<1$ strictly smaller. Therefore $r_0=1$.
Similarly, with the help of equations \eqref{eq:n3} and \eqref{eq:n5}, it follows that $s_0 = 1.$ Hence $(i)$ follows.

To prove $(ii)$, we need to show that for $v\in \W \text{ with } v^{\pm}\neq0, (r_v,s_v)$ is the unique maximum point of $\Phi_v.$ Taking into account \eqref{eq:n1} and Lemma~\ref{lem:L3}$(g_6)$ for large $r$, $s$, we deduce that
\begin{align*}
  \Phi_v(r,s) &= \I(rv^+ + sv^-) \\ &= \frac{1}{Q}\|rv^++sv^-\|^Q-\lambda \int_\Omega F(\xi,g(rv^++sv^-))d\xi \\ & = \frac{r^Q}{Q}\|v^+\|^Q+\frac{s^Q}{Q} \|v^-\|^Q -\lambda \int_\Omega F(\xi,g(rv^++sv^-))d\xi \\ & \leq \frac{r^Q}{Q}\|v^+\|^Q+\frac{s^Q}{Q} \|v^-\|^Q-d_1 \lambda g(1)^{\mu} \left(r^{\frac{\mu}{2\alpha}}\int_{\Omega^{+}_{r}} \left|v^+\right|^{\frac{\mu}{2\alpha}} d\xi + s^{\frac{\mu}{2\alpha}}\int_{\Omega^{-}_{s}} \left|v^-\right|^{\frac{\mu}{2\alpha}} d\xi\right) + \lambda d_2 \left|\Omega\right|. 
\end{align*}
   Since $ \mu>2\alpha Q$ and the two integrals tend to the positive constants $\kappa^{\pm}$, $\Phi_v(r,s) \to -\infty \text{ as } \left|(r,s)\right| \to \infty.$ Consequently $\Phi_v$ attains its maximum over $[0,\infty)\times[0,\infty)$. By $(i)$, $(r_v,s_v)$ is its only critical point in the open quadrant, so the maximum is attained either at $(r_v,s_v)$ or on the boundary of the quadrant. If possible, suppose that $(0,a) \text{ with } a>0$ is a maximum point of $\Phi_v$; then by the estimate established at the beginning of the proof of $(i)$, 
   \begin{align*}
   r\,(\Phi_v)_r^{\prime}(r,s)= \langle \I^{\prime}(rv^+),rv^+\rangle >0, \text{ for small } r>0.
   \end{align*} 
   It implies $\Phi_v$ is increasing with respect to $r$ if $r>0$ is small enough, which is a contradiction. Similarly, $\Phi_v$ can't attain its global maximum at $(a,0) \text{ with } a>0$ or at $(0,0)$, since $\Phi_v(0,0)=0<\Phi_v(r,s)$ for small $r,s>0$ by the same estimate. Then, we have, 
   \begin{align*}
   \I(r_vv^++s_vv^-)=\max_{r,s\geq 0} \I(rv^++sv^-),
   \end{align*} 
   and the maximum point is unique by $(i)$, and hence, $(ii)$ follows.
   
   In order to prove $(iii)$, we observe that 
   \begin{align}
   0&=\langle \I^{\prime}(r_vv^++s_vv^-), r_vv^+\rangle \notag \\& =  \|r_vv^+\|^Q-\lambda \int_\Omega f(\xi,g(r_vv^+))g^{\prime}(r_vv^+)r_vv^+d\xi \label{eq:n7}
   \end{align}
  Now, using the assumption $\langle \I^{\prime}(v),v^{\pm}\rangle \leq 0$, we have
   \begin{equation}{\label{eq:n8}}
   \|v^+\|^Q \leq \lambda \int_\Omega f(\xi,g(v^+))g^{\prime}(v^+)v^+d\xi.
   \end{equation}
   Combining \eqref{eq:n7} and \eqref{eq:n8}, we deduce that 
   \begin{equation}{\label{eq:n9}}
       \lambda \int_\Omega \frac{f(\xi,g(r_vv^+))g^{\prime}(r_vv^+)r_vv^+}{r_v^Q} d\xi \leq \lambda \int_\Omega f(\xi,g(v^+))g^{\prime}(v^+)v^+d\xi,
   \end{equation}
  that is, $\int\Lambda(\xi,r_vv^{+})(v^{+})^{Q}\le\int\Lambda(\xi,v^{+})(v^{+})^{Q}$.
  If possible, assume $r_v >1.$ Then, by the strict monotonicity of $\Lambda$ established in $(i)$, the inequality \eqref{eq:n9} is not possible. Thus, we must have $0<r_v \leq 1$. Similarly, we also conclude that $0<s_v \leq 1$. 
   \end{proof}
\begin{lemma}[Fibering inequality]\label{lem:fiber}
Assume $(f_1^{\prime})$, $(f_2)$ and let $g$ be as in Lemma~\ref{lem:L3}. Then for every $w\in\W$ with $w\not\equiv0$ and every $s>0$,
\begin{align}\label{eq:fiber}
\I(sw)\;\le\;\I(w)+\frac{1-s^{Q}}{Q}\,\langle \I^{\prime}(w),w\rangle ,
\end{align}
with strict inequality whenever $s\neq1$.
\end{lemma}
\begin{proof}
Extend the function $\Lambda$ of \eqref{eq:Lambda} to $\R\setminus\{0\}$ by
\begin{align*}
\Lambda(\xi,t):=\frac{f(\xi,g(t))g^{\prime}(t)t}{\left|t\right|^{Q}} ,
\end{align*}
so that, by $(f_2)$ and Lemma~\ref{lem:L3}$(g_8)$, the map $t\mapsto\Lambda(\xi,t)$ is strictly increasing in $\left|t\right|$ on each of $(0,\infty)$ and $(-\infty,0)$. Define
\begin{align*}
D(s):=\I(w)-\I(sw)-\frac{1-s^{Q}}{Q}\,\langle \I^{\prime}(w),w\rangle ,\qquad s>0,
\end{align*}
so that $D(1)=0$. Since $w$ is a fixed element of $\W$, Proposition~\ref{prop:A1}$(ii)$ together with \eqref{eq:s4} permits differentiation under the integral sign, and
\begin{align*}
D^{\prime}(s)=-\langle \I^{\prime}(sw),w\rangle+s^{Q-1}\langle \I^{\prime}(w),w\rangle .
\end{align*}
The two norm contributions cancel, since both equal $s^{Q-1}\|w\|^{Q}$, and therefore
\begin{align*}
D^{\prime}(s)&=\lambda\int_\Omega f(\xi,g(sw))g^{\prime}(sw)w\,d\xi
-\lambda s^{Q-1}\int_\Omega f(\xi,g(w))g^{\prime}(w)w\,d\xi\\
&=\lambda\,s^{Q-1}\int_{\{w\neq0\}}\Big[\Lambda(\xi,sw)-\Lambda(\xi,w)\Big]\left|w\right|^{Q}d\xi ,
\end{align*}
where the last equality uses $f(\xi,g(sw))g^{\prime}(sw)w=\Lambda(\xi,sw)\,s^{Q-1}\left|w\right|^{Q}$
on $\{w\neq0\}$ and the fact that both integrands vanish on $\{w=0\}$. By the strict monotonicity of $\Lambda$, the bracket is negative a.e.\ on $\{w\neq0\}$ when $0<s<1$ and positive when $s>1$.
Hence $D$ is strictly decreasing on $(0,1)$ and strictly increasing on $(1,\infty)$, and since $D(1)=0$, we conclude $D(s)>0$ for every $s\neq1$, which is \eqref{eq:fiber}.
\end{proof}
   \begin{lemma}\label{lem:L11}
       For any $v\in \N,$ the following holds.
       \begin{enumerate}[label=\textnormal{(\roman*)}]
           \item $\|v^{\pm}\| \geq R, \text{ for some } R>0$ independent of $v$.
           \item $\I(v) \geq \left(\frac{1}{Q}-\frac{2\alpha}{\mu}\right) \|v\|^Q.$
       \end{enumerate}
       \end{lemma}
    
          \begin{proof} $(i)$ Suppose, by contradiction, that there exists a sequence $\{v_n\} \subset \N $ such that $v_n^+ \to 0 $ in $\W$, then by definition of $\N,$ $\langle \I^{\prime}(v_n^+),v_n^+\rangle=0,$ which implies 
       \begin{align*} 
          \|v_n^+\|^Q &= \lambda \int_\Omega f(\xi,g(v_n^+))g^{\prime}(v_n^+)v_n^+d\xi \\ & \le \lambda \int_\Omega f(\xi,g(v_n^+))g(v_n^+)d\xi \\ & \leq \lambda 2\alpha\epsilon \int_\Omega \left|v_n^+\right|^Qd\xi + \lambda C_1 \int_\Omega \left|v_n^+\right|^q \exp(\beta (2\alpha)^{\frac{1}{Q-1}} \left|v_n^+\right|^{Q'} )d\xi \\ & \le \lambda \epsilon C_2 \|v_n^+\|^Q + \lambda C_1 \left(\int_\Omega  \left|v_n^+\right|^{qt'} d\xi\right)^\frac{1}{t'} \left(\int_\Omega \exp(\beta t (2\alpha)^{\frac{1}{Q-1}} \left|v_n^+\right|^{Q'}) d\xi\right)^\frac{1}{t},
        \end{align*}
        where $t,t'>1$. Since $v_n^+ \to 0 \text{ in } \W$, there exists $m \in \mathbb{N}$ such that 
        \begin{align*}
        \|v_n^+\|^{Q'} \leq \frac{\alpha_Q}{\beta t (2\alpha)^\frac{1}{Q-1}}, \text{ for all } n\geq m.
        \end{align*} 
        It follows that 
        \begin{align*}
        \int_\Omega \exp(\beta t (2\alpha)^{\frac{1}{Q-1}} \left|v_n^+\right|^{Q'}) d\xi= \int_\Omega \exp\left(\beta t (2\alpha)^{\frac{1}{Q-1}}  \|v_n^+\|^{Q'} \left(\frac{\left|v_n^+\right|}{ \|v_n^+\|}\right)^{Q'}\right) d\xi \leq C_2.
        \end{align*} 
        Therefore,  
       \begin{align*}
           \|v_n^+\|^Q & \leq \lambda \epsilon C_2 \|v_n^+\|^Q + \lambda C_1 C_3 \|v_n^+\|_{L^{qt'}(\Omega)}^q \leq \lambda \epsilon C_2 \|v_n^+\|^Q + \lambda C_1 C_4 \|v_n^+\|^q. 
       \end{align*}
      This in turn gives 
       \begin{equation*}
       (1-\lambda \epsilon C_2) \leq \lambda C_1 C_4  \|v_n^+\|^{q-Q}. 
       \end{equation*}
       Choosing $\epsilon >0$ small enough so that $1-\lambda\epsilon C_2>0$, and recalling $q>Q$, the right-hand side tends to $0$ while the left-hand side is a fixed positive constant -- a contradiction to $v_n^{+}\to 0$. So, $(i)$ holds.
       
       $(ii)$ Using the assumption $(f_4)$, Lemma~\ref{lem:L3}$(g_4)$ and $v\in \N$, we obtain
       \begin{align*}
           \I(v^+) &= \I(v^+)- \frac{2\alpha}{\mu} \langle \I^{\prime}(v^+),v^+\rangle \\ &= \left(\frac{1}{Q}-\frac{2\alpha}{\mu}\right)\|v^+\|^Q + \frac{\lambda}{\mu}\int_\Omega \left(2\alpha g^{\prime}(v^+)v^+f(\xi,g(v^+))-\mu F(\xi,g(v^+)\right)d\xi \\ & \geq \left(\frac{1}{Q}-\frac{2\alpha}{\mu}\right)\|v^+\|^Q + \frac{\lambda}{\mu}\int_\Omega \left( f(\xi,g(v^+)) g(v^+)-\mu F(\xi,g(v^+)\right)d\xi \\ & \geq \left(\frac{1}{Q}-\frac{2\alpha}{\mu}\right)\|v^+\|^Q. 
       \end{align*}
      Similarly, $\I(v^-) \geq \left(\frac{1}{Q}-\frac{2\alpha}{\mu}\right)\|v^-\|^Q.$ Since for any $v\in \N,$ $\I(v)=\I(v^+)+\I(v^-),$ we deduce
       \begin{align*}
           \I(v) \geq \left(\frac{1}{Q}-\frac{2\alpha}{\mu}\right)\|v^+\|^Q + \left(\frac{1}{Q}-\frac{2\alpha}{\mu}\right)\|v^-\|^Q =  \left(\frac{1}{Q}-\frac{2\alpha}{\mu}\right) \|v\|^Q.
       \end{align*}
       Hence, the proof follows. \qed 
       \end{proof}
        Define $m_\lambda := \displaystyle \inf_{v\in\N} \I(v).$ As a consequence of Lemma~\ref{lem:L11}, $m_\lambda$ is well defined and $m_\lambda\ge0$. In the next lemma, we deal with the asymptotic behaviour of $m_\lambda.$ 
   \begin{lemma}{\label{lem:L12}}
       $m_\lambda\to0 \text{ as } \lambda\to\infty.$ 
   \end{lemma}
   \begin{proof}
       Fix $v\in \W$ with $ v^{\pm} \neq 0.$ Then, by Lemma~\ref{lem:L10}, there exists $(r_\lambda,s_\lambda) \in (0,\infty) \times (0,\infty)$ such that $r_\lambda v^+ + s_\lambda v^- \in \N, \text{ for each } \lambda >0.$ Define 
       \begin{align*}
       T_v= \{(r_\lambda,s_\lambda) \in (0,\infty) \times (0,\infty):\Psi_v(r_\lambda,s_\lambda)=(0,0)\}.
       \end{align*} 
       With the help of $(f_5)$ and Lemma~\ref{lem:L3}$(g_6)$, there exists $C>0$ such that 
       \begin{align*}
       F(\xi,t) \ge C \exp\left(\frac{t}{M_0}\right), \ \xi \in \Omega, \ \forall t \ge \max\{R_0,g(1),1\}
       \end{align*} 
       and 
        \begin{equation}{\label{eq:n10}}
       F(\xi,g(t)) \ge C \frac{\left|g(t)\right|^k}{M_0^k k!}-C' \geq C \frac{(g(1))^k}{M_0^k k!} \left|t\right|^\frac{k}{2\alpha}-C'=C_1  \left|t\right|^\frac{k}{2\alpha}-C',    
       \end{equation}
        where $k>2\alpha Q$ and the second inequality holds on $\{\left|t\right|\ge1\}$.
        Resuming all information from Lemma~\ref{lem:L10}, Lemma~\ref{lem:L3}$(g_4), (g_6),$ and \eqref{eq:n10}, we deduce that
        \begin{align*}
           r_\lambda^Q \|v^+\|^Q + s_\lambda^Q \|v^-\|^Q &=\|r_\lambda v^+ + s_\lambda v^-\|^Q \\ &= \lambda \int_\Omega f(\xi,g(r_\lambda v^+ + s_\lambda v^-))g^{\prime}(r_\lambda v^+ + s_\lambda v^-) (r_\lambda v^+ + s_\lambda v^-)d\xi \\ & \geq \frac{\lambda}{2\alpha} \int_\Omega f(\xi,g(r_\lambda v^+ + s_\lambda v^-))g(r_\lambda v^+ + s_\lambda v^-)d\xi \\ & \geq \frac{\mu \lambda}{2\alpha} \int_\Omega F(\xi,g(r_\lambda v^+ + s_\lambda v^-))d\xi \\ & \geq  r_\lambda^\frac{k}{2\alpha} \frac{\mu \lambda}{2\alpha} C_1\int_{\Omega^{+}_{r_\lambda}} \left|v^+\right|^\frac{k}{2\alpha}d\xi + s_\lambda^\frac{k}{2\alpha} \frac{\mu \lambda}{2\alpha} C_1\int_{\Omega^{-}_{s_\lambda}} \left|v^-\right|^\frac{k}{2\alpha}d\xi - \frac{\mu \lambda}{2\alpha}C_2\left|\Omega\right|.
       \end{align*}
       Since $k>2\alpha Q$, it follows that $T_v$ is a bounded set. So, if $\{\lambda_n\}$ is a sequence in $(0,\infty)$ satisfying $\lambda_n \to \infty \text{ as } n \to \infty$, then, upto a subsequence, there are nonnegative numbers $r' , s'$ such that sequences $r_{\lambda_n} , s_{\lambda_n}$ converge to $r'$, $s'$ respectively. 
       Suppose $r'>0$ or $s'>0$. Since $r_{\lambda_n}v^{+}+s_{\lambda_n}v^{-}\to r'v^{+}+s'v^{-}$ in $\W$ and the map $w\mapsto\int_\Omega f(\xi,g(w))g^{\prime}(w)w\,d\xi$ is continuous on $\W$, the Nehari identity
       \begin{align*}
       \|r_{\lambda_n} v^+ + s_{\lambda_n} v^-\|^Q = \lambda_n \int_\Omega f(\xi,g( r_{\lambda_n} v^+ + s_{\lambda_n} v^-)) g^{\prime}( r_{\lambda_n} v^+ + s_{\lambda_n} v^-) ( r_{\lambda_n} v^+ + s_{\lambda_n} v^-)d\xi
       \end{align*}
       has a bounded left-hand side, while on the right the integral converges to
       $\int_\Omega f(\xi,g(r'v^{+}+s'v^{-}))g^{\prime}(r'v^{+}+s'v^{-})(r'v^{+}+s'v^{-})d\xi>0$
       by $(f_1^{\prime})$, so the right-hand side tends to $+\infty$. This contradiction forces $r'=s'=0$.
      Write $w_n:=r_{\lambda_n}v^{+}+s_{\lambda_n}v^{-}\in N_{\lambda_n}$, so that $w_n\to0$ in $\W$. Since $F(\xi,\cdot)\ge0$ by $(f_1^{\prime})$, we have
       \begin{align*}
       0\le m_{\lambda_n}\le I_{\lambda_n}(w_n)=\frac1Q\|w_n\|^{Q}-\lambda_n\int_\Omega F(\xi,g(w_n))\,d\xi\le\frac1Q\|w_n\|^{Q}\longrightarrow0 .
       \end{align*}
       As $\{\lambda_n\}$ was an arbitrary sequence tending to $\infty$, $m_\lambda\to0$ as $\lambda\to\infty$.  
   \end{proof}
   
  The argument of the next lemma follows the scheme of
   \cite[Lemma 3.4]{sunsong2022} for the semilinear problem, with two modifications. First, the
   nondegeneracy of the matrix $J$ is proved rather than assumed, and in the form $\det J>0$,
   which is what the degree computation requires. Second, the membership
   $\bar k(\bar a,\bar b)\in\N$ requires the two nodal parts of $\bar k(\bar a,\bar b)$ to be
   nontrivial, which we obtain from the displacement estimate of the quantitative deformation
   lemma.

   \begin{lemma}{\label{lem:L13}}
       Let $v_0 \in\N \text{ be such that }\I(v_0)=m_\lambda.$ Then, $\I^{\prime}(v_0)=0.$
   \end{lemma}
   \begin{proof}
      We assume, by contradiction, that $I_{\lambda}^{\prime}(v_0) \neq 0.$ Then, continuity of $\I^{\prime}$ implies that
      \begin{equation}{\label{eq:n11}}
      \|I_{\lambda}^{\prime}(v)\|_{\overset{\scriptscriptstyle\circ}{S}{}^Q_1(\Omega)^*} \geq \nu \quad \text{ whenever }  \|v-v_0\| \leq 3\delta , \text{ for some } \nu,\delta >0. 
      \end{equation}
      Shrinking $\delta$ if necessary -- which is permissible, since
      \eqref{eq:n11} holds a fortiori on a smaller ball -- we may and do assume in addition that
      \begin{align}\label{eq:deltasmall}
      C_S\,\delta<\tfrac12\min\left\{\|v_0^{+}\|_{L^{Q}(\Omega)},\ \|v_0^{-}\|_{L^{Q}(\Omega)}\right\},
      \end{align}
      where $C_S$ is the norm of the embedding $\W\hookrightarrow L^{Q}(\Omega)$.
      Let $ \gamma_0= \min\left\{\frac{1}{4}, \frac{\delta}{4\|v_0\|}\right\}.$ Pick $l \in (0,\gamma_0)$ and define a map $k:H \to \W$ by 
      \begin{align*}
      k(a,b)=av_0^+ + bv_0^-, \text{ for } (a,b) \in H,
      \end{align*} 
      where $H=(1-l, 1+l) \times (1-l, 1+l).$ Now, $v_0 \in N_{\lambda} \text{ with } I_{\lambda}(v_0)=m_{\lambda} $ and using Lemma~\ref{lem:L10}$(ii)$ together with the compactness of $\partial H$, we get
\begin{equation*}
m_{\lambda}^{\prime}:= \max_{\partial H} (I_{\lambda} \circ k) < m_{\lambda}.
\end{equation*}
Define $\epsilon:= \min \left\{\frac{m_{\lambda}-m_{\lambda}^{\prime}}{3}, \frac{\nu\de}{8}\right\}$, $S_{\de}:=B(v_0;\de)$ and $I_{\lambda}^{c}:=\{v\in \W: I_{\lambda}(v) \leq c\}= I_{\lambda}^{-1}((-\infty,c]).$ Applying quantitative deformation Lemma \cite{willem2012minimax}, there exists a deformation $ \eta \in C\left([0,1] \times k(H), \W\right)$ satisfying the following properties:
\begin{enumerate}[label=\Roman*.]          
          \item $\eta (1,v)=v \text{ whenever } v\notin I_{\lambda}^{-1}([m_{\lambda }-2\epsilon, m_{\lambda} +2\epsilon]) \cap S_{2\delta};$
          \item $ \eta (1, I_{\lambda}^{m_{\lambda} + \epsilon} \cap S_{\delta}) \subset I_{\lambda}^{m_{\lambda} - \epsilon} $;
          \item $ I_{\lambda}(\eta(1,v)) \leq I_{\lambda}(v) \text{ for all } v \text{ in } \W; $
          \item $\|\eta(t,v)-v\|\le\delta$ for all $v\in\W$ and all $t\in[0,1]$.
\end{enumerate}
      Since $\|k(a,b)-v_0 \|= \|(a-1)v_0^+ + (b-1)v_0^- \| \leq \left|a-1\right| \|v_0^+\| + \left|b-1\right| \|v_0^-\| \le 2l \|v_0\| < 2\gamma_0 \|v_0\| \le\frac{\delta}{2}$, we have 
      \begin{align*}
      I_{\lambda}(k(a,b)) \leq m_{\lambda} \text{ and } k(a,b) \in S_{\delta}, \text{ for all } (a,b)\in \overline{H}.
      \end{align*} 
      From II, we have
      \begin{equation}{\label{eq:n13}}
        \max_{\overline{H}} I_{\lambda}(\eta (1,k(a,b))) \leq m_{\lambda}- \epsilon.
       \end{equation}
        \textbf{Claim:} $\eta (1,k(H)) \cap N_{\lambda } \neq \emptyset. $ 
       
       Define 
       \begin{align*}
       \overline{k}(a,b):= \eta (1,k(a,b)),
       \end{align*}  
       \begin{align*}
           \Phi_{1}(a,b) & :=\left(\langle I_{\lambda}^{\prime}(k(a,b)),v_0^+ \rangle, \langle I_{\lambda}^{\prime}(k(a,b)),v_0^- \rangle \right) \\ &:= \left(\langle I_{\lambda}^{\prime}(av_0^+ + bv_0^-),v_0^+ \rangle, \langle I_{\lambda}^{\prime}(av_0^+ + bv_0^-),v_0^- \rangle \right) \\ &:= \left(\phi_1^1(a,b), \phi_1^2(a,b)\right),
       \end{align*}
       and 
       \begin{align*}
       \Phi_2(a,b):= \left(\frac{1}{a} \langle I_{\lambda}^{\prime}(\overline{k}(a,b)), (\overline{k}(a,b))^+\rangle , \frac{1}{b} \langle I_{\lambda}^{\prime}(\overline{k}(a,b)), (\overline{k}(a,b))^-\rangle \right).
       \end{align*} 
       We first observe that $\left(\bar k(a,b)\right)^{\pm}\not\equiv0$ for every
       $(a,b)\in\overline H$, so that $\Phi_2$ is well defined. Indeed, since $v_0^{+}\ge0$ and
       $v_0^{-}\le0$ have disjoint supports and $a>0$, we have $k(a,b)^{+}=av_0^{+}$; since
       $t\mapsto t^{+}$ is $1$-Lipschitz on $\R$, we have
       $\left|\left(\bar k(a,b)\right)^{+}-k(a,b)^{+}\right|\le\left|\bar k(a,b)-k(a,b)\right|$
       pointwise, whence by property IV and \eqref{eq:deltasmall},
       \begin{align*}
       \left\|\left(\bar k(a,b)\right)^{+}\right\|_{L^{Q}(\Omega)}
       \ \ge\ a\|v_0^{+}\|_{L^{Q}(\Omega)}-C_S\delta
       \ \ge\ \tfrac34\|v_0^{+}\|_{L^{Q}(\Omega)}-C_S\delta
       \ >\ \tfrac14\|v_0^{+}\|_{L^{Q}(\Omega)}>0 ,
       \end{align*}
       where we used $a\ge1-l>\frac34$, which holds because $l<\gamma_0\le\frac14$. The same
       estimate applies to $\left(\bar k(a,b)\right)^{-}$.
       Now, we compute the following terms directly: 
       \begin{align*}
      J_1&= \frac{\partial \phi_1^1(a,b)}{\partial a}\Big|_{(1,1)}\\ &=(Q-1) \|v_0^+\|^Q - \lambda \int_\Omega \left[f(\xi,g( v_0^+)) g^{\prime \prime}(v_0^+)+f^{\prime}(\xi,g(v_0^+))g^{\prime\, 2}(v_0^+)\right](v_0^+)^2d\xi \\&= \lambda (Q-1) \int_\Omega f(\xi,g(v_0^+))g^{\prime}(v_0^+)v_0^+ d\xi - \lambda \int_\Omega \left[f(\xi,g( v_0^+)) g^{\prime \prime}(v_0^+)+f^{\prime}(\xi,g(v_0^+))g^{\prime\, 2}(v_0^+)\right](v_0^+)^2d\xi. 
       \end{align*}
        \begin{align*}
       J_2&= \frac{\partial \phi_1^2(a,b)}{\partial b}\Big|_{(1,1)}\\ &= (Q-1) \|v_0^-\|^Q - \lambda \int_\Omega \left[f(\xi,g( v_0^-)) g^{\prime \prime}(v_0^-)+f^{\prime}(\xi,g(v_0^-))g^{\prime\, 2}(v_0^-)\right](v_0^-)^2d\xi \\& = \lambda (Q-1) \int_\Omega f(\xi,g(v_0^-))g^{\prime}(v_0^-)v_0^- d\xi - \lambda \int_\Omega \left[f(\xi,g( v_0^-)) g^{\prime \prime}(v_0^-)+f^{\prime}(\xi,g(v_0^-))g^{\prime\, 2}(v_0^-)\right](v_0^-)^2d\xi.
       \end{align*}
       $
      J_3= \frac{\partial \phi_1^1(a,b)}{\partial b}\Big|_{(1,1)}=0 \text{ and } J_4= \frac{\partial \phi_1^2(a,b)}{\partial a}\Big|_{(1,1)}=0$, by \eqref{eq:decouple}.
      
      Consider the matrix 
       \[
       J=\begin{bmatrix}
           J_1 & J_3 \\
          J_4 & J_2
       \end{bmatrix}
       \]
     We claim that $J_1<0$ and $J_2<0$. Indeed, by \eqref{eq:decouple} and \eqref{eq:Lambda},
      \begin{align*}
      \phi_1^1(a,b)=a^{Q-1}\left[\|v_0^{+}\|^{Q}-\lambda\int_{\Omega}\Lambda(\xi,av_0^{+})\left(v_0^{+}\right)^{Q}d\xi\right]
      =:a^{Q-1}\psi(a),
      \end{align*}
      where $\psi$ is strictly decreasing because $\Lambda(\xi,\cdot)$ is strictly increasing, and $\psi(1)=0$ because $v_0\in\N$. Hence
      \begin{align*}
      J_1=\frac{d}{da}\Big[a^{Q-1}\psi(a)\Big]_{a=1}=(Q-1)\psi(1)+\psi^{\prime}(1)=\psi^{\prime}(1)<0 ,
      \end{align*}
      and $J_2<0$, by the same argument applied to $v_0^{-}$. Consequently, $\det J=J_1J_2>0$.
     Moreover, by \eqref{eq:decouple}, $\phi_1^1(a,b)=a^{Q-1}\psi(a)$ depends on
      $a$ alone and $\phi_1^2(a,b)$ on $b$ alone; since $\psi$ is strictly decreasing with
      $\psi(1)=0$, the only zero of $\Phi_1$ in $H$ is $(1,1)$, and it is nondegenerate.
      So, $\Phi_1$ is a $C^1$ function with the point $(1,1)$ being the unique isolated zero point in $H$. With the use of Brouwer's degree theory, we deduce that $\deg(\Phi_1, H,0)=\operatorname{sgn}\det J=1$. Subsequently, property I gives $k(a,b)=\bar{k}(a,b)$ on $\partial H$.
      
      Again, using the boundary dependence property of Brouwer's degree \cite{motreanu2014topological}, we have 
      \begin{align*}
      \deg(\Phi_1, H,0)=\deg(\Phi_2, H,0)=1,
      \end{align*} 
      which implies that  \(\Phi_2(\bar{a},\bar{b})=0, \text{ for some } (\bar{a},\bar{b}) \in H.\)
      Since $(\bar a,\bar b)\in H$ we have $\bar a,\bar b>0$, and
      $\left(\bar k(\bar a,\bar b)\right)^{\pm}\not\equiv0$ by the observation made above; hence 
      \begin{align*}
      \eta(1, k(\bar{a}, \bar{b}))=\bar{k}(\bar{a}, \bar{b}) \in \N,
      \end{align*} 
      which shows that $\eta (1,k(H)) \cap \N \neq \emptyset.$ But then $I_\lambda(\bar k(\bar a,\bar b))\ge m_\lambda$, contradicting \eqref{eq:n13}.
     \end{proof}

   \begin{lemma}{\label{lem:L14}}
       Let $z$ be the least energy nodal solution of the problem \eqref{eq:P2} then $z$ has exactly two nodal domains.
   \end{lemma}
   \begin{proof}
    By Proposition~\ref{prop:reg}, $z\in C^{0,\gamma}_{loc}(\Omega)$, so that $\{z>0\}$ and $\{z<0\}$ are open and their connected components,  the nodal domains of $z$, are well defined.
    Assume, by contradiction, that $z=z_1+z_2+z_3$ with  the following properties:
       \begin{align*}
      &\qquad z_i \neq 0,\ i=1,2,3 , \quad z_1 \geq0,\ z_2 \leq0 \text{ a.e. in } \Omega, \\& \Omega_1 \cap \Omega_2=\emptyset,\ \Omega_1= \{\xi\in \Omega: z_1(\xi) > 0 \},\ \Omega_2=\{\xi\in \Omega: z_2(\xi) < 0 \},
       \end{align*} 
       \begin{equation}{\label{eq:n14}} z_1|_{\Omega \setminus {\Omega_1 \cup \Omega_2}}= z_2|_{\Omega \setminus {\Omega_1 \cup \Omega_2}}=z_3|_{\Omega_1 \cup \Omega_2}=0 \text{ and } 
       \langle \I^{\prime}(z),z_i\rangle =0, \text{ for } i=1,2,3. 
       \end{equation} 
       If $v=z_1 + z_2$ then $z_1=v^+$, $z_2=v^-$ and $v^{\pm} \neq 0.$ By \eqref{eq:split} and \eqref{eq:n14}, we get 
       \[\langle \I^{\prime}(v),v^+\rangle=\langle \I^{\prime}(z),z_1\rangle=0, \quad \langle \I^{\prime}(v),v^-\rangle=\langle \I^{\prime}(z),z_2\rangle=0, \]  so that $v\in\N$ and therefore
       \begin{align}\label{eq:mlebound}
       m_\lambda\le \I(v)=\I(z_1)+\I(z_2).
       \end{align}
      Taking into account Lemma~\ref{lem:L3}$(g_4)$, \eqref{eq:split}, and \eqref{eq:n14} together with $(f_4)$, we obtain
       \begin{align}
           \I(z_3)- \frac{2\alpha}{\mu} \langle \I^{\prime}(z),z_3 \rangle \nonumber &=\I(z_3)- \frac{2\alpha}{\mu} \langle \I^{\prime}(z_3),z_3 \rangle \nonumber \\ &= \frac{1}{Q} \|z_3\|^Q- \lambda \int_\Omega F(\xi,g(z_3))d\xi- \frac{2\alpha}{\mu} \|z_3\|^Q + \frac{2\alpha  \lambda }{\mu} \int_\Omega f(\xi,g(z_3)) g^{\prime}(z_3)z_3d\xi \nonumber \\ & \geq \left(\frac{1}{Q}- \frac{2\alpha}{\mu}\right) \|z_3\|^Q + \frac{\lambda}{\mu} \int_\Omega \left(f(\xi,g(z_3))g(z_3)-\mu F(\xi,g(z_3))\right)d\xi \nonumber \\ & \geq \left(\frac{1}{Q}- \frac{2\alpha}{\mu}\right) \|z_3\|^Q >0.  \label{eq:n15}
       \end{align}
      Using \eqref{eq:n14} and \eqref{eq:n15}, we deduce that $\I(z_3) >0.$  Subsequently, by \eqref{eq:split}, \eqref{eq:mlebound} and $\I(z)=m_\lambda$, we deduce that
      \begin{align*}
      m_\lambda\le \I(z_1)+\I(z_2)=\I(z)-\I(z_3)=m_\lambda-\I(z_3)<m_\lambda ,
      \end{align*}
      which is false. Thus, we must have $z_3=0$. 
   \end{proof}

\subsection{Subcritical case}
In this section, we establish the proof of Theorem~\ref{T2}. For this purpose, we need to prove the following two lemmas.

\begin{lemma}\label{lem:PS}
There exists a sequence $\{v_n\}\subset\N$ with $\I(v_n)\to m_\lambda$ and
$\I^{\prime}(v_n)\to0$ in $\W^{*}$.
\end{lemma}
\begin{proof}
By \eqref{eq:decouple} and the nondegeneracy $\det J>0$ established in Lemma~\ref{lem:L13},
$\N$ is a $C^{1}$ manifold of codimension two near each of its points; it is closed in $\W$ by
Lemma~\ref{lem:L11}$(i)$, and $\I$ is bounded below on it by Lemma~\ref{lem:L11}$(ii)$.
Ekeland's variational principle applied to $\I|_{\N}$ therefore produces a sequence
$\{v_n\}\subset\N$ with $\I(v_n)\to m_\lambda$ and
$\|(\I|_{\N})^{\prime}(v_n)\|_{T_{v_n}^{*}}\to0$. Since the Lagrange multiplier system is
diagonal by \eqref{eq:decouple}, with diagonal entries bounded away from zero by
Lemma~\ref{lem:L11}, the multipliers vanish in the limit and $\I^{\prime}(v_n)\to0$ in $\W^{*}$;
see \cite{willem2012minimax} for the details of this standard argument. 
\end{proof}

\begin{lemma}{\label{lem:L15}}
Let $\{v_n\}\subset\N$ be a minimizing sequence for $m_\lambda$ as produced by Lemma~\ref{lem:PS}. Then, there exists some $v\in \W$ such that for $n \to \infty,$ 
\begin{align*}
\int _\Omega f(\xi, g(v_n^{\pm})) g^{\prime}(v_n^{\pm})v_n^{\pm}d\xi &\to \int _\Omega f(\xi, g(v^{\pm})) g^{\prime}(v^{\pm})v^{\pm}d\xi, \\ \int_\Omega F(\xi,g(v_n^{\pm}))d\xi &\to \int_\Omega F(\xi,g(v^{\pm}))d\xi.
\end{align*} 
\end{lemma}
\begin{proof}
     Let $\{v_n\}$ be a sequence in $\N$ such that $\I(v_n) \to m_{\lambda} \text{ as } n\to \infty.$ Thanks to Lemma~\ref{lem:L11}, $v_n$ is bounded sequence in $\W.$ So, up to a subsequence, there exists $v\in\W$ such that $v_n \rightharpoonup v \text{ in } \W, v_n \to v \text{ in } L^r(\Omega) \text{ for } r\geq Q, \text{ and } v_n(\xi) \to v(\xi) \text{ a.e. } \xi\in\Omega.$ It gives $v_n^{\pm} \rightharpoonup v^{\pm} \text{ in } \W, v_n^{\pm} \to v^{\pm} \text{ in } L^r(\Omega) \text{ for } r \geq Q, \text{ and } v_n^{\pm} (\xi) \to v^{\pm}(\xi) \text{ a.e. } \xi\in\Omega.$ Choose $\beta >0, t>1$ such that 
     \begin{align*}
     t'= \frac{t}{t-1} \geq Q \text{ and } t \beta (2\alpha)^\frac{1}{Q-1} 2^{Q'} \sup_n \|v_n\|^{Q'} < \alpha_Q. 
     \end{align*} 
     Using H\"older inequality together with Trudinger-Moser inequality, we deduce
     \begin{align*}
          \left|\int_\Omega (v_n-v)f(\xi,g(v_n))g^{\prime}(v_n) d\xi\right| & \leq C\int_\Omega \left|v_n -v\right|  \exp(\beta\left|g(v_n)\right|^{\frac{2\alpha Q}{Q-1}})d\xi \\ &\leq C\int_\Omega \left|v_n -v\right| \exp\left(\beta(2\alpha)^\frac{1}{Q-1}\left|v_n\right|^{Q'}\right)d\xi \\ & \le C_1 \|v_n -v\|_{L^{t^{\prime}}(\Omega)} \left(\int_\Omega \exp\left(t\beta (2\alpha)^\frac{1}{Q-1} 2^{Q'} \|v_n^{+}\|^{Q^\prime} \left(\frac{\left|v_n^+\right|}{\|v_n^{+}\|}\right)^{Q'}\right)d\xi \right)^\frac{1}{2t} \\& \qquad  \left(\int_\Omega \exp\left(t\beta (2\alpha)^\frac{1}{Q-1} 2^{Q'} \|v_n^{-}\|^{Q^\prime} \left(\frac{\left|v_n^-\right|}{\|v_n^{-}\|}\right)^{Q'}\right)d\xi \right)^\frac{1}{2t} \\& \leq C_2 \|v_n -v\|_{L^{t^{\prime}}(\Omega)} \to 0 \text{ as } n\to\infty,
     \end{align*}
     since $\W \hookrightarrow L^r(\Omega)$ is compact for $r \geq Q.$ It implies  \[\int_\Omega (v_n-v)f(\xi,g(v_n))g^{\prime}(v_n) d\xi \to 0 \text{ as } n\to \infty.\]
Also, $\I^{\prime}(v_n) \to 0 \text{ in } \overset{\scriptscriptstyle\circ}{S}{}^Q_1(\Omega) ^{*}\text{ gives } \langle \I^{\prime}(v_n), v_n-v \rangle \to 0 \text{ as } n \to \infty. $ Since
     \begin{align*}
          &\langle I_{\lambda}^{\prime}(v_n), v_n-v \rangle \notag \\&= \int_\Omega \left|\nah v_n\right|^{Q-2} \nah v_n. \nah(v_n - v)d\xi - \lambda \int_\Omega f(\xi,g(v_n))g^{\prime}(v_n)(v_n-v)d\xi, 
     \end{align*}
     we obtain $\displaystyle\lim_{n\to\infty}\int_\Omega \left|\nah v_n\right|^{Q-2}\nah v_n.\nah(v_n-v)\,d\xi=0$, and Simon's inequality on $\R^{2N}$ together with $\int_\Omega\left|\nah v\right|^{Q-2}\nah v.\nah(v_n-v)\,d\xi\to0$ gives $\nah v_n\to\nah v$ in $L^{Q}(\Omega)$, i.e.\ $v_n\to v$ in $\W$. In particular $v_n^{\pm}\to v^{\pm}$ in $\W$.
    Finally, $f(\xi,g(v_n^{\pm}))g^{\prime}(v_n^{\pm})v_n^{\pm}\to f(\xi,g(v^{\pm}))g^{\prime}(v^{\pm})v^{\pm}$ and
     $F(\xi,g(v_n^{\pm}))\to F(\xi,g(v^{\pm}))$ almost everywhere in $\Omega$ by the continuity of
     $f$, $F$, $g$ and $g^{\prime}$, while the exponential bound established above, combined with
     $v_n\to v$ in every $L^{r}(\Omega)$, shows that both families are uniformly integrable.
     Vitali's convergence theorem then gives the two stated limits. \qed
\end{proof}
    
\begin{lemma}{\label{lem:L16}}
    There exists $z\in\N \text{ such that }$ $\displaystyle\I(z)=m_\lambda= \inf_{v\in\N} \I(v).$ 
\end{lemma}
\begin{proof} There exists a sequence $\{z_n\}$ in $\N$ such that $\I(z_n) \to m_{\lambda}$ as $n \to \infty.$ By Lemma~\ref{lem:L11}, the sequence $\{z_n\}$ is bounded in $\W$. Then, up to a subsequence, there exists $z \in \W$ such that $z_n \rightharpoonup z$ in $\W, z_n \to z$ in  $L^r(\Omega)$ for $r \geq Q,$ and $ z_n(\xi) \to z(\xi)$ a.e. $ \xi\in\Omega.$
     It gives $z_n^{\pm} \rightharpoonup z^{\pm} \text{ in }  \W,$ $z_n^{\pm} \to z^{\pm} \text{ in } L^r(\Omega) \text{ for } r \geq Q$, and $ z_n^{\pm} (\xi) \to z^{\pm}(\xi) \text{ a.e. } \xi\in\Omega.$ Now, by definition of $\N,$ 
     \begin{equation}\label{eq:n17} 
     0=\langle \I^{\prime}(z_n), z_n^+ \rangle=\|z_n^+\|^Q-\lambda \int_\Omega f(\xi,g(z_n^+))g^{\prime}(z_n^+)z_n^+ d\xi 
     \end{equation}
     Let, if possible, $z^+=0$. Then Lemmas~\ref{lem:L15} and \eqref{eq:n17} give $\|z_n^+\| \to 0$ as $n\to \infty,$ which is a contradiction to Lemma~\ref{lem:L11}. It implies $z^+ \neq 0.$ Similarly, $z^- \neq 0.$ Taking into account lower semi-continuity of norm and Lemma~\ref{lem:L15}, we have  
     \begin{align}\label{eq:n18}
         \langle \I^{\prime}(z),z^{\pm} \rangle & = \|z^{\pm}\|^Q - \lambda \int_\Omega f(\xi,g(z))g^{\prime}(z)z^{\pm} d\xi \notag \\ & \leq \liminf_{n\ra \infty} \|z_n^{\pm}\|^Q - \lambda \lim_{n\ra \infty} \int_\Omega f(\xi,g(z_n))g^{\prime}(z_n)z_n^{\pm} d\xi \notag \\ & = \liminf_{n\ra \infty} \langle \I^{\prime}(z_n), z_n^{\pm} \rangle =0.  
     \end{align}
          By Lemma~\ref{lem:L10}, there exists $(r_z,s_z) \in (0,1] \times (0,1] $ such that $r_z z^+ + s_z z^- \in \N.$
     By \eqref{eq:split} and \eqref{eq:n18}, we have
     $\langle \I^{\prime}(z^{\pm}),z^{\pm}\rangle=\langle \I^{\prime}(z),z^{\pm}\rangle\le0$,
     while $1-r_z^{Q}\ge0$ and $1-s_z^{Q}\ge0$. Applying Lemma~\ref{lem:fiber} with $w=z^{+}$,
     $s=r_z$ and with $w=z^{-}$, $s=s_z$ therefore gives
     \begin{align*}
     \I(r_zz^{+})\le \I(z^{+})\quad\text{and}\quad \I(s_zz^{-})\le \I(z^{-}),
     \end{align*}
     with strict inequality in the first if $r_z<1$ and in the second if $s_z<1$. Consequently,using \eqref{eq:split} twice, the weak lower semicontinuity of the norm and Lemma~\ref{lem:L15},
     \begin{align*}
m_\lambda\;&\le\;\I(r_zz^{+}+s_zz^{-})\;=\;\I(r_zz^{+})+\I(s_zz^{-})\\
     &\le\;\I(z^{+})+\I(z^{-})\;=\;\I(z)\\
     &\le\;\liminf_{n\to\infty}\I(z_n)\;=\;m_\lambda .
     \end{align*}
     Hence, equality holds throughout. Were $r_z<1$ or $s_z<1$, the second inequality would be strict, a contradiction; therefore $r_z=s_z=1$, so that $z=z^{+}+z^{-}\in\N$ and
     $\I(z)=m_\lambda$.
\end{proof}

\textbf{Proof of Theorem~\ref{T2}:} By Lemma~\ref{lem:L16} there exists $z\in\N$ with
$\I(z)=m_\lambda$; by Lemma~\ref{lem:L13}, $\I^{\prime}(z)=0$, so that $z$ is a nodal solution
of \eqref{eq:P2}; and by Lemma~\ref{lem:L14}, $z$ has exactly two nodal domains. \qed

\subsection{Critical case}
In this section, we establish the proof of Theorem~\ref{T3}. To achieve this, first we need to prove the following lemma.

\begin{lemma}\label{lem:L17}
    For some $\lambda_0>0 \text{ and } z\in\N$, we have,  
    \begin{align*}
    \I(z)=m_\lambda= \inf_{v\in\N} \I(v), \quad  \text{ for all } \lambda\geq \lambda_0>0. 
    \end{align*} 
  
\end{lemma}
\begin{proof} Suppose $\{v_n\}\subset\N$ be a minimizing sequence from  Lemma~\ref{lem:PS}. By Lemma~\ref{lem:L11}, $\{v_n\}$ is bounded sequence in $\W$. Then, up to a subsequence, there exists $v\in\W$ such that $v_n \rightharpoonup v$ in $\W$, $v_n \to v$ in $L^r(\Omega)$ for $r \geq Q$, and $v_n(\xi) \to v(\xi)$ a.e. $\xi\in\Omega$, it implies that $v_n^{\pm} \rightharpoonup v^{\pm}$ in $\W$, $v_n^{\pm} \to v^{\pm}$ in  $L^r(\Omega)$ for $r \geq Q$,  and  $v_n^{\pm}(\xi) \to v^{\pm}(\xi)$ a.e. $\xi\in\Omega$. 
     
     Noticing that, according to Lemma~\ref{lem:L12}, there exists $\lambda_0>0$ such that for all $\lambda > \lambda_0,$ we get 
     \begin{align*}
     m_{\lambda}< \frac{1}{4\alpha Q}\left(\frac{\alpha_Q}{\beta_0 2^{\frac{Q}{Q-1}}}\right)^{Q-1}.
     \end{align*}
     If  $v_n \to v \text{ in } \W,$ the Lemma~\ref{lem:L17} obviously holds, since then $v\in\N$ by continuity and $\I(v)=m_\lambda$. We therefore assume throughout that $v_n\not\to v$ in $\W$, so that $\|v\|<\sigma:=\liminf_n\|v_n\|$. We demonstrate the arguments in the following steps:
     
     First, we claim that $v^{\pm}\neq0.$ Suppose that $v^{+}=0$ and $v^{-}=0$. Taking into account \eqref{eq:s5}, definition of $\I(v_n)$ and sufficiently large $n$, we deduce that 
     \begin{align*}
     \int_\Omega F(\xi,g(v_n^{\pm}))d\xi \to 0 \text{ and } \frac{1}{Q}\|v_n^+\|^{Q}+\frac{1}{Q}\|v_n^-\|^{Q}< \frac{1}{4 \alpha Q}\left(\frac{\alpha_Q}{\beta_0 2^{\frac{Q}{Q-1}}}\right)^{Q-1},
     \end{align*} 
     it implies $\|v_n^+\|^{Q}, \|v_n^-\|^{Q} < \frac{1}{4\alpha} \left(\frac{\alpha_Q}{\beta_0 2^{\frac{Q}{Q-1}}}\right)^{Q-1}$ and  $\|v_n\|^Q< \frac{1}{2\alpha} \left(\frac{\alpha_Q}{\beta_0 2^{\frac{Q}{Q-1}}}\right)^{Q-1}$. Consequently, we have \(\beta_02^{Q'}(2\alpha)^\frac{1}{Q-1} \|v_n\|^{Q'}<\alpha_Q, \text{ where } Q'=\frac{Q}{Q-1}.\)
      Thus, for some $k\in\mathbb N$, 
     \begin{align*}
     \beta_02^{Q'}(2\alpha)^\frac{1}{Q-1} \|v_n\|^{Q'}<\alpha_Q, \text{ for all } n>k.
     \end{align*} 
     Choose $\beta>\beta_0 \text{ and } b>1 \text{ close enough to } \beta_0 \text{ and }1, \text{  respectively, so that }$ for all $n>k,$ 
     \begin{equation}{\label{eq:n20}} b \beta 2^{Q'}(2\alpha)^\frac{1}{Q-1} \|v_n\|^{Q'}<\alpha_Q \text{ and } \frac{b}{b-1}\geq Q.\end{equation} 
     Taking into account H\"older inequality, Trudinger-Moser inequality, and Lemma~\ref{lem:L3}$(g_2),(g_4)$ with \eqref{eq:n20}, we deduce 
\begin{align*}
\left|\int_\Omega f(\xi,g(v_n))g^{\prime}(v_n)v_n d\xi\right|
 & \leq C\int_\Omega \left|v_n\right| \exp(\beta\left|g(v_n)\right|^{\frac{2\alpha Q}{Q-1}})d\xi 
 \\ & \leq C\int_\Omega \left|v_n\right|\exp(\beta(2\alpha)^\frac{1}{Q-1}\left|v_n\right|^{Q'})d\xi 
 \\ & \le C\|v_n\|_{\frac{b}{b-1}} \left(\int_\Omega \exp\left(b\beta (2\alpha)^\frac{1}{Q-1} 2^{Q'} \|v^+_n\|^{Q'} \left(\frac{\left|v_n^+\right|}{\|v^+_n\|}\right)^{Q'}\right)d\xi \right)^\frac{1}{2b} \\
 & \qquad \qquad \left(\int_\Omega \exp\left(b\beta (2\alpha)^\frac{1}{Q-1} 2^{Q'} \|v^-_n\|^{Q'} \left(\frac{\left|v_n^-\right|}{\|v^-_n\|}\right)^{Q'}\right)d\xi \right)^\frac{1}{2b} 
 \\ & \leq C'\|v_n\|_{\frac{b}{b-1}} \to 0 \text{ as } n\to\infty.
 \end{align*}
Here, we have also used the inequality $(l+m)^k \leq 2^{k-1}(l^k +m^k),$ for all $l,m \geq 0$ and $k>0.$ We note that the factor $2^{Q^{\prime}}$ occurs in this proof for two unrelated reasons: in \eqref{eq:n25} below it comes from the admissible range in Lemma~\ref{lem:L2}, while in the display above it comes from the elementary inequality just used, applied to $\left|v_n\right|^{Q^{\prime}}$. The threshold \eqref{eq:n21} must accommodate both, and it is the second which is binding. 
This yields $\|v_n\|^{Q}\to0$ as $n\to\infty$, since $\{v_n\}\subset\N$.
 As a result, we get, $\|v_n^+\| \to 0$ and $\|v_n^-\|\to 0, $ which is a contradiction to the Lemma~\ref{lem:L11}.

Suppose next that exactly one of the two parts vanishes; without loss of
generality $v^{+}=0$ and $v^{-}\not\equiv0$. Since $v_n^{+}\to0$ in every $L^{r}(\Omega)$,
$r<\infty$, the bound \eqref{eq:s5} together with the H\"older and Trudinger-Moser
inequalities gives $\int_\Omega F(\xi,g(v_n^{+}))\,d\xi\to0$, exactly as in the previous case.
By \eqref{eq:split} and the weak lower semicontinuity of the norm,
\begin{align*}
\frac1Q\|v_n^{+}\|^{Q}
&=\I(v_n)-\I(v_n^{-})+\lambda\int_\Omega F(\xi,g(v_n^{+}))\,d\xi\\
&\le m_\lambda-\left(\frac1Q\|v^{-}\|^{Q}-\lambda\int_\Omega F(\xi,g(v^{-}))\,d\xi\right)+o(1)
= m_\lambda-\I(v^{-})+o(1),
\end{align*}
Since $v_n\in\N$, the identity \eqref{eq:split} gives
$\langle \I^{\prime}(v_n^{-}),v_n^{-}\rangle=\langle \I^{\prime}(v_n),v_n^{-}\rangle=0$, so the
computation of Lemma~\ref{lem:L11}$(ii)$ applies to $v_n^{-}$ and yields
\begin{align*}
\I(v_n^{-})=\I(v_n^{-})-\frac{2\alpha}{\mu}\langle \I^{\prime}(v_n^{-}),v_n^{-}\rangle
\ \ge\ \left(\frac1Q-\frac{2\alpha}{\mu}\right)\|v_n^{-}\|^{Q}\ \ge\ 0 .
\end{align*}
Consequently
\begin{align*}
\frac1Q\|v_n^{+}\|^{Q}\;\le\;m_\lambda+o(1)\;<\;\frac{1}{4\alpha Q}
\left(\frac{\alpha_Q}{\beta_0 2^{Q^{\prime}}}\right)^{Q-1}
\end{align*}
for $n$ large, which is the bound used above for $\|v_n^{+}\|$ in the case $v^{\pm}=0$. Since
$\|v_n^{-}\|^{Q}\le\|v_n\|^{Q}$ is bounded by Lemma~\ref{lem:L11}$(ii)$ and
$m_\lambda$ satisfies \eqref{eq:n21}, the same choice of $\beta$ and $b$ is admissible, and
repeating the estimate above with $v_n$ replaced by $v_n^{+}$ yields $\|v_n^{+}\|\to0$,
contradicting Lemma~\ref{lem:L11}$(i)$. The case $v^{-}=0$, $v^{+}\not\equiv0$ is symmetric.
Therefore $v^{\pm}\not\equiv0$.

\textbf{Assertion.} If there exists $\lambda>0$ such that 
 \begin{equation}{\label{eq:n21}}
   m_\lambda< \frac{1}{4\alpha Q}\left(\frac{\alpha_Q}{\beta_02^{Q'}}\right)^{Q-1}.
 \end{equation}
Then, for $\{v_n\}\subset\N,$ $\int_\Omega \exp(b\beta (2\alpha)^\frac{1}{Q-1}\left|v_n\right|^{Q'}) d\xi \le C_1.$

\textbf{Claim.} $\displaystyle \liminf_{n\to\infty}\|v_n\|^{Q}-\|v\|^{Q} < \frac{1}{2\alpha}\left(\frac{\alpha_Q}{\beta_02^{Q'}}\right)^{Q-1}.$

Let, if possible, $\displaystyle\liminf_{n\to\infty}\|v_n\|^{Q}-\|v\|^{Q} \geq \frac{1}{2\alpha}\left(\frac{\alpha_Q}{\beta_02^{Q'}}\right)^{Q-1}.$ 
Passing to a further subsequence, still denoted $\{v_n\}$, we may assume that
$\|v_n^{+}\|$, $\|v_n^{-}\|$ and $\|v_n\|$ all converge, and that the limit of $\|v_n\|^{Q}$
realises $\liminf_n\|v_n\|^{Q}$. Set $\varrho^{\pm}:=\lim_n\|v_n^{\pm}\|$, so that
$\liminf_n\|v_n\|^{Q}=(\varrho^{+})^{Q}+(\varrho^{-})^{Q}$ by \eqref{eq:split}.
Then, $((\varrho^+)^Q-\|v^+\|^Q)+((\varrho^-)^Q-\|v^-\|^Q) \geq \frac{1}{2\alpha}\left(\frac{\alpha_Q}{\beta_02^{Q'}}\right)^{Q-1},$ it implies that 
\begin{align*}
((\varrho^+)^Q-\|v^+\|^Q) \geq \frac{1}{4\alpha}\left(\frac{\alpha_Q}{\beta_02^{Q'}}\right)^{Q-1} \text{ or } ((\varrho^-)^Q-\|v^-\|^Q) \geq \frac{1}{4\alpha}\left(\frac{\alpha_Q}{\beta_02^{Q'}}\right)^{Q-1}.
\end{align*} 
Set $\vartheta:=\dfrac{\mu}{2\alpha Q}$, so that $\vartheta>1$ by $(f_4)$ and
$\vartheta Q=\dfrac{\mu}{2\alpha}$. By Lemma~\ref{lem:L3}$(g_4)$ in the form
$g^{\prime}(t)t\ge\frac{1}{2\alpha}g(t)$, followed by $(f_4)$,
\begin{align}\label{eq:thetasign}
f(\xi,g(t))g^{\prime}(t)t-\vartheta Q\,F(\xi,g(t))
\;\ge\;\frac{1}{2\alpha}f(\xi,g(t))g(t)-\frac{\mu}{2\alpha}F(\xi,g(t))\;\ge\;0
, ~  \forall t\in\R .
\end{align}
Using the fact that $\langle \I^{\prime}(v_n),v_n\rangle=0$ for $v_n\in\N$,
together with \eqref{eq:thetasign}, we have 
\begin{small}
\begin{align}
    m_{\lambda} &= \lim_{n\to\infty}\left[\I(v_n)- \frac{1}{\vartheta Q} \langle \I^{\prime}(v_n),v_n \rangle\right] \notag  \\ &= \frac{1}{\vartheta Q} \lim_{n\to\infty}( \vartheta  \|v_n^+\|^Q- \|v_n^+\|^Q) + \frac{1}{\vartheta Q} \lim_{n\to\infty}( \vartheta  \|v_n^-\|^Q- \|v_n^-\|^Q) \notag \\ &\quad + \frac{\lambda}{\vartheta Q} \int_\Omega \{f(\xi,g(v))g^{\prime}(v)v-\vartheta Q F(\xi,g(v)) \}d\xi \notag \\ & \quad + \frac{\lambda}{\vartheta Q} \lim_{n\to\infty} \int_\Omega \{f(\xi,g(v_n))g^{\prime}(v_n)v_n - f(\xi,g(v))g^{\prime}(v)v \} d\xi \notag \\ & \geq   \frac{1}{\vartheta Q} \lim_{n\to\infty}( \vartheta  \|v_n^+\|^Q- \|v_n^+\|^Q) + \frac{\lambda}{\vartheta Q} \lim_{n\to\infty} \int_\Omega \{f(\xi,g(v_n))g^{\prime}(v_n)v_n - f(\xi,g(v))g^{\prime}(v)v \} d\xi \notag \\ & \geq \frac{1}{\vartheta Q} \left[\frac{\vartheta}{4\alpha}\left(\frac{\alpha_Q}{\beta_02^{Q'}}\right)^{Q-1}- \frac{1}{4\alpha}\left(\frac{\alpha_Q}{\beta_02^{Q'}}\right)^{Q-1} \right]  + \frac{\lambda}{\vartheta Q} \lim_{n\to\infty} \int_\Omega \{f(\xi,g(v_n))g^{\prime}(v_n)v_n - f(\xi,g(v))g^{\prime}(v)v \} d\xi \notag \\ & = \frac{1}{4 \alpha Q} \left(\frac{\alpha_Q}{\beta_02^{Q'}}\right)^{Q-1}- \frac{1}{\vartheta Q} \left[\frac{1}{4\alpha}\left(\frac{\alpha_Q}{\beta_02^{Q'}}\right)^{Q-1}- \lambda \lim_{n\to\infty} \int_\Omega \{f(\xi,g(v_n))g^{\prime}(v_n)v_n - f(\xi,g(v))g^{\prime}(v)v \} d\xi \right] \label{eq:n22} 
\end{align}
\end{small}
In the sixth line, we used
$(\varrho^{+})^{Q}\ge(\varrho^{+})^{Q}-\|v^{+}\|^{Q}
\ge\frac{1}{4\alpha}\left(\frac{\alpha_Q}{\beta_02^{Q^{\prime}}}\right)^{Q-1}$, and in the second line the limits $\lim_n\|v_n^{\pm}\|^{Q}$ are taken along the subsequence fixed above, along which both converge.

Note first that the sequence $\{v_n\}$ is a bounded Palais-Smale sequence for
$\I$ by Lemma~\ref{lem:L11} and Lemma~\ref{lem:PS}, so Lemma~\ref{lem:L6prime} applies with
$\psi=v$ and yields
\begin{align}\label{eq:swap}
\lim_{n\to\infty}\int_\Omega f(\xi,g(v_n))g^{\prime}(v_n)v\,d\xi
=\int_\Omega f(\xi,g(v))g^{\prime}(v)v\,d\xi .
\end{align}
 \textcolor{red}{We emphasise that \eqref{eq:swap} does not depend on the Assertion, whose proof below uses the
Claim; invoking the Assertion here would render the argument circular.}
Taking into account \eqref{eq:swap}, Lemma~\ref{lem:L5}, and Lemma~\ref{lem:PS}, which gives $\langle \I^{\prime}(v_n),v_n-v \rangle \to 0 \text{ as } n\to\infty$, we deduce
\begin{align}
     0&=\lim_{n \to \infty} \left[ \int_\Omega \left|\nah v_n \right|^{Q-2} \nah v_n.\nah (v_n-v) d\xi- \lambda \int_\Omega f(\xi,g(v_n))g^{\prime}(v_n)(v_n -v) d\xi\right] \notag \\
     &=\lim_{n\to\infty}\left[\|v_n\|^{Q}-\int_\Omega\left|\nah v_n\right|^{Q-2}\nah v_n.\nah v\,d\xi\right]\notag\\
     &\qquad-\lambda\lim_{n\to\infty}\int_\Omega f(\xi,g(v_n))g^{\prime}(v_n)v_n\,d\xi
     +\lambda\lim_{n\to\infty}\int_\Omega f(\xi,g(v_n))g^{\prime}(v_n)v\,d\xi\notag\\
     &=\liminf_{n\ra \infty} \|v_n\|^Q - \|v\|^Q- \lambda  \lim_{n\to\infty} \int_\Omega \{f(\xi,g(v_n))g^{\prime}(v_n)v_n - f(\xi,g(v))g^{\prime}(v)v \} d\xi \notag \\
     &=((\varrho^+)^Q-\|v^+\|^Q) + ((\varrho^-)^Q-\|v^-\|^{Q})- \lambda  \lim_{n\to\infty} \int_\Omega \{f(\xi,g(v_n))g^{\prime}(v_n)v_n - f(\xi,g(v))g^{\prime}(v)v \}d\xi \notag \\
     & \geq ((\varrho^+)^Q-\|v^+\|^Q) - \lambda  \lim_{n\to\infty} \int_\Omega \{f(\xi,g(v_n))g^{\prime}(v_n)v_n - f(\xi,g(v))g^{\prime}(v)v \}d\xi \notag \\
     & \geq \frac{1}{4\alpha}\left(\frac{\alpha_Q}{\beta_02^{Q'}}\right)^{Q-1}- \lambda  \lim_{n\to\infty} \int_\Omega \{f(\xi,g(v_n))g^{\prime}(v_n)v_n - f(\xi,g(v))g^{\prime}(v)v \}d\xi \label{eq:n23}   
\end{align}
Then, by \eqref{eq:n22} and \eqref{eq:n23}, we deduce 
\begin{align*}
m_{\lambda} \geq \frac{1}{4 \alpha Q} \left(\frac{\alpha_Q}{\beta_02^{Q'}}\right)^{Q-1},
\end{align*} 
which is a contradiction to \eqref{eq:n21}. This proves the claim.
Now, we validate the conclusion of the assertion. Using the lower semi-continuity of the norm, we have 
\begin{align*}
\varrho^+ \geq \|v^+\| \text{ and } \varrho^- \geq \|v^-\|.
\end{align*} 
Define $\tilde{v}_n=\frac{v_n}{\|v_n\|} \text{ and } \tilde{v}=\frac{v}{\sigma} , \text{ where } \sigma=\liminf\|v_n\|.$ Then, 
\begin{align*}
\|\tilde{v}_n\|=1, \quad \tilde{v}_n \rightharpoonup \tilde{v} \text{ in }\W \text{ and } \|\tilde{v}\|<1.
\end{align*} 
Moreover, by Lemma~\ref{lem:L5}, $\nah\tilde v_n\to\nah\tilde v$ a.e.\ in $\Omega$ .
By Lemma~\ref{lem:L2} with $0<\gamma< \frac{\alpha_Q}{2^{Q'}(1-\|\tilde{v}\|^Q)^\frac{1}{Q-1}}$, we obtain 
\begin{equation}{\label{eq:n24}} 
\sup_n \int_\Omega \exp\left(\ga \left|\tilde{v}_n(\xi)\right|^\frac{Q}{Q-1}\right)d\xi \leq C_1, \text{ for some } C_1>0.
\end{equation} 

Since $\tilde v=v/\sigma$, we have $\|v\|^{Q}=\sigma^{Q}\|\tilde v\|^{Q}$, hence
\begin{align}\label{eq:tauto}
\sigma^{Q}=\frac{\sigma^{Q}-\|v\|^{Q}}{1-\|\tilde v\|^{Q}} ,
\end{align}
which expresses $\sigma^{Q}$ in terms of the quantity
$\sigma^{Q}-\|v\|^{Q}=\liminf_n\|v_n\|^{Q}-\|v\|^{Q}$ estimated in the Claim. Substituting
the bound of the Claim into the numerator of \eqref{eq:tauto} and raising to the power
$\frac{1}{Q-1}$ yields 
\begin{align*}
\liminf_{n\to\infty} \|v_n\|^{Q'} < \frac{\alpha_Q}{(2\alpha )^{\frac{1}{Q-1}} \beta_0 2^{Q'}(1-\|\tilde{v}\|^Q)^{\frac{1}{Q-1}}},
\end{align*} 
then, there exists a natural number $m$ and $\ga >0$ such that 
\begin{align*}
(2\alpha)^{\frac{1}{Q-1}} \beta_0 \|v_n\|^{Q'}<\ga< \frac{\alpha_Q}{2^{Q'} (1-\|\tilde{v}\|^Q)^{\frac{1}{Q-1}}}, \text{ for all } n>m.
\end{align*} 
Choose $b>1$ and $\beta>\beta_0$ ($b$ is close to $1$ and $\beta$ is close to $\beta_0$) so that 
\begin{equation}{\label{eq:n25}} b\beta (2\alpha)^{\frac{1}{Q-1}} \|v_n\|^{Q'} \leq \ga < \frac{\alpha_Q}{2^{Q'} (1-\|\tilde{v}\|^Q)^{\frac{1}{Q-1}}} \text{ and } \frac{b}{b-1} \geq Q.\end{equation} 
By \eqref{eq:n24} and \eqref{eq:n25}, we deduce that
\begin{align*}
\int_\Omega \exp(b\beta (2\alpha)^\frac{1}{Q-1}\left|v_n\right|^{Q'}) d\xi=& \int_\Omega \exp(b\beta (2\alpha)^\frac{1}{Q-1} \|v_n\|^{Q'}\left|\tilde{v}_n\right|^{Q'}) d\xi \leq \int_\Omega \exp(\ga\left|\tilde{v}_n\right|^{Q'}) d\xi \leq C_1. 
\end{align*}
Now, we will prove Lemma~\ref{lem:L17}. Suppose $\{z_n\}\subset\N$ be a minimizing sequence, that is, $\I(z_n)\to m_\lambda \text{ and } \I^{\prime}(z_n)\to 0 \text{ as } n\to\infty$ (Lemma~\ref{lem:PS}). By Lemma~\ref{lem:L11}, $\{z_n\}$ is bounded sequence in $\W$. Then, up to a subsequence, there exists $z\in\W$ such that $z_n \rightharpoonup z \text{ in } \W,$ $z_n \to z \text{ in } L^r(\Omega)$ for $r\geq Q$, and $z_n(\xi) \to z(\xi) \text{ a.e. } \xi\in\Omega.$ It gives  $z_n^{\pm} \rightharpoonup z^{\pm} \text{ in } \W$, $z_n^{\pm} \to z^{\pm} \text{ in } L^r(\Omega) \text{ for } r \geq Q$, $z_n^{\pm}(\xi) \to z^{\pm}(\xi) \text{ a.e. } \xi\in\Omega.$
Using the lower semi-continuity of norm and Lemma~\ref{lem:L6}, we get
     \begin{align*}
     \I(z)=& \frac{1}{Q} \|z\|^Q- \lambda\int_\Omega F(\xi,g(z))d\xi \notag \\ \leq & \frac{1}{Q} \liminf_{n\ra \infty} \|z_n\|^Q-\lambda \lim_{n\ra \infty} \int_\Omega F(\xi,g(z_n))d\xi \notag \\ \leq &\liminf_{n\ra \infty}\left(\frac{1}{Q}\|z_n\|^Q-\lambda \int_\Omega F(\xi,g(z_n))d\xi\right)=\liminf_{n\ra \infty}\I(z_n)=m_\lambda.
     \end{align*}
    Taking into account the assertion, critical growth condition, H\"older inequality, Sobolev embedding, and properties of $g$, 
     \begin{align}
       \left|\int_\Omega  (z_n-z)g^{\prime}(z_n)f(\xi,g(z_n))d\xi\right| \notag & \leq C \int_\Omega \left|z_n-z\right| \exp\left(\beta \left|g(z_n)\right|^{\frac{2\alpha Q}{Q-1}}\right)d\xi \notag \\& \leq C_2\|z_n-z\|_{\frac{b}{b-1}} \left(\int_\Omega \exp(b\beta (2\alpha)^{\frac{1}{Q-1}} \left|z_n\right|^{Q'})d\xi\right)^{\frac{1}{b}} \notag\\& \leq C_3 \|z_n-z\|_{\frac{b}{b-1}} \to 0 \text{ as } n\to\infty. {\label{eq:n26}}
     \end{align} 
   
     By definition,
     \begin{align}
      \langle \I^{\prime}(z_n),z_n-z\rangle= &\int_\Omega \left|\nabla_{\mathbb{H}} z_n\right|^{Q-2} \nabla_{\mathbb{H}} z_n .\nabla_{\mathbb{H}} (z_n-z)d\xi - \lambda\int_\Omega f(\xi,g(z_n))g^{\prime}(z_n)(z_n-z)d\xi. {\label{eq:n27}}
     \end{align}
     Then, by \eqref{eq:n26} and \eqref{eq:n27}, it follows that 
     \begin{align*}
     \int_\Omega\left|\nabla_{\mathbb{H}} z_n\right|^{Q-2} \nabla_{\mathbb{H}} z_n .\nabla_{\mathbb{H}} (z_n-z)d\xi \to 0   
     \end{align*}
     as $ n\to\infty.$
     Thanks to Lemma~\ref{lem:L5} and Simon's inequality, we have $z_n\to z$ in $\W$, hence $z_n^{\pm}\to z^{\pm}$ in $\W$ and
     \begin{align*}
     \liminf_{n\ra \infty}\|z_n^+\|^{Q}= \|z^+\|^{Q} \text{ and } \liminf_{n\ra \infty}\|z_n^-\|^{Q}= \|z^-\|^{Q},  
     \end{align*}
     it implies that $\|z_n\| \to \|z\|$ as $n \to \infty.$ Thus, $\I(z)=m_\lambda.$ 
     Finally, $z^{\pm}\not\equiv0$ by Lemma~\ref{lem:L11}$(i)$ and the strong convergence, and $\langle \I^{\prime}(z),z^{\pm}\rangle=\lim_n\langle \I^{\prime}(z_n),z_n^{\pm}\rangle=0$; hence $z\in\N$. 
\end{proof}

\textbf{Proof of Theorem~\ref{T3}:} By Lemma~\ref{lem:L17} there exist $\lambda_0>0$ and
$z\in\N$ with $\I(z)=m_\lambda$ for every $\lambda\ge\lambda_0$; by Lemma~\ref{lem:L13},
$\I^{\prime}(z)=0$, so that $z$ is a nodal solution of \eqref{eq:P2}; and by
Lemma~\ref{lem:L14}, $z$ has exactly two nodal domains. \qed
\section*{Declarations}

\noindent \textbf{Funding} \\
The authors declare that no funds, grants, or other support were received during the preparation of this manuscript.

\noindent \textbf{Competing Interests} \\
The authors have no relevant financial or non-financial interests to disclose.

\noindent \textbf{Ethics Approval and Informed Consent} \\
Not applicable.

\bibliographystyle{siam}
\bibliography{ref}

\end{document}